\documentclass[10pt, a4]{amsart}

\usepackage{lmodern}
\usepackage[utf8]{inputenc}
\usepackage{amsmath}
\usepackage{graphicx}
\usepackage{amssymb}
\usepackage{esint}
\usepackage[dvipsnames]{xcolor}
\usepackage{tikz}
\usepackage{xxcolor}
\usepackage{floatrow}
\usepackage{color}
\usepackage{amsthm}
\usepackage{epsfig}
\usepackage[english]{babel}
\usepackage{hyperref}
\usepackage[normalem]{ulem}
\usepackage{mathrsfs}
\usepackage{tikz}
\usetikzlibrary{decorations.markings,backgrounds}
\usetikzlibrary{arrows.meta}
\usepackage{pgfplots}
\usepackage{subfigure}
\usepackage{caption}
\usepackage{bbm}

\usepackage{stmaryrd}

\usepackage{float}
\usepackage{pst-plot}

\hypersetup{
	colorlinks,
	linkcolor={red!80!black},
	citecolor={blue!50!black},
	urlcolor={blue!80!black}
}
 \usepackage{diagbox}

\usepackage[a4paper,top=3.5 cm,bottom=3 cm,left=2.5 cm,right=2.5 cm]{geometry}

\theoremstyle{plain}
\newtheorem{lemma}{Lemma}[section]
\newtheorem{proposition}[lemma]{Proposition}
\newtheorem{theorem}{Theorem}[section]

\theoremstyle{definition}

\theoremstyle{remark}
\newtheorem{remark}{Remark}[section]

\numberwithin{equation}{section}

\def\to{\rightarrow}

\usepackage{dsfont}

\newcommand{\nchi}{{\raise.3ex\hbox{$\chi$}}}

\def\XXint#1#2#3{{\setbox0=\hbox{$#1{#2#3}{\int}$ }
		\vcen{\hbox{$#2#3$ }}\kern-.6\wd0}}

\newcommand{\cA}{\mathcal{A}}

\newcommand{\cI}{\mathcal{I}}

\usepackage{mathtools}

\makeatletter
\newcommand{\justified}{%
	\rightskip\z@skip%
	\leftskip\z@skip}
\makeatother
\usepackage{amsfonts}
\usepackage{aurical}

\newcommand\restr[2]{{
		\left.\kern-\nulldelimiterspace 
		#1 
		\vphantom{\big|} 
		\right|_{#2} 
}}

\newcommand{\var}{\varepsilon}

\DeclareFontFamily{U}{mathx}{\hyphenchar\font45}
\DeclareFontShape{U}{mathx}{m}{n}{<-> mathx10}{}
\DeclareSymbolFont{mathx}{U}{mathx}{m}{n}
\DeclareMathAccent{\widebar}{0}{mathx}{"73}

\renewcommand{\i}{\ifmmode\mathit{\mathchar"7010 }\else\char"10 \fi}
\renewcommand{\j}{\ifmmode\mathit{\mathchar"7011 }\else\char"11 \fi}

\def\char{{1\!\mbox{\rm l}}}

\usepackage{dsfont}

\definecolor{orange}{rgb}{1,.549,0}
\definecolor{GreenYellow }{rgb}{ 0.15,   0.69, 0}
\definecolor{Yellowone}{rgb}{ 0, 1., 0} \definecolor{Goldenrod }{rgb}{  0, 0.10, 0.84}
\definecolor{Dandelion }{rgb}{ 0, 0.29, 0.84} 
\definecolor{Apricot }{rgb}{ 0, 0.32, 0.52}
\definecolor{Peach }{rgb}{ 0, 0.50, 0.70} 
\definecolor{GreenYellow}{cmyk}{0.15,0,0.69,0}
\definecolor{RoyalPurple}{cmyk}{0.75,0.90,0,0}
\definecolor{Yellow}{cmyk}{0,0,1,0}
\definecolor{BlueViolet}{cmyk}{0.86,0.91,0,0.04}
\definecolor{Goldenrod}{cmyk}{0,0.10,0.84,0}
\definecolor{Periwinkle}{cmyk}{0.57,0.55,0,0}
\definecolor{Dandelion}{cmyk}{0,0.29,0.84,0}
\definecolor{CadetBlue}{cmyk}{0.62,0.57,0.23,0}
\definecolor{Apricot}{cmyk}{0,0.32,0.52,0}
\definecolor{CornflowerBlue}{cmyk}{0.65,0.13,0,0}
\definecolor{Peach}{cmyk}{0,0.50,0.70,0}
\definecolor{MidnightBlue}{cmyk}{0.98,0.13,0,0.43}
\definecolor{Melon}{cmyk}{0,0.46,0.5,0}
\definecolor{NavyBlue}{cmyk}{0.94,0.54,0,0}
\definecolor{YellowOrange}{cmyk}{0,0.42,1,0}
\definecolor{RoyalBlue}{cmyk}{1,0.50,0,0}
\definecolor{Orange}{cmyk}{0,0.61,0.87,0}
\definecolor{Blue}{cmyk}{1,1,0,0}
\definecolor{BurntOrange}{cmyk}{0,0.51,1,0}
\definecolor{Cerulean}{cmyk}{0.94,0.11,0,0}
\definecolor{Bittersweet}{cmyk}{0,0.75,1,0.24}
\definecolor{Cyan}{cmyk}{1,0,0,0}
\definecolor{RedOrange}{cmyk}{0,0.77,0.87,0}
\definecolor{ProcessBlue}{cmyk}{0.96,0,0,0}
\definecolor{Mahogany}{cmyk}{0,0.85,0.87,0.35}
\definecolor{SkyBlue}{cmyk}{0.62,0,0.12,0}
\definecolor{Maroon}{cmyk}{0,0.87,0.68,0.32}
\definecolor{Turquoise}{cmyk}{0.85,0,0.20,0}
\definecolor{BrickRed}{cmyk}{0,0.89,0.94,0.28}
\definecolor{TealBlue}{cmyk}{0.86,0,0.34,0.02}
\definecolor{Red}{cmyk}{0,1,1,0}
\definecolor{Aquamarine}{cmyk}{0.82,0,0.30,0}
\definecolor{OrangeRed}{cmyk}{0,1,0.50,0}
\definecolor{BlueGreen}{cmyk}{0.85,0,0.33,0}
\definecolor{RubineRed}{cmyk}{0,1,0.13,0}
\definecolor{Emerald}{cmyk}{1,0,0.50,0}
\definecolor{WildStrawberry}{cmyk}{0,0.96,0.39,0}
\definecolor{JungleGreen}{cmyk}{0.99,0,0.52,0}
\definecolor{Salmon}{cmyk}{0,0.53,0.38,0}
\definecolor{SeaGreen}{cmyk}{0.69,0,0.50,0}
\definecolor{CarnationPink}{cmyk}{0,0.63,0,0}
\definecolor{Green}{cmyk}{1,0,1,0}
\definecolor{Magenta}{cmyk}{0,1,0,0}
\definecolor{ForestGreen}{cmyk}{0.91,0,0.88,0.12}
\definecolor{VioletRed}{cmyk}{0,0.81,0,0}
\definecolor{PineGreen}{cmyk}{0.92,0,0.59,0.25}
\definecolor{Rhodamine}{cmyk}{0,0.82,0,0}
\definecolor{LimeGreen}{cmyk}{0.50,0,1,0}
\definecolor{Mulberry}{cmyk}{0.34,0.90,0,0.02}
\definecolor{YellowGreen}{cmyk}{0.44,0,0.74,0}
\definecolor{RedViolet}{cmyk}{0.07,0.90,0,0.34}
\definecolor{SpringGreen}{cmyk}{0.26,0,0.76,0}
\definecolor{Fuchsia}{cmyk}{0.47,0.91,0,0.08}
\definecolor{OliveGreen}{cmyk}{0.64,0,0.95,0.40}
\definecolor{Lavender}{cmyk}{0,0.48,0,0}
\definecolor{RawSienna}{cmyk}{0,0.72,1,0.45}
\definecolor{Thistle}{cmyk}{0.12,0.59,0,0}
\definecolor{Sepia}{cmyk}{0,0.83,1,0.70}
\definecolor{Orchid}{cmyk}{0.32,0.64,0,0}
\definecolor{Brown}{cmyk}{0,0.81,1,0.60}
\definecolor{DarkOrchid}{cmyk}{0.40,0.80,0.20,0}
\definecolor{Tan}{cmyk}{0.14,0.42,0.56,0}
\definecolor{Purple}{cmyk}{0.45,0.86,0,0}
\definecolor{Gray}{cmyk}{0,0,0,0.50}
\definecolor{Plum}{cmyk}{0.50,1,0,0}
\definecolor{Black}{cmyk}{0,0,0,1}
\definecolor{Violet}{cmyk}{0.79,0.88,0,0}
\definecolor{White}{cmyk}{0,0,0,0}
\definecolor{rltred}{rgb}{0.75,0,0}
\definecolor{rltgreen}{rgb}{0,0.5,0}
\definecolor{oneblue}{rgb}{0,0,0.75}
\definecolor{marron}{rgb}{0.64,0.16,0.16}
\definecolor{forestgreen}{rgb}{0.13,0.54,0.13}
\definecolor{purple}{rgb}{0.62,0.12,0.94}
\definecolor{dockerblue}{rgb}{0.11,0.56,0.98}
\definecolor{freeblue}{rgb}{0.25,0.41,0.88}
\definecolor{myblue}{rgb}{0,0.2,0.4}
\definecolor{Melon}{rgb}{ 0.46, 0.50, 0}
\definecolor{Melone}{rgb}{ 0, 0.46, 0.50}

\allowdisplaybreaks
\usepackage[colorinlistoftodos]{todonotes}
\usepackage{url}

\usepackage{graphicx,tikz} 

\title[Asymptotics for partially dissipative hyperbolic systems]{Large time asymptotics for partially dissipative hyperbolic systems without Fourier analysis: application to the nonlinearly damped p-system}

\author[T. Crin-Barat]{Timothée Crin-Barat${^*}$}
\address[T. Crin-Barat]{Chair for Dynamics, Control, Machine Learning and Numerics, Alexander Von Humboldt- Professorship, Department of
Mathematics, Friedrich-Alexander-Universität Erlangen-Nürnberg, 91058 Erlangen, Germany.}
\email{timothee.crin-barat@fau.de}

\author[L.-Y. Shou]{Ling-Yun Shou}
\address[L.-Y. Shou]{School of Mathematics and Key Laboratory of Mathematical MIIT, Nanjing University of Aeronautics and Astronautics, Nanjing, 211106,
P. R. China}
\email{shoulingyun11@gmail.com}

\author[E. Zuazua]{Enrique Zuazua}
\address[E. Zuazua]{Chair for Dynamics, Control, Machine Learning and Numerics, Alexander Von Humboldt- Professorship, Department of
Mathematics, Friedrich-Alexander-Universität Erlangen-Nürnberg, 91058 Erlangen, Germany.
	\newline \indent 
	Chair of Computational Mathematics, Fundación Deusto,	Avenida de las Universidades, 24, 48007 Bilbao, Basque Country, Spain. 
	\newline \indent
	Universidad Autónoma de Madrid, Departamento de Matemáticas, Ciudad Universitaria de Cantoblanco, 28049 Madrid, Spain.}
	\email{enrique.zuazua@fau.de}

\subjclass[201,0]{35Q35; 76N10}
\keywords{Partially dissipative hyperbolic systems, nonlinear damping, asymptotic estimates, hypocoercivity, Lyapunov functional.\\\quad$^*$ Corresponding author: timotheecrinbarat@gmail.com}

\begin{document}
\maketitle

\begin{abstract}
A new framework to obtain time-decay estimates for partially dissipative hyperbolic systems set on the real line is developed. Under the classical Shizuta-Kawashima (SK) stability condition, equivalent to the Kalman rank condition in control theory, the solutions of these systems decay exponentially in time for high frequencies and polynomially for low ones. This allows to derive a sharp description of the space-time decay of solutions for large time. However, such analysis relies heavily on the use of the Fourier transform that we avoid here, developing the ``physical space version" of the \textit{hyperbolic hypocoercivity} approach introduced by Beauchard and Zuazua in \cite{BZ}, to prove new asymptotic results in the linear and nonlinear settings.

The new physical space version of the hyperbolic hypocoercivity approach allows to recover the natural heat-like time-decay of solutions under sharp rank conditions, without employing Fourier analysis or $L^1$ assumptions on the initial data.
Taking advantage of this Fourier-free framework, we establish 
new enhanced time-decay estimates for initial data belonging to weighted Sobolev spaces.  These results are then applied to the nonlinear compressible Euler equations with linear damping. We also prove the logarithmic stability of the nonlinearly damped $p$-system.
\end{abstract}

\section{Introduction}
\label{sec:intro}

\subsection{Presentation of the model}
We study the long-time behaviour of one-dimensional partially dissipative hyperbolic systems of the form
\begin{equation}\label{SystGen}
\begin{aligned}
    &\partial_t U + A(U)\partial_xU =-BU, & (x,t)\in \mathbb{R}\times \mathbb{R}_{+},
\end{aligned}
\end{equation}
where $U=U(x,t)\in \mathbb{R}^{n}$ ($n\geq2$) is the vector-valued unknown, $A:\mathbb{R}^n\rightarrow\mathbb{R}^{n\times n}$ is a smooth matrix-valued symmetric function and $B$ is a  positive semidefinite symmetric $n\times n$ matrix.
 The system \eqref{SystGen} models non-equilibrium processes in physics for media with hyperbolic responses and also arises in the numerical simulation of conservation laws by relaxation schemes (see \cite{JinXin,vicenti1,zeng1} and references therein).

We assume that \eqref{SystGen} has a \textit{partially dissipative structure}: The matrix $B$ takes the form 
\begin{align} \label{BD}
B=\begin{pmatrix} 0 & 0 \\ 0 & D
    \end{pmatrix},
\end{align}
    with $D$ a positive definite symmetric $n_2\times n_2$  matrix $(1\leq n_{2}<n)$. 
    Under these conditions, $D$ satisfies the \textit{strong dissipativity condition}: there exists a constant $\kappa>0$ such that, for all $X\in\mathbb{R}^{n_{2}}$, 
    \begin{align}\label{Strong Dissipativity}
        \langle DX,X\rangle\geq \kappa |X|^2,
    \end{align}
    where $\langle\:,\rangle $ denotes the inner product on $\mathbb{R}^n$.
    
    A classical system fitting the description \eqref{SystGen}-\eqref{BD}, which we investigate in this manuscript, is the compressible Euler equations with damping:
\begin{equation}\label{Euler}
\left\{
\begin{aligned}
&\partial_{t}\rho+\partial_x (\rho u)=0,\\
&\partial_{t} (\rho u)+\partial_x (\rho u^2)+\partial_x P(\rho)=-\lambda \rho u,
\end{aligned}
\right.
\end{equation}
where $\rho=\rho(x,t)\geq 0$ denotes the fluid density,  $u=u(x,t)\in\mathbb{R}$ stands for the fluid velocity, $P(\rho)$ is the pressure function, and the friction coefficient $\lambda>0$ is assumed to be constant. For a $\gamma$-law pressure with the adiabatic coefficient $\gamma>1$, a standard symmetrization procedure (see \cite[Chapter 4, p.171-172]{HJR}) allows to rewrite the system \eqref{Euler} into the symmetric form \eqref{SystGen}:
\begin{equation}\label{SymEuler}
\left\{
\begin{aligned}
&\partial_{t}c+u\partial_x c+\frac{\gamma-1}{2}c\partial_x u=0,\\
&\partial_{t}u+u\partial_x u+\frac{\gamma-1}{2}c\partial_x c=-\lambda u,
\end{aligned}
\right.
\end{equation}
where $c:=2/(\gamma-1)\sqrt{\partial P(\rho)/\partial\rho}$ is a multiple of the speed of sound. The system \eqref{Euler} describes compressible gas flows passing through porous media and can be interpreted as a relaxation approximation (as $\lambda\to\infty$ and under a diffusive scaling, see \cite{mar1,CoulombelGoudon,CBD3,XuWang}) of the porous medium equation describing fluid flow, heat transfer or diffusion \cite{Vazquez}.
\medbreak
We are also interested in partial nonlinear dissipation phenomena. We investigate the stability of the nonlinearly damped $p$-system:
\begin{equation}\label{syst:EulerNLD}
\left\{
\begin{aligned}
 & \partial_t \rho+\partial_x u=0,\\
  &\partial_t u+\partial_x \rho +|u|^{r-1}u=0,\\
  &(\rho,u)(x,0)=(\rho_{0},u_{0})(x),
\end{aligned}
\right.
\end{equation}
with $r>1$. For small velocities $|u|\ll1$, the system \eqref{syst:EulerNLD} is often used to model gas networks, cf. \cite[eq.(1.2) p.2]{LeugeringMophou}. Moreover, the system \eqref{syst:EulerNLD} is strongly connected to the wave equations with nonlinear damping (see \eqref{wave}), for which numerous stability results have been established in contexts different from the one studied here. The interested readers may refer to \cite{HarauxZuazua,MochizukiMotai,NakaoJung,one1,Todorova1,Wakasa,ZuazuaX,ZuazuaX2} and references therein.

\medbreak
\subsection{Aims of the paper}
Fourier analysis is a very important tool in the study of linear and nonlinear PDEs, particularly for partially dissipative hyperbolic systems \eqref{SystGen}-\eqref{BD}. Their inherent frequency dependence, see Section \ref{sec:intro2}, has led most studies of their large-time asymptotics to rely on Fourier analysis (cf. \cite{BZ,BHN,CBD3,CBD2,SK,Kawashimadoctoral,KYDecay,XK2,XK1D} and references therein). However, Fourier analysis has a few limitations: it is not easily applied to equations set on bounded or exterior domains, it can make it harder to extract beneficial properties from nonlinear terms and to handle space-dependent matrices, and it is not well suited to analyse numerical schemes on non-uniform meshes.

In this paper, we develop a new method inspired by the hyperbolic hypocoercivity calculus in \cite{BZ} but entirely developed in the physical space so as to derive and pave the way for new asymptotic results that cannot be obtained via Fourier analysis. To do it we get inspired on the earlier works by H\'erau and Nier \cite{HerauNier,Herau} and Porretta and Zuazua \cite{PorrettaZuazua2016}.

We establish two main classes of results. First, we study general stability properties for the linearization of the system \eqref{SystGen} around a constant equilibrium, recovering
the optimal time-decay rates without using Fourier analysis and enhanced rates for initial data belonging to weighted Sobolev spaces. Then, we apply our new analysis to study the large time asymptotics of two concrete nonlinear systems:  the compressible Euler equations with linear damping \eqref{Euler} and the $p$-system with nonlinear damping \eqref{syst:EulerNLD}.

\subsection{Existing stability results} \label{sec:intro2}
Before presenting our main results, we recall some well-known properties concerning the stability of the system \eqref{SystGen}-\eqref{BD}. Its partially dissipative nature \eqref{BD} does not play a role when studying its local well-posedness but is crucial to justifying large-time results. For $B=0$, \eqref{SystGen} reduces to a system of hyperbolic conservation laws, and it is well-known that for smooth initial data there exist local-in-time solutions \cite{Katolocal,Majdalocal,Serre} that may develop singularities (shock waves) in finite time \cite{Dafermos1,Lax1}. On the other hand, when $\textrm{rank}(B)=n$, Li \cite{LiTT} proved the existence of global-in-time solutions that are exponentially damped. In our partially dissipative setting \eqref{BD}, i.e. $\textrm{rank}(B)<n$, the dissipation induced by $BU$ lacks coercivity, as it affects only some components of the solution. Nevertheless, as observed by Shizuta and Kawashima \cite{SK,Kawashimadoctoral},  interaction effects between the hyperbolic and dissipative parts of the system may generate dissipation in directions that are not affected by $B$. More recently, Beauchard and Zuazua \cite{BZ} have framed this phenomenon in the spirit of Villani's hypocoercivity theory \cite{Villani} and improved its understanding. Below, we present some key points of their approach.
\medbreak
Linearizing the system \eqref{SystGen}-\eqref{BD} around  a constant equilibrium $\bar{U}\in \text{Ker}(B)$ and denoting $A:=A(\bar{U})$, we obtain the linear system
\begin{equation}\label{Systlin}
\begin{aligned}
    &\partial_t U + A\partial_xU =-BU.
\end{aligned}
\end{equation}
To further highlight the partially dissipative structure of \eqref{Systlin}-\eqref{BD}, we write $U=(U_1,U_2)$ where $U_1\in \textrm{Ker}(B)=\mathbb{R}^{n_1}$ and $U_2\in  \textrm{Im}(B)=\mathbb{R}^{n_2}$ with $n_{1},n_{2}$ satisfying $1\leq n_{1}, n_{2}<n$ and $n_{1}+n_{2}=n$. By the decomposition
\begin{equation}\nonumber
\begin{aligned}
A=\begin{pmatrix} 
A_{1,1} & A_{1,2} \\ A_{2,1} & A_{2,2}
    \end{pmatrix},
\end{aligned}
\end{equation}
we rewrite the system \eqref{Systlin} as
    \begin{equation}\label{GE}
    \begin{aligned}
    \left\{ 
    \begin{matrix}
    \displaystyle \hspace{-0.8cm}\partial_tU_1 +  A_{1,1}\partial_{x}U_1+ A_{1,2}\partial_{x}U_2=0,\\ 
    \displaystyle  \partial_tU_2 +  A_{2,1}\partial_xU_1+ A_{2,2}\partial_x U_2=-D U_2, 
    \end{matrix} \right. 
\end{aligned} 
\end{equation}
with the initial datum $(U_1,U_2)(x,0)=U_0(x)=(U_{1,0},U_{2,0})(x)$.

Using the symmetry of $A$ and the condition \eqref{Strong Dissipativity}, one has the energy dissipation law
\begin{align}\label{ZZ}\dfrac{1}{2}\dfrac{d}{dt}\|(U_1,U_2)(t)\|_{L^2}^2+\kappa\|U_2(t)\|_{L^2}^2\leq0.
\end{align}
The law \eqref{ZZ} demonstrates a lack of coercivity as only dissipation for the component $U_2$ is observed. To recover dissipation for $U_1$, in \cite{BZ}, the authors first apply the Fourier transform to \eqref{Systlin}, which yields the parameterized ODE 
\begin{align}\label{ODE}\partial_t\widehat{U}+i\xi A \widehat{U}+B\widehat{U}=0.
\end{align}
Then, a key observation is that, for a fixed $\xi\ne 0$, the exponential stability of the solutions of \eqref{ODE} is equivalent to the {\emph{Kalman rank condition}} for the pair $(A,B)$:
        \begin{align} \label{K}
 \text{The matrix    }\:\mathcal{K}(A,B):=(B,AB,\ldots,A^{n-1}B)\quad \text{has full rank } n.
 \end{align}
Such a result, established for instance in \cite{Kalman} in the context of control theory, shows that the large-time stability can hold even if the rank of the dissipative matrix is not full. For hyperbolic systems, results in the same vein can be obtained, but due to the presence of the parameter $\xi$, especially as $\xi\to 0$, uniform exponential stability may not be expected. 

In order to get decay estimates with explicit control on the dependence of the frequency parameter $\xi$, inspired by hypocoercivity arguments, in \cite{BZ} the following Lyapunov functional was introduced:
\begin{equation}\label{intro:lya}
  \mathcal{L}_{\xi}(t):=|\widehat{U}|^2+\min(\frac{1}{|\xi|},|\xi|)\, {\textrm{Re}} \sum_{k=1}^{n-1} \epsilon_k \langle BA^{k-1}\widehat{U}, BA^k\widehat{U}\rangle.
  \end{equation}
The following proposition plays a fundamental role when quantifying the decay rates in terms of the rank conditions, as a function of $\xi$.
\begin{proposition}[\cite{BZ}] \label{PropBZ}
    Let $A$ and $B$ be symmetric matrices such that $B$ satisfies \eqref{BD}. The following conditions are equivalent.
    \begin{itemize}
        \item The system \eqref{Systlin} is polynomially stable and behaves as the heat equation for large times.
        
        \item The pair $(A,B)$ verifies the Kalman rank condition \eqref{K}.
        
\item For every $y\in \mathbb{C}^n$, there exist constants $c,C>0$ such that \vspace{-0.2cm}
\begin{align}\label{intro:normeq}
c|y|^2\leq \bigg(\sum_{k=0}^{n-1} |BA^{k}y|^2\bigg)^{\frac{1}{2}} \leq C|y|^2.\end{align}
 \end{itemize}

\end{proposition}
The parameters $\epsilon_k$ in the Lyapunov functional \eqref{intro:lya} need to be chosen small enough and the frequency-weight $\min(1/|\xi|,|\xi|)$ are used so the dissipative and hyperbolic effects interact efficiently. Once this is done, differentiating  \eqref{intro:lya} with respect to time and using the fact that \eqref{intro:normeq} define a full norm, for some constant $c>0$, one obtains
\begin{align}\label{DecayBZ}
  |\widehat{U}(\xi,t)|^2\lesssim |\widehat{U_{0}}(\xi)|^2e^{-c\min\{1,|\xi|^2\}t}
 \end{align}
for all frequencies $\xi$, which leads to sharp decay rates (see Proposition \ref{PropBZ}). 

For more details concerning this approach, the interested reader may refer to Section \ref{sec:appedixSK} or directly to \cite{BZ}.

\medbreak

\subsection{Outline of the paper} 
As mentioned above, the main goal of this paper is to develop the physical space version of the frequency-dependent hypocoercivity calculus. 

In Section \ref{sec:mainres}, we state our main results and present our methodology. Section \ref{sec:3} is devoted to proving the natural time-decay estimates for linear systems without Fourier analysis and without additional $L^1$ regularity assumption on the initial data. These estimates are further improved in Sections \ref{sec:4}-\ref{sec:wavemethod} under additional space-weighted conditions on the initial data. Section \ref{sec:EulerLin} is devoted to the analysis of the nonlinear Euler system \eqref{Euler}, while the nonlinearly damped $p$-system \eqref{syst:EulerNLD} is studied in Section \ref{sec:EulerNonlinD}. Section \ref{sec:ext} presents additional results and comments on possible extensions of our methods. Some technical lemmas are relegated to the appendix.

\section{Main results and methodology}\label{sec:mainres}

\subsection{Natural time-decay rates without Fourier analysis} 
In our first result, we retrieve the natural large-time asymptotics of linear partially dissipative hyperbolic systems \eqref{Systlin} without using Fourier analysis or $L^1$-type assumptions on the initial data.

\begin{theorem}\label{DecayThm1}
 Let $U_{0}\in H^1$, $A$ and $B$ be symmetric $n\times n$ matrices with $B$ as in \eqref{BD}, satisfying the Kalman rank condition \eqref{K}. Then, for all $t>0$, the solution $U$ of  \eqref{Systlin} with the initial datum $U_{0}$ satisfies
 \begin{equation}
\label{e:decayThm1}
\| U_2(t)\|_{L^2}+\|\partial_x U(t)\|_{L^2} \leq C(1+t)^{-\frac{1}{2}}\|U_0\|_{H^1},
\end{equation}
where $C>0$ is a constant independent of time and $U_0$.
\end{theorem}
 \begin{remark}[On the sharpness of the decay rates]\label{rmq:optimaldecay} \hfill
    \begin{itemize} 
    \item According to Proposition \ref{PropBZ} and the Fourier representation  \eqref{DecayBZ}, the rate obtained for $U$ in the estimate \eqref{e:decayThm1} is sharp in our Hilbertian functional framework.  It actually corresponds to that of the heat equation, cf. Lemma \ref{lemmaheat0} in the appendix, but, due to the absence of regularizing effects, under the assumption that the initial datum lies in $H^1$ but not only in $L^2$.
    
       \item Unlike Fourier-based approaches, our method is not well suited to exploiting the $L^1$ control of the initial data. However, it is adaptable to weighted spaces, as we will see in the next results.
    \end{itemize}
    \end{remark}

\vspace{1mm}

\textit{Strategy of proof of Theorem \ref{DecayThm1}.} The core of the proof is the construction of an augmented energy functional in the spirit of \eqref{intro:lya} but entirely defined in the physical space. Inspired by the works of H\'erau and Nier \cite{HerauNier,Herau} and Porretta and Zuazua \cite{PorrettaZuazua2016}, concerning the asymptotic decay of kinetic equations, we consider the time-weighted functional
\begin{equation}\label{def:GeneralLyaNoF0}
\begin{aligned}
  \mathcal{L}(t)= \|U(t)\|_{H^1}^2+\eta_{0} t\|\partial_x U(t)\|_{L^2}^2+\sum_{k=1}^{n-1} \varepsilon_k \bigl(BA^{k-1}U,BA^k\partial_xU\bigl)_{L^2},
  \end{aligned}
  \end{equation}
where $(\,,\,)_{L^2}$ denotes the inner product in $L^2$. Setting the constants $\eta_{0}$ and $\var_{k}$, $k=1,2,...,n-1$, suitably small and using the Kalman rank condition \eqref{K} and Proposition \ref{PropBZ}, we obtain
\begin{align}
\label{Intro:lya2}
\frac{d}{dt}\mathcal{L}(t)+c\Big(\|U_{2}(t)\|_{L^2}^2+t\|\partial_x U_2(t)\|_{L^2}^2+\|\partial_x U(t)\|_{L^2}^2\Big)\leq 0,
\end{align}
from which we infer
\begin{align}\label{eq:comthm1}
\|\partial_xU(t)\|_{L^2}\leq C(1+t)^{-\frac{1}{2}}\|U_0\|_{H^1}. 
\end{align} 
Then, combining \eqref{eq:comthm1} with the fact that $U_2$ verifies a damped equation with the linear source term $-A_{2,1}\partial_{x}U_1-A_{2,2}\partial_{x}U_2$, Grönwall's inequality yields the extra decay for $U_2$ stated in \eqref{e:decayThm1}.

\medbreak

\subsection{Enhanced decay rates in weighted Sobolev spaces}

As depicted in Proposition \ref{PropBZ}, the solutions of \eqref{Systlin} behave similarly to the solutions of the heat equation as the time evolves. More precisely, as explained in \cite{BHN,HsiaoLiu92CMP,Liu87CMP,Yong}, they decay as the solutions of\begin{equation}\label{parabolic}
\left\{
\begin{aligned}
&\partial_{t}N-A_{1,2}D^{-1}A_{2,1} \partial^2_{x}N=0,\\
&N(x,0)=U_{1,0}(x),
\end{aligned}
\right.
\end{equation}
where the operator $-A_{1,2}D^{-1}A_{2,1}\partial_x^2$ is strongly elliptic when the Kalman rank condition \eqref{K} holds, see \cite{CBD3,Yong0,Yong}.
Based on Fourier analysis tools, Bianchini, Hanouzet and Natalini \cite{BHN} obtained faster decay rates of the error of solutions between \eqref{Systlin} and \eqref{parabolic} using $L^1$-type assumptions on the initial data. 

 Our next theorem justifies the validity of the large-time parabolic profile \eqref{parabolic} of the system \eqref{Systlin} 
without Fourier analysis and provides new enhanced time-decay rates when the initial datum belongs to weighted Sobolev spaces.

\begin{theorem}\label{ThmDecayGeneral1}

 Let the hypotheses of Theorem \ref{DecayThm1} be satisfied and let $U$ be the solution of \eqref{Systlin} with the initial datum $U_0$. In addition, suppose also that $|x| U_{2,0}\in L^2$ and
 \begin{align}\label{a11}
     A_{1,1}=0.
 \end{align} 
 Then,  for all $t>0$ and $0<\var\ll1/2$,
 \begin{align}
 &\|(U_1-N)(t)\|_{L^2}\leq C(1+t)^{-\frac{1}{2}+\var} (\|U_0\|_{H^1}+\| |x| U_{2,0}\|_{L^2}),\label{U1NL2}
\end{align}
where $N$ is the solution of \eqref{parabolic} associated to the initial datum $U_{1,0}$ and $C>0$ is a constant independent of time and $U_{0}$.

 Furthermore, let $0<\mu\leq 1$. If we further assume 
\begin{align}  X_0:=\|U_{0}\|_{H^1}+\||x|^{\mu}U_{1,0}\|_{L^2}+\| |x| U_{2,0}\|_{L^2}< \infty,\label{a112}
\end{align}
then, for all $t>0$,
\begin{equation}\label{decay111}
\left\{
\begin{aligned}
&\|U(t)\|_{L^2}\leq C (1+t)^{-\frac{\mu}{2}}X_0,\\
&\|U_2(t)\|_{L^2}+\|\partial_x U(t)\|_{L^2}\leq C(1+t)^{-\frac{\mu}{2}-\frac{1}{2}}X_0.
\end{aligned}
\right.
\end{equation}
\end{theorem}

\begin{remark}

Some comments are in order.

\begin{itemize}

  \item In the context of fluid mechanics, the condition \eqref{a11} typically implies that the velocity equilibrium satisfies $\bar{u}=0$. 
   It is a natural condition when $U_1$ is a scalar (like e.g. the fluid density) as it is always satisfied up to the Galilean change of frame $(x,t)\rightarrow (x-A_{1,1}t,t)$, see \cite[p. 6]{Arnold}. 
\item When $1/2<\mu\leq 1$, the assumption $X_0<\infty$ is a stronger condition than the $L^1$ assumption usually used in Fourier-based approaches, but allows for faster decay rates. For $\mu=1/2$, the rate in $\eqref{decay111}_1$ would correspond to the one recovered with Fourier analysis and for initial data belonging to $L^1$, cf. Lemma \ref{appendix1}.

\end{itemize}
\end{remark}

\textit{Strategy of proof of Theorem \ref{ThmDecayGeneral1}.} 
Inspired by considerations from \cite{CBD3,CBD2}, we consider the damped mode \begin{align}
R:=D^{-1}A_{2,1}U_1+D^{-1}A_{2,2} U_2+\int^x_{-\infty} U_2(y,t)\,dy,\label{R}
\end{align}
which has faster decay rates compared to $U_1$ and $U_2$ since $R$  satisfies the purely damped system
\begin{equation}\label{eqR}
\begin{aligned}
&\partial_{t}R+D R=D^{-1}A_{2,1} \partial_{t} U_1+D^{-1} A_{2,2} \partial_{t} U_{2}.
\end{aligned}
\end{equation}
Inserting $R$ in the system satisfied by $U_1$, we have the parabolic system
\begin{equation}
\partial_{t}U_1-A_{1,2}D^{-1}A_{2,1} \partial^2_{x} U_1=\partial_{x}^2(-A_{1,2} R+D^{-1}A_{2,2} U_2).\label{U1para1}
\end{equation}
 Estimating the system satisfied by $U_1-N$ and using the fact that $R$ and $U_2$ decay rapidly, we obtain the asymptotic stability estimate \eqref{U1NL2}. Then, in order to recover $L^2$ decay estimates for $U_1$, we first recover the $t^{-\frac{\mu}{2}}$ time-decay estimates of $\|N(t)\|_{L^2}$ by performing an energy argument on \eqref{parabolic} and using the Caffarelli-Kohn-Nirenberg inequality. Together with \eqref{U1NL2}, this yields $\eqref{decay111}_1$ for $0<\mu<1$. In the case $\mu=1$, we prove $\eqref{decay111}_1$ by deriving time-weighted energy estimates for $(U_1-N,U_2)$ in the spirit of Theorem \ref{DecayThm1}. Finally, combining the $L^2$ decay we obtained for $U_1$ together with the Lyapunov inequality \eqref{Intro:lya2} obtained in Theorem \ref{DecayThm1}, we derive the faster decay rates $\eqref{decay111}_{2}$.
\bigbreak

In our next result, we establish additional enhanced decay estimates, for a larger class of initial data, but under stronger structural conditions on the system.

\begin{theorem}\label{ThmDecayGeneral2}
Let $1/2<\mu\leq 1$ and assume that the hypotheses of Theorem \ref{DecayThm1} are satisfied.  Suppose also \eqref{a11}, $n_1=n_2$, $A_{1,2}A_{2,1}$ positive definite,
\begin{align}
    Y_{0}:=\|U_{0}\|_{H^1}+\||x|^{\mu} U_{1,0}\|_{L^2}+\||x|^{\mu-\frac{1}{2}} U_{2,0}\|_{L^2} < \infty,\label{a21}
\end{align}
and $\partial_{t}U_{1} |_{t=0}=-A_{1,2}\partial_x U_{2,0}$.
In addition, in the case $1/2<\mu<1$, let $|A_{1,2}A_{2,1}|\leq 1$. Then, for all $t>0$, the solution $U$ of the system \eqref{Systlin} with the initial datum $U_{0}$ satisfies
\begin{equation}\label{decay1112}
\left\{
\begin{aligned}
&\|U(t)\|_{L^2}\leq C(1+t)^{-\mu+\frac{1}{2}}Y_0,\\
&\|U_2(t)\|_{L^2}+\|\partial_x U(t)\|_{L^2}\leq C(1+t)^{-\mu}Y_0,
\end{aligned}
\right.
\end{equation}
where $C>0$ is a constant independent of time, $\mu$ and $U_{0}$.
\end{theorem}
\begin{remark} Some remarks are in order.
\begin{itemize}
\item The time-decay rates obtained in \eqref{decay1112} are sharp for $\mu=1$, but for $\mu<1$ we obtain a slower decay compared to Theorem \ref{ThmDecayGeneral1} since $-\mu+1/2> -\mu/2$.

\item 
The conditions imposed on $A$ ensure that the system \eqref{Systlin} can be rewritten in a wave-like formulation (see \eqref{wave0}). Such conditions are usually imposed on the equilibrium states used to study the stability of systems in fluid mechanics such as the compressible Euler equations \eqref{Euler}, see \cite{HJR}.

\item One of the main interesting aspects of Theorem \ref{ThmDecayGeneral2} is that it can be extended to deal with nonlinear dissipative phenomena (cf. Section \ref{mainres:NLD}), contrary to Theorem \ref{ThmDecayGeneral1}.
\end{itemize}
\end{remark}

\vspace{1mm}

\textit{Strategy of proof of Theorem \ref{ThmDecayGeneral2}.} Defining the unknown $W$ such that $U_{1}=\partial_x W$ and $U_{2}=-A_{1,2}^{-1}\partial_{t} W$, we rewrite the system \eqref{Systlin} as the extended damped wave equation 
\begin{align} 
&\partial_{t}^2 W-A_{1,2} A_{2,1} \partial_{x}^2 W +A_{1,2} A_{2,2}A_{1,2}^{-1} \partial_t \partial_x W+ A_{1,2} D A_{1,2}^{-1}\partial_t W =0.\label{wave0}
\end{align}
Under this formulation, it suffices to estimate the wave energy $\|(\partial_{t}W,\partial_{x}W)(t)\|_{L^2}^2$ to get the decay for $U$ in $L^2$. To this end, we generalize and combine the works of Mochizuki and Motai \cite{MochizukiMotai} and Ikehata \cite{Ikehata1D}. In Lemma \ref{lemmaW}, performing hypocoercivity estimates with space-time weights on \eqref{wave0}, we get
\begin{equation}\label{intro:decay1112}
\begin{aligned}
&\int_{\mathbb{R}} (1+t+|x|)^{2\mu-1}(|\partial_{t}W|^2+|\partial_{x}W|^2)\,dx\lesssim Y^2_0,
\end{aligned}
\end{equation}
which only holds for $\mu\leq 1$ and provides the desired decay rates of $U$ in  \eqref{decay1112} since $(1+t)^{2\mu-1}\leq (1+t+|x|)^{2\mu-1}$, for $x\in \mathbb{R}$ and $\mu \geq 1/2$.
The faster rates of $U_2$ and $\partial_x U$ in \eqref{decay1112} are recovered by combining estimate \eqref{intro:decay1112} with the hypocoercive inequality \eqref{Intro:lya2}.
\medbreak

\subsection{Application to the compressible Euler equations with damping}

In our next result, we justify a nonlinear analogue of Theorem \ref{DecayThm1} for the compressible Euler system with linear damping \eqref{Euler}, when the initial datum is close to a constant equilibrium in $H^2$.

\begin{theorem}\label{ThmEuler1}
Consider the equilibrium state $(\bar{\rho},0)$, with $\bar{\rho}>0$ a given constant, and suppose that 
\begin{equation}
    \begin{aligned}
        &P(\rho)\in C^{\infty}(\mathbb{R}_{+}),\quad\quad P'(\rho)>0.\label{P}
    \end{aligned}
\end{equation}
Then, there exists a constant $\delta_0>0$ such that if the initial datum $(\rho_{0},u_{0})$ satisfies
\begin{align}
    &\|(\rho_{0}-\bar{\rho},u_{0})\|_{H^{2}}\leq \delta_{0},\label{eulera1}
\end{align}
the system \eqref{Euler} admits a unique global-in-time solution $(\rho,u)\in C(\mathbb{R}_{+};H^{2} )$ 
such that 
\begin{equation}\label{eulerdecay1}
      \begin{aligned}
          & \|u(t)\|_{L^2}+\|\partial_x (\rho-\bar{\rho},u)(t)\|_{L^2}\leq C(1+t)^{-\frac{1}{2}}\|(\rho_{0}-\bar{\rho},u_{0})\|_{H^{2}},
      \end{aligned}
  \end{equation}
where $C>0$ is a constant independent of time and $U_{0}$.
  \end{theorem}

The next result corresponds to nonlinear versions of Theorems \ref{ThmDecayGeneral1} and \ref{ThmDecayGeneral2} for the system \eqref{Euler}. 

\begin{theorem} \label{ThmEuler2}
Assume that the hypotheses of Theorem \ref{ThmEuler1} are satisfied and let $(\rho,u)$ be the global solution of the system \eqref{Euler} subject to the initial datum $(\rho_0,u_0)$. There exists a constant $C>0$ independent of time such that the following statements hold. 
\medbreak
$\bullet$ $1)$ If, in addition to \eqref{eulera1}, we assume $|x| u_{0}\in L^2 $. Then,  for all $t>0$ and $0<\var\ll1/4$, there exists a constant independent of $T$ such that 
\begin{align}
&\|(\rho-\rho_*)(t)\|_{L^2}\leq C (1+t)^{-\frac{1}{4}+\var},\label{errorEuler}
\end{align}
where $\rho_*$ solves the parabolic equation
\begin{equation}\label{parabolicEuler}
\left\{
\begin{aligned}
&\partial_{t}\rho_*-\frac{P'(\bar{\rho})}{\lambda}\partial_x^2 \rho_*=0,\\
&\rho_*(0,x)=\rho_0(x).
\end{aligned}
\right.
\end{equation}
If for any $0<\mu\leq 1$, we further assume  $ |x|^{\mu}(\rho_{0}-\bar{\rho})\in L^2$, then, for all $t>0$, 
 \begin{equation}\label{eulerdecay2}
  \left\{
      \begin{aligned}
                   &\|(\rho-\bar{\rho})(t)\|_{L^2}\leq C (1+t)^{-\frac{\mu}{2}},\\
          & \|u(t)\|_{L^2}+\|\partial_x (\rho-\bar{\rho},u)(t)\|_{L^2}\leq  C (1+t)^{-\frac{\mu}{2}-\frac{1}{2}}.
      \end{aligned}
      \right.
  \end{equation}

\medbreak
$\bullet$ $2)$ Let $1/2<\mu\leq 1$. If, in addition to \eqref{eulera1}, we assume $(|x|^{\mu}(\rho_{0}-\bar{\rho}),|x|^{\mu-\frac{1}{2}}u)\in L^{2} $  and $\partial_t \rho|_{t=0}=-\partial_x(\rho_{0}u_{0})$. Additionally, let $P'(\bar{\rho})\leq 1$ in the case  $1/2<\mu<1$. Then, for all $t>0$,
  \begin{equation}\label{eulerdecay3}
  \left\{
      \begin{aligned}
          &\|(\rho-\bar{\rho})(t)\|_{L^2}\leq C(1+t)^{-\mu+\frac{1}{2}},\\
          &\|u(t)\|_{L^2}+ \|\partial_x (\rho-\bar{\rho},u)(t)\|_{L^2}\leq C(1+t)^{-\mu}.
\end{aligned}
\right.
\end{equation}

\end{theorem}

\begin{remark}
Due to the nonlinear term arising from the pressure, the decay rate for the error unknown in \eqref{errorEuler} is slower than the linear case in \eqref{U1NL2}. 
\end{remark}

\bigbreak
\subsection{Application to the nonlinearly damped $p$-system}\label{mainres:NLD}

Our final result concerns nonlinear dissipative phenomena. We justify the logarithmic stability of the system \eqref{syst:EulerNLD} when the initial datum belongs to logarithmically weighted Sobolev spaces.

\begin{theorem}\label{ThmEulerNLD}
Let $1<r<3$ and suppose that  $(\rho_{0},u_{0})\in  H^1 $. Then, for all $t>0$, the system \eqref{syst:EulerNLD} with the initial datum $(\rho_0,u_0)$ admits a unique global-in-time solution $ (\rho,u)\in C(\mathbb{R}_{+};H^1 )$ satisfying
\begin{equation}
\begin{aligned}
&\| (\rho,u)(t)\|_{H^1}^2+\int_{0}^{t}\Big(\|u(\tau)\|_{L^{r+1}}^{r+1}+\|(\partial_x \rho^{\frac{r+1}{2}},\partial_x u^{\frac{r+1}{2}})(\tau)\|_{L^2}^2\Big)\,d\tau\\
&\quad\leq C\| (\rho_{0},u_{0})\|_{H^1}^2,
\end{aligned}
\end{equation}
where $C>0$ is a constant independent of time and $(\rho_{0},u_{0})$.
\smallbreak
 Furthermore, let $q>0$ and suppose
    \begin{align}
  \rho_0 \in L^1, \quad \quad \log^{q}{(1+ |x|)} (\rho_{0},u_{0})\in L^2,\label{NDw}
    \end{align} 
    and $\partial_{t}\rho|_{t=0}=-\partial_x u_{0}$. Then,  for all $t>0$,
\begin{equation}
\begin{aligned}\label{logdecay}
&\| (\rho,u)(t)\|_{L^2}\leq \frac{C_{q}}{\log^{q}(1+t)},
\end{aligned}
\end{equation}
 where $ C_{q}>0$ is a constant independent of time.
\end{theorem}

\vspace{1mm}

\begin{remark}

Some remarks are in order.

\begin{itemize}
\item The restriction $1<r<3$ appears naturally in our computations, see inequality \eqref{conditionr3}. It comes from the fact that the nonlinear dissipation becomes weaker and weaker when $r$ grows. Moreover, it is known that the solutions of the nonlinearly damped wave equation decay in time for  $1<r<3$ and do not decay if $r>3$  (cf. \cite{daoulatli1,MochizukiMotai,nakao}). For the system \eqref{syst:EulerNLD}, the asymptotic behaviour in the case $r=3$ is open.

   \item In contrast with the linear damping setting ($r=1$), it would be difficult to handle the nonlinear damping term $|u|^{r-1}u$ via Fourier analysis, since this would result in a non-local nonlinear term mixing all frequency components. 

   \item  Our proof is inspired by the hypocoercive arguments developed in the proof of Theorem \ref{ThmDecayGeneral2} and the results from \cite{MochizukiMotai}, see Section \ref{sec:EulerNonlinD} for more details.

    \end{itemize}
\end{remark}

\subsection{Conclusions}
In this paper, we lay the foundations for  \textit{hyperbolic hypercoercivity} without resorting to Fourier analysis. This enables us to extend the analysis of partially dissipative hyperbolic systems to previously unexplored contexts. Furthermore, our approach provides a guideline for future research on partially dissipative systems in many contexts, such as initial boundary value problems, more general nonlinearly damped systems, space-dependent hyperbolic matrices, and numerical approximation schemes. Interested readers can refer to Section \ref{sec:ext} for more details on these possible extensions.
\bigbreak

\section{Proof of Theorem \ref{DecayThm1}} \label{sec:3}

In this section, we prove Theorem \ref{DecayThm1} by employing pure energy arguments and avoiding the use of the Fourier transform. We introduce the Lyapunov functional
\begin{equation}\label{def:GeneralLyaNoF}
\begin{aligned}
  \mathcal{L}(t):= \|U(t)\|_{H^1}^2+\eta_{0} t\|\partial_x U(t)\|_{L^2}^2+\cI(t),
  \end{aligned}
  \end{equation}
where the corrector term  $\cI(t)$ is defined by
\begin{align}
&\cI(t):= \sum_{k=1}^{n-1} \varepsilon_k \bigl(BA^{k-1}U,BA^k\partial_xU\bigr)_{L^2},\label{I}
\end{align}
with positive constants $\eta_{0}$ and $\var_{i}$, $i=1,2,...,k-1$, to be determined later.

\subsection{Time-derivative of $\mathcal{L}$}
\label{subsectionL}
\subsubsection{Energy estimates}
Standard energy estimates for \eqref{GE} lead to
\begin{align}
    \dfrac{d}{dt} \|U(t)\|_{L^2}^2+2\bigl(DU_2,U_2\bigl)_{L^2}&=0, \nonumber\\
    \dfrac{d}{dt} \|\partial_x U(t)\|_{L^2}^2+2\bigl(D \partial_x U_2,\partial_x U_2\bigl)_{L^2}&=0,\nonumber
     \\
    \dfrac{d}{dt}\big(t \|\partial_xU(t)\|_{L^2}^2\big)+2t\bigl(D \partial_x U_2,\partial_x U_2\bigl)_{L^2}&=\|\partial_xU(t)\|_{L^2}^2,\nonumber
\end{align}
which, together with the strong dissipativity condition \eqref{Strong Dissipativity} of $D$, gives
\begin{equation}
\begin{aligned}
    &\dfrac{d}{dt}\big(\mathcal{L}(t)-\cI(t) \big)+2\kappa \|U_2(t)\|_{L^2}^2+2\kappa(1+ \eta_{0} t) \|\partial_xU_2(t)\|_{L^2}^2\leq \eta_{0}\|\partial_xU(t)\|_{L^2}^2.\label{L}
\end{aligned}
\end{equation}
In \eqref{L}, dissipative estimates for $\partial_xU$ in $L^2 $ are missing. They are recovered in the next subsection thanks to the corrector term $\cI(t)$ in the Lyapunov functional \eqref{def:GeneralLyaNoF}.

\subsubsection{Estimation of the corrector term}
Differentiating $\cI(t)$ in time, we obtain
\begin{equation}\label{TimederivativeCorrector}
\begin{aligned}
  \frac d{dt}\cI(t) + \sum_{k=1}^{n-1} \varepsilon_k\|BA^k\partial_xU(t)\|_{L^2}^2 
 =& -\sum_{k=1}^{n-1} \varepsilon_k\bigl(BA^{k-1} BU,BA^k\partial_xU\bigr)_{L^2}
  \\&-\sum_{k=1}^{n-1} \varepsilon_k \bigl(BA^{k-1}U,BA^k B \partial_xU\bigr)_{L^2}
  \\&-\sum_{k=1}^{n-1} \varepsilon_k \bigl(BA^{k-1}U, BA^{k+1}\partial_{x}^2U\bigr)_{L^2}\cdotp
\end{aligned}
\end{equation}
Thanks to Proposition \ref{PropBZ}, the second term on the left-hand side of \eqref{TimederivativeCorrector} leads to time-decay information for $\partial_xU$. To deal with the remainder terms, we proceed as in \cite{BZ,CBD2,Handbook} with some adaptations needed to bypass the Fourier analysis. 

\begin{lemma}[Time-derivative of $\cI$]\label{lem:corrector}
For any positive constant $\var_{0}$, there exists a sequence $\{\var_{k}\}_{k=1,...,n-1}$ of small positive constants such that
\begin{align}
\dfrac{d}{dt}\cI(t) +\frac{1}{2}\sum_{k=1}^{n-1} \varepsilon_k\|BA^k\partial_xU(t)\|_{L^2}^2 \leq \var_{0}\|U_2(t)\|_{L^2}^2+\var_{0}\| \partial_xU_2(t)\|_{L^2}^2.\label{ddtI}
\end{align}
\end{lemma}

\begin{proof}

To begin with, we fix a positive constant $\varepsilon_0$ and estimate the terms in the right-hand side of \eqref{TimederivativeCorrector} as follows. 
\begin{itemize}
\item The terms $\cI^1_k:= \varepsilon_k\bigl(BA^{k-1}B  U, BA^k\partial_x U\bigr)$
with $k\in\{1,\cdots, n-1\}$: Due to $B U=D  U_{2}$ and the fact that the matrices $A$, $D$ are bounded, we obtain
\begin{equation}\nonumber
\begin{aligned}
 |\cI^1_k| &\leq C\varepsilon_k\|D U_{2}(t)\|_{L^2} \|BA^k \partial_x U(t)\|_{L^2} \\
&\leq \frac{\varepsilon_0}{4n}{\| U_{2}(t)\|_{L^2}^2}+\frac{C\varepsilon_k^2}{\varepsilon_0}\| BA^k   \partial_xU(t)\|_{L^2}^2.
\end{aligned}
\end{equation}
\item  The term $\cI^2_1:= \varepsilon_1\bigl(BU, BA B \partial_xU\bigr)_{L^2}$: One has
\begin{equation}\nonumber
\begin{aligned}
|\cI^2_1| &\leq C\varepsilon_1 \|DU_{2}(t)\|_{L^2}\| D\partial_xU_{2}(t)\|_{L^2}\\
&\leq \frac{\varepsilon_0}{4n} \|U_{2}(t)\|_{L^2}^2+ \frac{C\varepsilon_1^2}{\var_{0}} \|\partial_xU_{2}(t)\|_{L^2}^2.
\end{aligned}
\end{equation}
\item The terms  $\cI^2_k:=\varepsilon_k\bigl(BA^{k-1}  U, BA^k B \partial_xU\bigr)_{L^2}$ with  $k\in\{2,\cdots, n-1\}$ if $n\geq3$: We deduce, after integrating by parts, that
\[\begin{aligned}
 |\cI^2_k|&=\varepsilon_k \big|\bigl(BA^{k-1}  \partial_x U, BA^k B U\bigr)_{L^2} \big|\\
 &\leq C\varepsilon_k \|BA^{k-1} \partial_x U(t)\|_{L^2}\|B U(t)\|_{L^2}\\
&\leq \frac{\varepsilon_0}{4n}{\|  U_{2}(t)\|_{L^2}^2}+\frac{C\varepsilon_{k-1}^2}{\varepsilon_0}\| BA^{k-1}   \partial_xU(t)\|_{L^2}^2.
\end{aligned}\]
\item The terms $\cI^3_k:= \varepsilon_k \bigl( BA^{k-1}  U,  BA^{k+1}  \partial_{x}^2U\bigr)_{L^2}$ with $k\in\{1,\cdots, n-2\}$ if $n\geq3$: A similar argument yields
\begin{equation}\nonumber
\begin{aligned}
\cI^3_k&=\varepsilon_k \big| \bigl( BA^{k-1}  \partial_x U,  BA^{k+1}  \partial_{x} U\bigr)_{L^2} \big|\\
&\leq \frac{\varepsilon_{k-1}}{8}  \|BA^{k-1}  \partial_xU(t)\|_{L^2}^2 +\frac{C\varepsilon_k^2}{\varepsilon_{k-1}}  \| BA^{k+1} \partial_x U(t)\|_{L^2}^2. 
\end{aligned}
\end{equation}
\item  The term $\cI^3_{n-1}:= \varepsilon_{n-1} \bigl( BA^{n-2}  U,  BA^{n}  \partial_{x}^2U\bigr)_{L^2}$:
Owing to the Cayley-Hamilton theorem, there exist coefficients
$c_{*}^j$ ($j=0,2,...,n-1$) such that
\begin{align}
    &A^n =\sum_{j=0}^{n-1}  c_{*}^j A^j.\label{CH}
\end{align}
Consequently, one gets
\begin{equation}\nonumber
\begin{aligned}
\hspace{-0.2cm}|\cI^3_{n-1}|&\leq \varepsilon_{n-1}\sum_{j=0}^{n-1}c_{*}^j \| BA^{n-2}  \partial_xU(t)\|_{L^2}
\|BA^j  \partial_xU(t)\|_{L^2} \\
&\leq  \frac{\varepsilon_{n-2}}{8}\| BA^{n-2}  \partial_xU(t)\|_{L^2}^2+\sum_{j=1}^{n-1} \frac{C\varepsilon_{n-1}^2}{\varepsilon_{n-2}}\| BA^j  \partial_xU(t)\|_{L^2}^2+\frac{C\varepsilon_{n-1}^2}{\varepsilon_{n-2}} \|\partial_x U_{2}(t)\|_{L^2}^2 .
\end{aligned} 
\end{equation}
\end{itemize}
In order to absorb the right-hand side terms of $\mathcal{I}_k^1$ and $\mathcal{I}_k^2$ by the left-hand side of \eqref{TimederivativeCorrector}, we take the constant  $\varepsilon_k$ small enough so that 
\begin{equation}\label{eq:e1}
C\varepsilon_1^2\leq \frac{\var_{0}^2}{8},\quad\quad C\varepsilon_k^2\leq \frac{\varepsilon_k\var_{0}}{8},\quad\quad k=1,2,...,n-1. \end{equation}
To handle the above estimates of  $\cI^3_k$ with $k=1,2,...,n-2$, one may let
\begin{equation}\label{eq:e11}
 C\varepsilon_k^2\leq \frac{1}{8}\varepsilon_{k-1}\varepsilon_{k+1},\quad\quad k=1,2,...,n-2\quad\text{if}~n\geq3.
 \end{equation} 
In addition, to handle the term $\cI^3_{n-1}$, we assume 
\begin{equation}\label{eq:e2}
C\varepsilon_{n-1}^2\leq \frac{1}{8}\varepsilon_{j}\varepsilon_{n-2},\qquad j=0,\cdots,n-1.
\end{equation}
Clearly, the inequality \eqref{ddtI} holds if we find $\varepsilon_1,\cdots,\varepsilon_{n-1}$ fulfilling 
\eqref{eq:e1} and \eqref{eq:e2}. 
As in \cite{BZ}, for $k=1,\cdots,n-2$, one can take  $\varepsilon_k=\varepsilon^{m_k}$ with some suitably small constant  $\varepsilon\leq \var_{0}$  and 
$m_1,\cdots,m_{n-1}$ satisfying for some $\delta>0$ (that can be taken arbitrarily small): 
\[m_{k}>1,\quad m_k\geq \frac{m_{k-1}+m_{k+1}}2+\delta \quad \text{and} \quad
m_{n-1}\geq\frac{m_{k}+m_{n-2}}2+ \delta.\]
This concludes the proof of Lemma \ref{lem:corrector}.
\end{proof}

\subsection{Decay of the $\dot{H}^1$ norm}
First, we fix suitably small $\varepsilon_{k}, k=1,2,...,n-1$, such that \eqref{ddtI} holds and
 \begin{align}
 & \mathcal{L}(t)\sim \|U(t)\|_{H^1}^2+\eta_{0} t\|\partial_x U(t)\|_{L^2}^2. \label{Lsim}
 \end{align}
Combining the Lyapunov inequality \eqref{L} and the estimate \eqref{ddtI} of  the corrector term, we obtain
\begin{equation}\label{mmm}
\begin{aligned}
    &\dfrac{d}{dt}\mathcal{L}(t)+\kappa \|U_2(t)\|_{L^2}^2+\kappa(1+ 2 \eta_{0} t) \|\partial_xU_2(t)\|_{L^2}^2+\frac{1}{2}\sum_{k=1}^{n-1} \varepsilon_k\|BA^k\partial_xU(t)\|_{L^2}^2\\
    &\leq \eta_{0}\|\partial_xU(t)\|_{L^2}^2+\var_{0}\|U_2(t)\|_{L^2}^2+\var_{0}\|\partial_xU_2(t)\|_{L^2}^2.
\end{aligned}
\end{equation}
In view of Proposition \ref{PropBZ}, it holds that
\[
\kappa\|\partial_xU_2(t)\|_{L^2}^2+\sum_{k=1}^{n-1} \varepsilon_k\|BA^k\partial_xU(t)\|_{L^2}^2\geq \frac{\var_{*}}{C_{K}}\|\partial_x U(t)\|_{L^2}^2
\]
with $\var_{*}:=\min\{\kappa, \var_{1}, \var_{1},...,\var_{n-1}\}$ and $C_{K}>0$ a constant depending only on $(A,B)$ and $n$.  Therefore, in order to ensure the coercivity of \eqref{mmm}, we adjust the coefficients appropriately as
\[
0<\eta_{0}<\frac{\var_{*}}{4C_{K}} ,\quad\quad 0<\var_{0}<\frac{\kappa}{2}
\]
such that 
\begin{align}
\dfrac{d}{dt}\mathcal{L}(t) +\frac{\kappa}{2} \| U_2(t)\|_{L^2}^2+\kappa (\frac{1}{2}+\eta_{0} t) \| \partial_x U_2(t)\|_{L^2}^2+\frac{\var_{*}}{4C_{K}}\|\partial_x U(t)\|_{L^2}^2\leq 0.\label{Lineq}
\end{align} 
Therefore, by \eqref{Lsim} and \eqref{Lineq}, we have
\begin{equation}
\begin{aligned}
&\|U(t)\|_{L^2}+(1+t)^{\frac12}\|\partial_x U(t)\|_{L^2}\leq C\|U_{0}\|_{H^1}.\label{uptoU1}
\end{aligned}
\end{equation}

\subsection{Improved decay for the damped component}

Taking the inner product of $\eqref{GE}_{2}$ with $U_2$ and using the property \eqref{Strong Dissipativity}, we get
\begin{equation}\label{ddtU22}
    \begin{aligned}
    &\frac{d}{dt}\|U_2(t)\|_{L^2}^2+2\kappa\|U_2(t)\|_{L^2}^2\leq C \|\partial_{x}U(t)\|_{L^2}\|U_2(t)\|_{L^2}.
    \end{aligned}
\end{equation}
Dividing the above inequality \eqref{ddtU22} by $\sqrt{\|U_2(t)\|_{L^2}^2+\var}$, employing Gr\"onwall's inequality and then letting $\var\rightarrow0$, we have
\begin{align}
&\|U_2(t)\|_{L^2}\leq e^{-\kappa t}\|U_{2,0}\|_{L^2}+C\int_{0}^{t}e^{-\kappa(t-\tau)} \|\partial_{x}U(\tau)\|_{L^2} \,d\tau.\label{this111}
\end{align} Together with the time-decay estimates \eqref{uptoU1} of $\partial_x U$, this leads to
\begin{equation}
\begin{aligned}
\|U_2(t)\|_{L^2}&\leq e^{-\kappa t}\|U_{2,0}\|_{L^2}+C\|U_0\|_{H^1}\int_{0}^{t}e^{-\kappa (t-\tau)}(1+\tau)^{-\frac{1}{2}}\,d\tau\\
&\leq C(1+t)^{-\frac{1}{2}}\|U_0\|_{H^1},\nonumber
\end{aligned}
\end{equation}
which concludes the proof of Theorem \ref{DecayThm1}. \qed

\begin{remark}
    For $t=1$, the computations  in this section also lead to \begin{align}\label{estimateThm1}
\|U(t)\|_{H^1}^2+\int_{0}^{t}\big(\|U_2(\tau)\|_{L^2}^2+\|\partial_x U(\tau)\|_{L^2}^2\big)\,d\tau \leq C\|U_0\|_{H^1}^2,
\end{align}
which will be useful in the following sections.
\end{remark}

\section{Faster time-decay rates: Proof of Theorem \ref{ThmDecayGeneral1}} \label{sec:4}

\subsection{Time-decay estimates for the parabolic system in weighted Sobolev spaces} \label{subsec:4}

This section aims to capture time-decay rates for $U_1$ in $L^2 $. Our method allows us to recover faster decay rates compared to the decay $(1+t)^{-\frac{1}{4}}$ obtained in \cite{BZ,BHN,Kawashimadoctoral,KYDecay}. In these references, inspired by the work of Matsumura and Nishida \cite{MatsumuraNishida}, the authors assumed $L^1$ regularity on the initial data to recover decay in low frequencies. Here, to avoid using the Fourier transform, we decompose \eqref{Systlin} into a pure parabolic system and a hyperbolic remainder part, and perform hypocoercivity estimates for space-weighted initial data.

First, we provide time-decay estimates for the equation associated with the large-time parabolic profile of the hyperbolic system:
\begin{equation}\label{parabolic2}
\left\{
\begin{aligned}
&\partial_{t}N-A_{1,2}D^{-1}A_{2,1} \partial^2_{x}N=0,\\
&N(x,0)=U_{1,0}(x).
\end{aligned}
\right.
\end{equation}
We recall a classical result ensuring that the operator $-A_{1,2}D^{-1}A_{2,1}$ is strongly elliptic.

\begin{lemma}[\cite{CBD3,Yong0,Yong}]
\label{l:parabolic}
Assume that $ A_{1,1}=0$.
Then, the following assertions are equivalent: \begin{itemize}
    \item $(A,B)$ satisfies the {\textrm(SK)} or the Kalman rank condition. \smallbreak
    \item $\hbox{The operator } \cA:=- A_{1,2}D^{-1} A_{2,1}\partial_{x}^2 \text{ is strongly elliptic.}$
    \end{itemize}
\end{lemma}
Relying on this result and an energy argument, we derive time-decay estimates for the solutions of \eqref{parabolic}.

\begin{lemma}\label{lemmaparabolic}
Assume $U_{1,0}\in L^2$ and let $N$ be the solution of \eqref{parabolic2}. Then, for all $k=1,2,...$, there exists a generic constant $C_{k}>0$ such that, for all $t>0$,
\begin{equation}\label{L2x}
\begin{aligned}
&\|\partial_x^{k}N(t)\|_{L^2}\leq C_{k} t^{-\frac{k}{2}} \|U_{1,0}\|_{L^2}.
\end{aligned}
\end{equation}
 Additionally, if $\| |x|^{\mu} U_{1,0}\|_{L^2}<\infty$ with $0< \mu\leq 1$, then
\begin{equation}\label{Hk}
\left\{
\begin{aligned}
&\|N(t)\|_{L^2}\leq C_k (1+t)^{-\frac{\mu}{2}} \|(1+|x|^{\mu})U_{1,0}\|_{L^2},\\
&\|\partial_x^{k}N(t)\|_{L^2}\leq C_{k} t^{-\frac{k}{2}-\frac{\mu}{2}} \|(1+|x|^{\mu})U_{1,0}\|_{L^2}.
\end{aligned}
\right.
\end{equation}

\end{lemma}

\begin{proof}
The Kalman rank condition \eqref{K} and Lemma \ref{l:parabolic} imply that there exists a constant $\kappa_0>0$ such that
\begin{align}
-\bigl(A_{1,2}D^{-1} A_{2,1}\partial_{x}^2 V, V\bigl) \geq \kappa_0 \|\partial_x V\|_{L^2}^2\quad\quad\text{for all} ~V\in \mathbb{R}^{n_1}.\label{disspara}
\end{align}
Hence, using \eqref{disspara}, standard energy estimates give
\begin{equation}\nonumber
\begin{aligned}
&\frac{d}{dt}\|N(t)\|_{L^2}^2+2\kappa_{0} \|\partial_x N(t)\|_{L^2}^2=0,\\
&\frac{d}{dt}\|\partial_x^k N(t)\|_{L^2}^2+2\kappa_{0} \|\partial^{k+1}_x N(t)\|_{L^2}^2=0,\quad\quad k=1,2,....
\end{aligned}
\end{equation}
Using time-weights, we get
\begin{equation}\nonumber
\begin{aligned}
&\frac{d}{dt}\Big(t^{k} \|\partial^{k}_x N(t)\|_{L^2}^2\Big)+ 2\kappa_{0} t^{k} \|\partial^{k+1}_x N(t)\|_{L^2}^2=k t^{k-1} \|\partial^k_x N(t)\|_{L^2}^2.
\end{aligned}
\end{equation}
Defining the Lyapunov functional
\begin{equation}\nonumber
\begin{aligned}
\mathcal{L}_{N,k}(t):=\|N(t)\|_{L^2}^2+\sum_{1\leq k'\leq k}\widetilde{\var}_{k'} t^{k'} \|\partial^{k'}_x N(t)\|_{L^2}^2,
\end{aligned}
\end{equation}
we have
\begin{equation}\label{LYN}
\begin{aligned}
&\frac{d}{dt}\mathcal{L}_{N,k}(t)+(2\kappa_{0}-\widetilde{\var}_{1}) \|\partial_x N(t)\|_{L^2}^2\\
&\quad+\sum_{1\leq k'\leq k}(2\kappa_0  \widetilde{\var}_{k'}- (k'+1) \widetilde{\var}_{k'+1}) t^{k'}\|\partial^{k'+1}_x N(t)\|_{L^2}^2=0.
\end{aligned}
\end{equation}
Choosing 
\[
\widetilde{\var}_{1}=\kappa_0,\quad\quad \widetilde{\var}_{k+1}=\frac{\kappa_0}{k+1}\widetilde{\var}_{k},
\]
and integrating \eqref{LYN} in time, we obtain
\begin{equation}\label{mu10}
\begin{aligned}
&\|N(t)\|_{L^2}^2+ \sum_{1\leq k'\leq k}t^{k'} \|\partial^{k'}_x N(t)\|_{L^2}^2\\
&\quad\quad+\int_{0}^{t}\Big(\|\partial_x N(\tau)\|_{L^2}^2+\sum_{1\leq k'\leq k}\tau^{k'} \|\partial^{k'+1}_x N(\tau)\|_{L^2}^2 \Big)\,d\tau\leq C_k \|U_{1,0}\|_{L^2}^2,
\end{aligned}
\end{equation}
which leads to \eqref{L2x}. Next, we prove \eqref{Hk}. We define
\begin{align}
S(x,t)=\int^{x}_{-\infty}N(y,t)\,dy \quad \text{and} \quad S_0(x)=\int^{x}_{-\infty}U_{1,0}(y)\,dy.\nonumber
\end{align}
Clearly, $S$ also satisfies the parabolic equation \eqref{parabolic} with the initial datum $S_0$. Similarly, we introduce the functional
\begin{equation}\label{matcalL*}
\begin{aligned}
\mathcal{L}^*_{N,k}(t):=\|S(t)\|_{L^2}^2+\sum_{1\leq k'\leq k}\var^*_{k'} t^{k'} \|\partial^{k'}_x S(t)\|_{L^2}^2,
\end{aligned}
\end{equation}
with some suitable small constants $\var^*_k$, $k=1,2,...$, such that
\begin{equation}\label{LS}
\begin{aligned}
&\frac{d}{dt}\mathcal{L}^*_{N,k}(t)+\kappa_0 \|\partial_x S(t)\|_{L^2}^2+\kappa_0\sum_{1\leq k'\leq k}  \var^*_{k'} t^{k'}\|\partial^{k'+1}_x S(t)\|_{L^2}^2\leq 0.
\end{aligned}
\end{equation}
Taking advantage of the Caffarelli-Kohn-Nirenberg inequality \eqref{CKNbest}, we have
\begin{equation}\label{S0}
\begin{aligned}
\|S_{0}\|_{L^2}\leq 2\| |x| U_{1,0}\|_{L^2}.
\end{aligned}
\end{equation}
Hence, integrating \eqref{LS} over $[0,t]$ and using \eqref{S0} and the fact that $\partial_x S=N$, we arrive at
\begin{equation}\label{mu1N}
\begin{aligned}
&\sum_{0\leq k'\leq k-1}t^{k'+1} \|\partial^{k'}_x N(t)\|_{L^2}^2+\sum_{0\leq k'\leq k-1}\int_{0}^{t}\tau^{k'+1} \|\partial^{k'+1}_x N(\tau)\|_{L^2}^2\,d\tau\\
&\hspace{8cm}\leq C_k \| |x| U_{1,0}\|_{L^2}^2.
\end{aligned}
\end{equation}
For $k'=0,1,2,...$, let the linear operator $T_k'$ be defined by $T_{k'}(U_{1,0})=\partial_x^{k'} N$. Then, the inequalities \eqref{mu10} and \eqref{mu1N} yield
\begin{equation}\nonumber
\begin{aligned}
\|T_{k'}(U_{1,0})\|_{L^2}\leq C_{k'}t^{-\frac{k'}{2}}\| U_{1,0}\|_{L^2(dx)}\:\:\text{and}\:\: \|T_{k'}(U_{1,0})\|_{L^2}\leq C_{k'}t^{-\frac{k'}{2}-\frac{1}{2}}\| U_{1,0}\|_{L^2(|x|^2 dx)}.
\end{aligned}
\end{equation}
Employing the Stein-Wassin interpolation theorem (see  \cite[Theorem 5.4.1]{berghinter} or \cite{steinwassin}), we have
\[
\|T_{k'}(U_{1,0})\|_{L^2}\leq C_{k'}t^{-\frac{k'}{2}-\frac{\mu}{2}}\| U_{1,0}\|_{L^2(|x|^{2\mu} dx)},\quad \quad 0<\mu<1,
\]
which completes the proof of $\eqref{Hk}$.

\end{proof}

\subsection{Time-decay estimates of the error}

In this subsection, we establish faster decay rates for the error between $U_1$ and $N$. Inspired by \cite{CBD2,CBD3}, we introduce the damped mode
\begin{align}
R=D^{-1}A_{2,1}U_1+D^{-1}A_{2,2} U_2+\int^x_{-\infty} U_2(y,t)\,dy,\label{R0}
\end{align}
which solves the damped system
\begin{equation}\label{eqR0}
\begin{aligned}
&\partial_{t}R+D R=D^{-1}A_{2,1} \partial_{t} U_1+D^{-1} A_{2,2} \partial_{t} U_{2}.
\end{aligned}
\end{equation}
Setting $U_2=\partial_{x}R-D^{-1}A_{2,1} \partial_{x} U_1-D^{-1}A_{2,2}\partial_{x} U_2 $ and since $A_{1,1}=0$, the equation $\eqref{GE}_{1}$ can be rewritten as
\begin{equation}
\partial_{t}U_1-A_{1,2}D^{-1}A_{2,1} \partial^2_{x} U_1=\partial_{x}^2(-A_{1,2} R+D^{-1}A_{2,2} U_2).\label{U1para}
\end{equation}
Therefore, as the right-hand side terms in \eqref{U1para} decay rapidly, one expects that the system \eqref{U1para} is asymptotically close to the linear parabolic system \eqref{parabolic}. We have the following lemma.

\begin{lemma}\label{lemmaerror}
Let the assumptions of Theorem \ref{DecayThm1} be satisfied, and $U=(U_1,U_2)$ be the solution to \eqref{Systlin} supplemented with the initial datum $U_0=(U_{0,1},U_{0,2})\in H^1$. In addition, assume $A_{1,1}=0$ and $|x| U_{0,2}\in L^2$. Then, for  all $t>0$ and any constant $0<\var \ll1/2$, we have
\begin{align}
\|R(t)\|_{L^2}&\leq C(1+t)^{-\frac{1}{2}}(\|U_0\|_{H^1}+\| |x| U_{0,2}\|_{L^2}), \label{Rdecay} \\
\|(U_1- N)(t)\|_{L^2}&\leq C(1+t)^{-\frac{1}{2}+\var} (\|U_0\|_{H^1}+\| |x| U_{0,2}\|_{L^2}),\label{asyU1}
\end{align}
where $R$ is defined by \eqref{R}, $N$ is the solution to \eqref{parabolic} subject to the initial datum $U_{0,1}$ and  $C>0$ is a generic constant.

Furthermore, under the additional assumption $ |x| U_{1,0}\in L^2$, we have
\begin{align}
&\|(U_1- N)(t)\|_{L^2}\leq C(1+t)^{-\frac{1}{2}} (\|U_0\|_{H^1}+\| |x| U_{0}\|_{L^2}).\label{asyU2}
\end{align}
\end{lemma}

\begin{proof}
To recover the time-decay estimates, we adapt the hypocoercive approach from Section \ref{sec:3}. The proof is divided into three steps.
\smallbreak
\noindent $\bullet$  \emph{Step 1: Decay estimates for $R$}. 
Performing $L^2$-energy estimate on \eqref{eqR} and using \eqref{Strong Dissipativity}, we obtain
\begin{equation}\nonumber
\begin{aligned}
&\frac{d}{dt}\|R(t)\|_{L^2}^2+2\kappa\|R(t)\|_{L^2}^2\\
&\leq 2\big(\|D^{-1}A_{2,1} \partial_{t} U_1(t)\|_{L^2}+\|D^{-1} A_{2,2} \partial_{t} U_{2}(t)\|_{L^2}\big) \|R(t)\|_{L^2},
\end{aligned}
\end{equation}
which implies
\begin{equation}\label{Rt1}
\begin{aligned}
&\|R(t)\|_{L^2}\leq e^{-\kappa t} \|R(0)\|_{L^2}+C\int_{0}^{t}e^{-\kappa (t-\tau)}\|(\partial_{t}U_{1},\partial_{t}U_{2})(\tau)\|_{L^2}\, d\tau.
\end{aligned}
\end{equation}
One deduces from the Caffarelli-Kohn-Nirenberg inequality \eqref{CKNbest} that
\begin{equation}
\begin{aligned}
\|R(0)\|_{L^2}&\leq C \|U_{0}\|_{L^2}+C\||x|U_{2,0}\|_{L^2}.\label{Rt2}
\end{aligned}
\end{equation}
And according to \eqref{GE} and the decay \eqref{e:decayThm1} of $\partial_x U$ and $U_2$ at hand, we have
\begin{equation}\label{Rt3}
\begin{aligned}
\|(\partial_{t}U_{1},\partial_{t}U_{2})(t)\|_{L^2}\leq C\|(\partial_x U, U_2)(t)\|_{L^2}\leq C(1+t)^{-\frac{1}{2}}\|U_0\|_{H^1}.
\end{aligned}
\end{equation}
Substituting \eqref{Rt2} and \eqref{Rt3} into \eqref{Rt1} yields \eqref{Rdecay}.

\smallbreak
\noindent $\bullet$ \emph{Step 2: Decay estimates for the error term}. In this step, we establish the decay estimates \eqref{asyU1} of the error uknown $\widetilde{U}_{1}:=U_1- N$. To this matter, we observe that $\widetilde{U}_{1}$ satisfies
\begin{equation}
\partial_{t}\widetilde{U}_{1}-A_{1,2}D^{-1}A_{2,1} \partial^2_{x} \widetilde{U}_{1}=\partial_x^2 \widetilde{F},\quad\quad \widetilde{U}_{1}(0,x)=0,\label{erreq}
\end{equation}
with $\widetilde{F}:=-A_{1,2} R+D^{-1}A_{2,2} U_2$. 
Defining
\begin{align}
\widetilde{Q}(x,t):=\int^{x}_{-\infty}\int^{y}_{-\infty}\widetilde{U}_{1}(z,t)\,dzdy,\label{wideQ}
\end{align}
we deduce from \eqref{erreq} that
\begin{equation}
\partial_{t}\widetilde{Q}-A_{1,2}D^{-1}A_{2,1} \partial^2_{x} \widetilde{Q}=\widetilde{F}\quad \text{and}\quad \widetilde{Q}(0,x)=0.\label{wideN}
\end{equation}
Hence, Duhamel's principle applied to \eqref{wideN} yields
\begin{equation}\label{duhamel}
\begin{aligned}
&\widetilde{Q}(x,t)=\int_{0}^{t} \mathcal{G}^sF(x,t)\,ds,
\end{aligned}
\end{equation}
where $\mathcal{G}^sF(x,t)$ is the solution to 
\begin{equation}\label{duhamel1}
\begin{aligned}
&\partial_{t}(\mathcal{G}^sF)-A_{1,2}D^{-1}A_{2,1} \partial^2_{x}(\mathcal{G}^sF)=0,\quad t>s,\quad \mathcal{G}^sF(x,s)=F(x,s).
\end{aligned}
\end{equation}
Applying Lemma \ref{lemmaparabolic} with $k=2$ to  \eqref{duhamel1} gives rise to
\begin{equation}\label{Gs1}
\begin{aligned}
\|\partial_x^2 \mathcal{G}^sF(t)\|_{L^2}\leq (t-s)^{-1}\|F(s)\|_{L^2}.
\end{aligned}
\end{equation}
Then, differentiating \eqref{duhamel1} with respect to $x$ and taking advantage of Lemma \ref{lemmaparabolic} with $k=1$ yield
\begin{equation}\label{Gs2}
\begin{aligned}
\|\partial_x^2 \mathcal{G}^s F(t)\|_{L^2}\leq (t-s)^{-\frac{1}{2}}\|\partial_x F(s)\|_{L^2}.
\end{aligned}
\end{equation}
Using an interpolation argument between \eqref{Gs1} and \eqref{Gs2}, for any $0<\var<1/2$, we have
\begin{equation}\label{Gs}
\begin{aligned}
\|\partial_x^2 \mathcal{G}^sF(t)\|_{L^2}&\leq \Big( (t-s)^{-1}\|F(s)\|_{L^2}\Big)^{1-2\var} \Big( (t-s)^{-\frac{1}{2}}\|\partial_x F(s)\|_{L^2} \Big)^{2\var}\\
&\leq (t-s)^{-(1-\var)}(\|F(s)\|_{L^2}+\|\partial_x F(s)\|_{L^2}).
\end{aligned}
\end{equation}
Gathering the estimates \eqref{e:decayThm1} and \eqref{Rdecay}, we derive decay for $F$ as follows
\begin{equation}\label{FH1decay}
\begin{aligned}
\|F(s)\|_{L^2}+\|\partial_x F(s)\|_{L^2}&\leq C\|R(s)\|_{L^2}+C\|U_2(s)\|_{L^2}+C\|\partial_x U(s)\|_{L^2}\\
&\leq C s^{-\frac{1}{2}}  (\|U_{0}\|_{H^1}+C\||x|U_{2,0}\|_{L^2}).
\end{aligned}
\end{equation}
It thus follows from \eqref{duhamel}, \eqref{Gs} and \eqref{FH1decay} that
\begin{equation}\label{U1homo}
\begin{aligned}
\|\widetilde{U}_{1}(t)\|_{L^2}& =\|\partial_x^2 \widetilde{Q}(t)\|_{L^2}\\
&\leq \int_{0}^{t} \|\partial_x^2 \mathcal{G}^sF(t)\|_{L^2}\,ds\\
&\leq C\int_{0}^{t} (t-s)^{-(1-\var)} s^{-\frac{1}{2}}\,ds~ (\|U_{0}\|_{H^1}+\|(1+|x|)U_{2,0}\|_{L^2})\\
&\leq CB(\var, \frac{1}{2}) t^{-\frac{1}{2}+\var} (\|U_{0}\|_{H^1}+\||x|U_{2,0}\|_{L^2}),
\end{aligned}
\end{equation}
where $B(s_1,s_2)$ denote the beta function. This leads to \eqref{asyU1}.

\smallbreak
\noindent $\bullet$ \emph{Step 3: Improved decay estimates for the error part}. Under the additional condition $|x|U_{1,0}\in L^2$, we now justify the improved decay estimates \eqref{asyU2} for $\widetilde{U}_1$. Taking the inner product of \eqref{erreq} with $-\widetilde{Q}$ given by \eqref{wideQ} and applying  \eqref{disspara}, we obtain
\begin{equation}\label{QL2}
\begin{aligned}
&\frac{d}{dt}\|\partial_x\widetilde{Q}(t)\|_{L^2}^2+2\kappa_0\| \widetilde{U}_{1}(t)\|_{L^2}^2=-2\bigl(\partial_x \widetilde{F}, \widetilde{U}_{1}\bigl)_{L^2}.
\end{aligned}
\end{equation}
On the other hand, $(\widetilde{U}_{1},U_2)$ solves the partially dissipative hyperbolic system 
\begin{equation}\label{wideU1U2}
\left\{
\begin{aligned}
&\partial_{t}\widetilde{U}_{1}+A_{1,2} \partial_x U_2=-A_{1,2}D^{-1}A_{2,1} \partial^2_{x} N,\\
&\partial_{t} U_2+A_{2,1}\partial_x \widetilde{U}_{1}+A_{2,2}\partial_x U_2+D U_2=-A_{2,1}\partial_x N,\\
&(\widetilde{U}_{1}, U_2)(0,x)=(0,U_{2,0})(x).
\end{aligned}
\right.
\end{equation}
Performing time-weighted $L^2$-energy estimates, we obtain
\begin{equation}
\begin{aligned}
& \frac{d}{dt}\Big(t\|(\widetilde{U}_{1},U_2)(t)\|_{L^2}^2\Big)+2t\kappa_0 \|U_2(t)\|_{L^2}^2\\
 &=\|(\widetilde{U}_{1},U_2)(t)\|_{L^2}^2-2t\bigl(A_{1,2}D^{-1}A_{2,1} \partial^2_{x} N ,\widetilde{U}_{1}\bigl)_{L^2}-2t\bigl(A_{2,1}\partial_x N, U_2\bigl)_{L^2}.\label{LywideU1U2}
  \end{aligned}
  \end{equation}
We define the Lyapunov functional 
\[
\widetilde{\mathcal{L}}(t):=\mathcal{L}(t)+\mathcal{L}^*_{N,2}(t)+\eta^*_1 \|\partial_x\widetilde{Q}(t)\|_{L^2}^2+\eta^*_2 t \|(\widetilde{U}_{1},U_2)(t)\|_{L^2}^2,
\]
where $\mathcal{L}(t)$ and $\mathcal{L}^*_{N,2}(t)$ are given by \eqref{def:GeneralLyaNoF} and \eqref{matcalL*}, respectively. In view of \eqref{Lineq}, \eqref{LS}, \eqref{QL2} and \eqref{LywideU1U2}, we get
\begin{equation}\label{widematL}
\begin{aligned}
&\frac{d}{dt} \widetilde{\mathcal{L}}(t)+(\frac{\kappa}{2}-\eta_2) \| U_2(t)\|_{L^2}^2+\kappa (\frac{1}{2}+\eta_{0} t) \| \partial_x U_2(t)\|_{L^2}^2\\
&\quad+\frac{\var_{*}}{4C_{K}}\|\partial_x U(t)\|_{L^2}^2+\kappa_0\sum_{0\leq k'\leq 3}  t^{k'}\|\partial^{k'}_x N(t)\|_{L^2}^2\\
&\quad+(2\kappa_0\eta^*_1-\eta_2^*) \| \widetilde{U}_{1}(t)\|_{L^2}^2+2 \kappa_0 \eta^*_2 t\|U_2(t)\|_{L^2}^2\\
&\quad\leq -2\eta^*_1 \bigl(\partial_x \widetilde{F}, \widetilde{U}_{1}\bigl)_{L^2}-2\eta^*_2 t\bigl(A_{1,2}D^{-1}A_{2,1} \partial^2_{x} N ,\widetilde{U}_{1}\bigl)_{L^2}-2\eta^*_2  t\bigl(A_{2,1}\partial_x N, U_2\bigl)_{L^2}.
\end{aligned}
\end{equation}
The right-hand side terms of \eqref{widematL} are analyzed as follows. First, one has
\begin{equation}\nonumber
\begin{aligned}
-2\eta^*_1 \bigl(\partial_x \widetilde{F}, \widetilde{U}_{1}\bigl)_{L^2}&\leq \kappa_0 \eta^*_1 \| \widetilde{U}_{1}(t)\|_{L^2}^2+C \eta^*_1 ( \|U_2(t)\|_{L^2}^2+\|\partial_x U(t)\|_{L^2}^2).
\end{aligned}
\end{equation}
Similarly, we have
\begin{equation}\nonumber
\begin{aligned}
-2\eta^*_2 t\bigl(A_{1,2}D^{-1}A_{2,1} \partial^2_{x} N ,\widetilde{U}_{1}\bigl)_{L^2}&\leq C\eta^*_2  t^2 \|\partial^2_{x} N(t)\|_{L^2}^2+C\eta^*_2 \|\widetilde{U}_{1}(t)\|_{L^2}^2,
\end{aligned}
\end{equation}
and
\begin{equation}\nonumber
\begin{aligned}
-2\eta^*_2  t\bigl(A_{2,1}\partial_x N, U_2\bigl)_{L^2}&\leq \eta^*_2 \kappa_0 t \|U_2(t)\|_{L^2}^2+C\eta^*_2 t \|\partial_x N(t)\|_{L^2}^2.
\end{aligned}
\end{equation}
Substituting the above three estimates into \eqref{widematL} and choosing
\[
\eta^*_1=\min\{\frac{\kappa}{4}, \frac{\var_*}{8C_K}\},\quad\quad \eta^*_2=\frac{\kappa_0\eta^*_1}{2(C+1)},
\]
yield
\begin{equation}\nonumber
\begin{aligned}
&\frac{d}{dt} \widetilde{\mathcal{L}}(t)+\frac{\kappa}{4} \| U_2(t)\|_{L^2}^2+\kappa (\frac{1}{4}+\eta_{0} t) \| \partial_x U_2(t)\|_{L^2}^2+\frac{\var_{*}}{8C_{K}}\|\partial_x U(t)\|_{L^2}^2\\
&\quad+\frac{\kappa_0}{2}\sum_{0\leq k'\leq 2}  t^{k'}\|\partial^{k'}_x N(t)\|_{L^2}^2+\kappa_0\eta^*_1 \| \widetilde{U}_{1}(t)\|_{L^2}^2+\kappa_0 \eta^*_2 t\|U_2(t)\|_{L^2}^2\leq 0.
\end{aligned}
\end{equation}
Integrating in time leads to
\begin{equation}\nonumber
\begin{aligned}
&t\|\widetilde{U}_{1}(t)\|_{L^2}^2+\int_{0}^{t}\Big(\|\widetilde{U}_{1}(\tau)\|_{L^2}^2+\tau \|U_2(\tau)\|_{L^2}^2\Big)\,d\tau\leq C(\|U_0\|_{H^1}^2+\| |x| U_{0}\|_{L^2}^2),
\end{aligned}
\end{equation}
which concludes the proof of Lemma \ref{lemmaerror}.

\end{proof}

\noindent
\textbf{Proof of Theorem \ref{ThmDecayGeneral1}:} The $L^2$-decay estimate \eqref{U1NL2} of the error $U_1-N$ follows directly from Lemma \ref{lemmaerror}. Then, we impose the condition $|x|^{\mu}U_{0,1}\in L^2$ with $0<\mu\leq 1$. In the case  $0<\mu<1$, the estimates \eqref{U1NL2} for $0<\var<\frac{1}{2}(1-\mu)$ together with \eqref{Hk} guarantees that
\begin{equation}\nonumber
\begin{aligned}
\|U_1(t)\|_{L^2}&\leq \|N(t)\|_{L^2}+\|(U_1-N)(t)\|_{L^2}\\
&\leq C(1+t)^{-\frac{\mu}{2}} \|(1+|x|^{\mu})U_{1,0}\|_{L^2}+C(1+t)^{-\frac{1}{2}+\var}(\|U_0\|_{H^1}+\| |x| U_{2,0}\|_{L^2})\\
&\leq C  (1+t)^{-\frac{\mu}{2}} X_0.
\end{aligned}
\end{equation}
In the case $\mu=1$, one deduces from \eqref{Hk} and \eqref{asyU2} that
\begin{equation}\nonumber
\begin{aligned}
\|U_1(t)\|_{L^2}&\leq \|N(t)\|_{L^2}+\|(U_1-N)(t)\|_{L^2}\\
&\leq C(1+t)^{-\frac{1}{2}} \|(1+|x|)U_{1,0}\|_{L^2}+C(1+t)^{-\frac{1}{2}}(\|U_0\|_{H^1}+\| |x| U_{0}\|_{L^2})\\
&\leq C  (1+t)^{-\frac{1}{2}} X_0.
\end{aligned}
\end{equation}
Combining the above decay estimates for $U_1$ and the estimates derived for $U_2$ in Theorem \ref{DecayThm1}, we get $L^2$-decay estimates for $U$ in \eqref{decay111}.

We now establish the faster decay rates of $\partial_x U$ and $U_2$ in \eqref{decay111}. The Lyapunov inequality \eqref{Lineq} can be rewritten as
\begin{equation}\label{llddd}
    \begin{aligned}
    &\frac{d}{dt}\Big(\mathcal{L}_{*}(t)+\eta_{0}t\|\partial_x U(t)\|_{L^2}^2\Big)+\frac{\kappa}{2} \| U_2(t)\|_{L^2}^2\\
    &\quad\quad\quad\quad+\kappa (\frac{1}{2}+\eta_{0} t) \| \partial_x U_2(t)\|_{L^2}^2+\frac{\var_{*}}{4C_{K}}\|\partial_x U(t)\|_{L^2}^2\leq 0,
    \end{aligned}
\end{equation}
with $\mathcal{L}_{*}(t):=\|U(t)\|_{H^1}^2+\mathcal{I}(t)\sim \|U(t)\|_{H^1}^2.$
Choosing the constant $\eta_{0}$ sufficiently small, applying Lemma \ref{decayineq} to the differential inequality \eqref{llddd} and noticing that $\partial_x U$ is uniformly bounded in $L^2$, we conclude
\begin{equation}\nonumber
\begin{aligned}
&\|\partial_x U(t)\|_{L^2}\leq C (1+t)^{-\frac{\mu}{2}-\frac{1}{2}}X_{0}.
\end{aligned}
\end{equation}
Then, using \eqref{this111} together with the decay for $U_{1}$ and $\partial_xU$ at hand, leads to
\begin{equation}\nonumber
\begin{aligned}
    \|U_2(t)\|_{L^2}&\leq e^{-\kappa t}\|U_{2,0}\|_{L^2}+C X_{0}\int_{0}^{t}e^{-\kappa (t-\tau)}(1+\tau)^{-\frac{\mu}{2}-\frac{1}{2}}\,d\tau \\
    &\leq C (1+t)^{-\frac{\mu}{2}-\frac{1}{2}}X_{0},\nonumber
\end{aligned}
\end{equation}
which concludes the proof of Theorem \ref{ThmDecayGeneral1}.

\section{Wave formulation method: Proof of Theorem  \ref{ThmDecayGeneral2}} \label{sec:wavemethod}

In this section we prove Theorem  \ref{ThmDecayGeneral2}. We introduce the unknown
\begin{align}
W(x,t)=\int^{x}_{-\infty}U_1(y,t)\,dy,\label{W}
\end{align}
which satisfies the following damped wave formulation
\begin{align}
&\partial_{t}^2 W-A_{1,2} A_{2,1} \partial_{x}^2 W +A_{1,2} A_{2,2}A_{1,2}^{-1} \partial_t \partial_x W+ A_{1,2} D A_{1,2}^{-1}\partial_t W =0.\label{dampedwave1}
\end{align}
We note that $A_{1,2}^{-1}$ is well-defined as, if $A_{1,2}$ is not invertible, then $A_{1,2}$ is not a $n_1\times n_1$ matrix of full rank, which, together with $A_{1,1}=0$, contradicts the Kalman rank condition \eqref{K}. To obtain \eqref{dampedwave1}, we integrated $\eqref{GE}_{1}$ over $(-\infty, x)$ which 
\begin{align}
\partial_t W+A_{1,2} U_2=0,\label{dampedwave111}
\end{align}
Then, differentiating in time the above system and making use of $\eqref{GE}_{2}$, we get
\begin{align}
\partial^2_{tt} W-A_{1,2} A_{2,1}\partial_x U_1-A_{1,2} A_{2,2}\partial_x U_2-A_{1,2} D U_2=0.\label{dampedwave112}
\end{align}
Combining \eqref{W}, \eqref{dampedwave111} and \eqref{dampedwave112} together, we have \eqref{dampedwave1}.

Since $\partial_x W=U_1$ and $|\partial_t W|\sim |U_2|$,  we can derive $L^2$ time-decay estimates for $U$ once we establish decay estimates of the wave energy $\|(\partial_t W,\partial_x W)(t)\|_{L^2}^2$. In the following lemma we establish time-space weighted energy estimates for $W$.
\begin{lemma}\label{lemmaW}
Let $W$ be defined by \eqref{W}. Then under the assumptions of Theorem  \ref{ThmDecayGeneral2}, for all $t>0$, we have
\begin{equation}\label{estimateW}
    \begin{aligned}
        &\int_{\mathbb{R}} (1+t+|x|)^{2\mu-1}(|\partial_t W|^2+|\partial_x W|^2)\,dx\\
        &\quad\quad+\int_{0}^{t}\int_{\mathbb{R}}\Big( (1+t+|x|)^{2\mu-1} |\partial_t W |^2+(1+t+|x|)^{2\mu-2} |\partial_x W |^2\Big)\,dxd\tau \leq CY_{0}^2,
    \end{aligned}
\end{equation}
with $Y_{0}:=\|U_{0}\|_{H^1}+\||x|^{\mu} U_{1,0}\|_{L^2} + \||x|^{\mu-\frac{1}{2}}U_{2,0}\|_{L^2}$.
\end{lemma}

\begin{proof}

The proof is split into the cases $\mu=1$ and $1/2<\mu<1$.
\smallbreak
\noindent $\bullet$  \textbf{Case 1:} $\mu=1.$ Taking the $L^2$ inner product of \eqref{dampedwave1} with $(1+t+|x|)\partial_t W$, we get
\begin{equation}\label{pppp}
\begin{aligned}
&\frac{d}{dt}\int_{\mathbb{R}} \frac{1}{2} (1+t+|x|) (|\partial_t W|^2 + A_{1,2} A_{2,1} \partial_x W \cdot \partial_x W)\,dx\\
&\quad+\int_{\mathbb{R}} \Big( (1+t+|x|) A_{1,2} D A_{1,2}^{-1} \partial_t W \cdot\partial_t W -\frac{1}{2}A_{1,2} A_{2,1} \partial_x W  \cdot\partial_x W\Big)\,dx   \\
&=\int_{\mathbb{R}} \Big( \frac{1}{2}|\partial_t W|^2- \frac{x}{|x|} A_{1,2} A_{2,1}\partial_x W  \cdot\partial_t W+\frac{1}{2}\frac{x}{|x|} A_{1,2} A_{2,2}A_{1,2}^{-1} \partial_t W  \cdot\partial_t W \Big) \,dx,
\end{aligned}
\end{equation}
where we used 
\begin{equation}\nonumber
\begin{aligned}
&\int_{\mathbb{R}} (1+t+|x|) (\partial_t^2 W \cdot\partial_t W- A_{1,2} A_{2,1} \partial_{x}^2 W  \cdot\partial_t W) \,dx\\
&\quad=\frac{d}{dt}\int_{\mathbb{R}}  \frac{1}{2}(1+t+|x|) (|\partial_t W|^2+A_{1,2} A_{2,1} \partial_x W \cdot \partial_x W) \,dx\\
&\quad\quad-\frac{1}{2}\int_{\mathbb{R}} (|\partial_t W|^2+A_{1,2} A_{2,1} \partial_x W \cdot \partial_x W) \,dx+\int_{\mathbb{R}} \frac{x}{|x|} A_{1,2} A_{2,1}\partial_x W  \cdot\partial_t W\,dx,
\end{aligned}
\end{equation}
and
\begin{equation}\nonumber
\begin{aligned}
\int_{\mathbb{R}}  (1+t+|x|)A_{1,2} A_{2,2}A_{1,2}^{-1} \partial_t \partial_x W \cdot\partial_t W \,dx=-\frac{1}{2}\int_{\mathbb{R}}\frac{x}{|x|} A_{1,2} A_{2,2}A_{1,2}^{-1} \partial_t W \cdot \partial_t W \,dx.
\end{aligned}
\end{equation}
In addition, taking the inner product of $\eqref{dampedwave1}$ with $W$, we obtain
\begin{equation}\label{pppp1}
\begin{aligned}
&\frac{d}{dt}\int_{\mathbb{R}} \Big( W \cdot\partial_t W+\frac{1}{2}  A_{1,2} D A_{1,2}^{-1} |W|^2 \Big) \,dx+\int_{\mathbb{R}}A_{1,2} A_{2,1}\partial_x W \cdot\partial_x W \,dx\\
&=\int_{\mathbb{R}}\Big(|\partial_t W|^2+A_{1,2} A_{2,2}A_{1,2}^{-1}\partial_t W \cdot\partial_x W\Big)\,dx. 
\end{aligned}
\end{equation}
It follows from \eqref{pppp} and \eqref{pppp1} that
\begin{equation}\label{pppp12}
    \begin{aligned}
    &\frac{d}{dt}\mathcal{W}(t)+\mathcal{H}(t)  \\
&=\int_{\mathbb{R}} \Big(\frac{3}{2} |\partial_t W|^2- \frac{x}{|x|} A_{1,2} A_{2,1}\partial_x W\cdot  \partial_t W\\
&\quad+\frac{1}{2}\frac{x}{|x|} A_{1,2} A_{2,2}A_{1,2}^{-1} \partial_t W \cdot \partial_t W+A_{1,2} A_{2,2}A_{1,2}^{-1}\partial_t W \cdot\partial_x W\Big)\,dx,
    \end{aligned}
\end{equation}
where $ \mathcal{W}(t)$ and $\mathcal{H}(t)$ are defined by
\begin{equation}\nonumber
    \begin{aligned}
        \mathcal{W}(t):&=\int_{\mathbb{R}}   \frac{1}{2} (1+t+|x|) (|\partial_t W|^2 + A_{1,2} A_{2,1} \partial_x W  \cdot\partial_x W)\, dx\\
        &\quad+ \int_{\mathbb{R}} \Big( W\cdot \partial_t W+\frac{1}{2}  A_{1,2} D A_{1,2}^{-1} |W|^2 \Big) \,dx,\\
        \mathcal{H}(t):&=\int_{\mathbb{R}} (1+t+|x|) A_{1,2} D A_{1,2}^{-1} \partial_t W\cdot \partial_t W \,dx+\int_{\mathbb{R}} \frac{1}{2} A_{1,2} A_{2,1} \partial_x W \cdot \partial_x W \, dx .
    \end{aligned}
\end{equation}
Notice that  $A_{1,2} D A_{1,2}^{-1}$ satisfies \eqref{Strong Dissipativity} since the eigenvalues of $D$  and $A_{1,2}D A_{1,2}^{-1}$ are the same. By the strong dissipation conditions \eqref{Strong Dissipativity} and the positive definiteness of $A_{1,2}A_{2,1}$, we have
\begin{equation}\nonumber
    \begin{aligned}
   &\mathcal{W}(t)\geq \int_{\mathbb{R}} \Big( (1+t+|x|) (\frac{1}{2}|\partial_t W|^2 +\kappa_{1} |\partial_x W |^2)+\kappa |W|^2-C|\partial_t W|^2\Big)\,dx,
    \end{aligned}
\end{equation}
    and
    \begin{equation}\nonumber
    \begin{aligned}
    &\mathcal{H}(t) \geq \int_{\mathbb{R}} \Big(\kappa (1+t+|x|)  |\partial_t W|^2 +\frac{\kappa_{1}}{2} |\partial_x W|^2\Big) \,dx.
    \end{aligned}
\end{equation}
Since $|\partial_t W|\sim |U_2|$, due to \eqref{dampedwave111} and the fact that $A_{1,2}$ is invertible, one can estimate the right-hand side of \eqref{pppp12} as follows:
\begin{equation}\nonumber
    \begin{aligned}
&\int_{\mathbb{R}} \Big(\frac{3}{2} |\partial_t W|^2- \frac{x}{|x|} A_{1,2} A_{2,1}\partial_x W \cdot \partial_t W\\
&\quad+\frac{1}{2}\frac{x}{|x|} A_{1,2} A_{2,2}A_{1,2}^{-1} \partial_t W\cdot  \partial_t W+A_{1,2} A_{2,2}A_{1,2}^{-1}\partial_t W \cdot\partial_x W\Big)\,dx\\
&\quad\leq \frac{\kappa_{1}}{4} \int_{\mathbb{R}}|\partial_x W|^2 \, dx+ C\int_{\mathbb{R}} |U_2|^2\,dx.
        \end{aligned}
\end{equation}
The above estimates give rise to
\begin{equation}\label{arrr}
    \begin{aligned}
    &\int_{\mathbb{R}} \Big( (1+t+|x|) (\frac{1}{2}|\partial_t W|^2 +\kappa_{1} |\partial_x W|^2)+\kappa |W|^2\Big)\,dx\\
    &\quad\quad+\int_{0}^{t}\int_{\mathbb{R}} \Big(\kappa (1+\tau+|x|)  |\partial_t W|^2 +\frac{\kappa_{1}}{4} |\partial_x W|^2\Big)\, dxd\tau\\
    &\quad \leq \mathcal{W}(0)+C\int_{\mathbb{R}}|U_{2}|^2\,dx+C\int_{0}^{t}\int_{\mathbb{R}} |U_{2}|^2\, dxd\tau.
    \end{aligned}
\end{equation}
 Under the assumptions \eqref{a21} and $\partial_{t}U_{1} |_{t=0}=-A_{1,2}\partial_x U_{2,0}$,
 one deduces from the Caffarelli-Kohn-Nirenberg inequality \eqref{CKNbest} that
 \begin{equation}\label{arrr1}
\begin{aligned}
\mathcal{W}(0)&\lesssim \int_{\mathbb{R}} \Big( (1+|x|)|U_{0}|^2 + \Big|\int^{x}_{-\infty}U_{1,0}(y)\,dy\Big|^2\Big)\,dx\lesssim Y_{0}^2.
\end{aligned}
\end{equation}
 By $|\partial_t W|\sim |U_2|$ and \eqref{estimateThm1}, it holds that
\begin{align}
    &\int_{\mathbb{R}}|\partial_t W|^2\,dx+\int_{0}^{t}\int_{\mathbb{R}} |\partial_t W|^2\, dxd\tau\\
    &\quad\leq \int_{\mathbb{R}}|U_2|^2\,dx+\int_{0}^{t}\int_{\mathbb{R}} |U_2|^2 \,dxd\tau\lesssim Y_{0}^2.\label{arrr2}
\end{align}
Inserting \eqref{arrr1} and \eqref{arrr2} into \eqref{arrr}, we get \eqref{estimateW} with $\mu=1$.

\smallbreak
\noindent $\bullet$  {\textbf{Case 2:}} $1/2<\mu<1.$ In this case, let $\varphi: \mathbb{R}^{+}\rightarrow \mathbb{R}^{+}$ be a weight function to be determined later. Similarly to the case $\mu=1$, one gets
\begin{equation}\nonumber
\begin{aligned}
& \frac{d}{dt}\int_{\mathbb{R}}\frac{1}{2} \varphi(t+|x|) (|\partial_t W|^2 + A_{1,2} A_{2,1} \partial_x W  \cdot\partial_x W)\,dx\\
&\quad+\int_{\mathbb{R}} \Big( \varphi(t+|x|) A_{1,2} D A_{1,2}^{-1} \partial_t W \cdot\partial_t W -\frac{1}{2} \varphi'(t+|x|) A_{1,2} A_{2,1} \partial_x W  \cdot\partial_x W\Big)\,dx   \\
&=\int_{\mathbb{R}} \varphi'(t+|x|)\Big( \frac{1}{2} |\partial_t W|^2- \frac{x}{|x|} A_{1,2} A_{2,1}\partial_x W \cdot \partial_t W+\frac{1}{2}\frac{x}{|x|} A_{1,2} A_{2,2}A_{1,2}^{-1} \partial_t W  \cdot\partial_t W \Big) \,dx.
\end{aligned}
\end{equation}
After taking the inner product of $\eqref{dampedwave1}$ with $\varphi'(t+|x|) W$, we verify that
\begin{equation}\nonumber
\begin{aligned}
&\frac{d}{dt}\int_{\mathbb{R}} \Big( \varphi'(t+|x|) \partial_t W  W-\varphi''(t+|x|) |W|^2 +\frac{1}{2}\varphi'(t+|x|)  A_{1,2} D A_{1,2}^{-1} W\cdot W \Big) \,dx\\
&\quad+\int_{\mathbb{R}}\Big( \varphi'(t+|x|) A_{1,2} A_{2,1} \partial_x W \cdot\partial_x W +\frac{1}{2}\varphi'''(t+|x|) | W|^2  \Big)\,dx\\
&\quad-\int_{\mathbb{R}} \Big( \frac{1}{2}\varphi'''(t+|x|) +\varphi''(t+|x|)\delta_{0}(x)\Big)A_{1,2} A_{2,1}  W \cdot W\, dx\\
&=\int_{\mathbb{R}}\varphi'(t+|x|) |\partial_t W|^2\,dx. 
\end{aligned}
\end{equation}
Here we have used
\begin{equation}\nonumber
\begin{aligned}
&\int_{\mathbb{R}} \varphi'(t+|x|)  \partial_{t}^2 W \cdot W \,dx\\
&\quad=\frac{d}{dt}\int_{\mathbb{R}} \Big( \varphi'(t+|x|)  \partial_t W\cdot W -\frac{1}{2}\varphi''(t+|x|) |W|^2\Big) \,dx\\
&\quad\quad+\int_{\mathbb{R}} \Big(\frac{1}{2}\varphi'''(t+|x|) |W|^2 -\varphi'(t+|x|)|\partial_t W|^2\Big)\,dx,\\
&\int_{\mathbb{R}} \varphi'(t+|x|) A_{1,2} D A_{1,2}^{-1} \partial_t W \cdot W \,dx\\
&\quad=\frac{d}{dt}\int_{\mathbb{R}} \frac{1}{2}\varphi'(t+|x|)  A_{1,2} D A_{1,2}^{-1} W \cdot W \,dx-\int_{\mathbb{R}}\varphi''(t+|x|) A_{1,2} D A_{1,2}^{-1} W \cdot W  \,dx,
\end{aligned}
\end{equation}
and
\begin{equation}\nonumber
\begin{aligned}
&-\int_{\mathbb{R}} \varphi'(t+|x|) A_{1,2} A_{2,1} \partial_{x}^2 W \cdot W\, dx\\
&\quad =-\int_{\mathbb{R}} \Big(\frac{1}{2} \varphi'''(t+|x|)+\varphi''(t+|x|)\delta_{0}(x)\Big)A_{1,2} A_{2,1}  W  \cdot W\,dx\\
&\quad\quad-\int_{\mathbb{R}} \varphi'(t+|x|) A_{1,2} A_{2,1} \partial_x W \cdot \partial_x W \,dx,
\end{aligned}
\end{equation}
where $\delta_{0}(x)$ denotes the  Dirac function at $0$. Gathering the previous estimates we get the following inequality
\begin{equation}\label{ppppp2}
\begin{aligned}
\frac{d}{dt}\mathcal{W}_{\mu}(t)+\mathcal{H}_{\mu}(t)&=\int_{\mathbb{R}} \varphi'(t+|x|)\Big( \frac{3}{2} |\partial_t W|^2-\frac{x}{|x|} A_{1,2} A_{2,1}\partial_x W  \cdot\partial_t W\\&\quad\quad\quad\quad\quad\quad\quad\quad+\frac{1}{2}\frac{x}{|x|} A_{1,2} A_{2,2}A_{1,2}^{-1} \partial_t W\cdot  \partial_t W \Big)\, dx,
\end{aligned}
\end{equation}
with
\begin{equation}\nonumber
\begin{aligned}
\mathcal{W}_{\mu}(t):&=\int_{\mathbb{R}} \frac{1}{2} \varphi(t+|x|) (|\partial_t W|^2 + A_{1,2} A_{2,1} \partial_x W \cdot \partial_x W)\,dx\\
&\quad+\int_{\mathbb{R}} \Big( \varphi'(t+|x|) \partial_t W \cdot W-\frac{1}{2}\varphi''(t+|x|) |W|^2 \Big)\,dx\\
&\quad+\int_{\mathbb{R}}\frac{1}{2}\varphi'(t+|x|)  A_{1,2} D A_{1,2}^{-1} W \cdot W \, dx,\\
\mathcal{H}_{\mu}(t):&=\int_{\mathbb{R}} \Big( \varphi(t+|x|) A_{1,2} D A_{1,2}^{-1} \partial_t W\cdot \partial_t W +\frac{1}{2} \varphi'(t+|x|) A_{1,2} A_{2,1} \partial_x W \cdot \partial_x W\Big)\,dx  \\
&\quad+\int_{\mathbb{R}} \frac{1}{2}\varphi'''(t+|x|) | W|^2 \,dx\\
&\quad - \int_{\mathbb{R}} \bigg( \Big( \frac{1}{2}\varphi'''(t+|x|) +\varphi''(t+|x|)\delta_{0}(x)\Big)A_{1,2} A_{2,1}  W \cdot W \bigg)\, dx.
\end{aligned}
\end{equation}
In order to recover the coercivity estimates on $\mathcal{W}_{\mu}(t)$ and $\mathcal{H}_{\mu}(t)$, one requires that $\varphi(s)$ satisfies
\begin{align}
&\varphi'> 0,\quad\quad \varphi''< 0,\quad\quad \varphi'''> 0,\quad\quad \frac{1}{4}\varphi(t+|x|)\geq \frac{1}{\kappa_{1}}\varphi'(t+|x|).\label{varphi1}
\end{align}
Indeed, due to \eqref{Strong Dissipativity} and \eqref{varphi1}, for some constant $\kappa_1>0$, there holds that
\begin{equation}\nonumber
    \begin{aligned}
   \mathcal{W}_{\mu}(t)&\geq \int_{\mathbb{R}} \Big( \frac{1}{4}\varphi(t+|x|) |\partial_t W|^2+\kappa \varphi(t+|x|)|\partial_x W |^2\Big)\,dx\\
   &\quad+\int_{\mathbb{R}}\big(\frac{\kappa_{1}}{4}\varphi'(t+|x|)- \frac{1}{2}\varphi''(t+|x|) \big)|W|^2\,dx,
        \end{aligned}
\end{equation}
and
\begin{equation}\nonumber
    \begin{aligned}
    \mathcal{H}_{\mu}(t) &\geq \int_{\mathbb{R}} \Big(\kappa \varphi(t+|x|)  |\partial_t W|^2 +\frac{\kappa_{1}}{2} \varphi'(t+|x|)|\partial_x W|^2\Big)\,dx\\
    &\quad+\int_{\mathbb{R}} \frac{1}{2}\varphi'''(t+|x|)(|W|^2-A_{1,2} A_{2,1}W\cdot W) \,dx-\kappa_{1}\varphi''(t)|W(0,t)|^2\\
    &\geq \int_{\mathbb{R}} \Big(\kappa \varphi(t+|x|)  |\partial_t W|^2 +\frac{\kappa_{1}}{2} \varphi'(t+|x|)|\partial_x W|^2\Big)\,dx,
    \end{aligned}
\end{equation}
where we have used the positive definiteness of $A_{1,2}A_{2,1}$ with $|A_{1,2}A_{2,1}|\leq 1$. In addition, one has
    \begin{equation}\nonumber
    \begin{aligned}
    &\int_{\mathbb{R}} \varphi'(t+|x|)\Big( \frac{3}{2} |\partial_t W|^2- \frac{x}{|x|} A_{1,2} A_{2,1}\partial_x W \cdot \partial_t W+\frac{1}{2}\frac{x}{|x|} A_{1,2} A_{2,2}A_{1,2}^{-1} \partial_t W \cdot \partial_t W \Big)\, dx\\
    &\quad\leq \int_{\mathbb{R}} \Big( C \varphi'(t+|x|) | \partial_t W|^2 +\frac{\kappa_1}{4}  \varphi'(t+|x|) |\partial_x W|^2\Big)\, dx,
    \end{aligned}
\end{equation}
which can be controlled by the left-hand side of \eqref{ppppp2} provided that
\begin{align}
& C \varphi'(t+|x|)\leq \frac{\kappa}{2}\varphi(t+|x|).    \label{varphi2}
\end{align}
In addition, under the assumptions \eqref{a21}, one needs $\mu>1/2$ and
\begin{align}
& \varphi(s)\sim (1+s)^{2\mu-1},\quad\quad \varphi'(s)\sim (1+s)^{2\mu-2},    \label{varphi3}
\end{align}
so as to bound the initial energy $\mathcal{W}(0)$ by $Y_{0}^2$ in terms of \eqref{a21} and the Caffarelli-Kohn-Nirenberg inequality \eqref{CKNbest}. One can show that the function
\begin{align}
&\varphi(s)=(a+s)^{2\mu-1}\quad\text{with $\frac{1}{2}<\mu<1$ and some constant $a>\max\{\frac{4}{\kappa_{1}},\frac{2C}{\kappa}\}$}
\end{align}
 fulfills the conditions \eqref{varphi1}, \eqref{varphi2} and \eqref{varphi3}. Therefore, integrating \eqref{ppppp2} over $[0,t]$ and using \eqref{estimateThm1}, we obtain the desired inequality \eqref{estimateW}.

\end{proof}

\vspace{1ex}

\noindent
\textbf{Proof of Theorem \ref{ThmDecayGeneral2}:} In view of the estimate \eqref{estimateW} obtained in Lemma \ref{lemmaW} and the facts that $U_1=\partial_x W$, $U_{2}=-A_{1,2}^{-1}\partial_t W$ and that, for $x\in \mathbb{R}, t>0$ and $\mu \geq 1/2$, \[(1+t)^{2\mu-1}\leq (1+t+|x|)^{2\mu-1},\] we get the $L^2$ rate $(1+t)^{-\mu+\frac{1}{2}}$ of $U$ in \eqref{decay1112}. Applying the $L^2$ rate of $U$ in \eqref{decay1112} and Lemma \ref{decayineq} to the differential inequality \eqref{llddd}, we recover faster time-decay rates for $\partial_xU$ in \eqref{decay1112}. Finally, the faster decay rates for $U_{2}$ follows from the decay rates obtained for $\partial_xU$ and \eqref{this111}.
The proof of Theorem \ref{ThmDecayGeneral2} is now complete.

\section{Proof of Theorem \ref{ThmEuler1} and Theorem \ref{ThmEuler2}}
\label{sec:EulerLin}

\subsection{Asymptotic estimates for the compressible Euler equations with damping}

In this subsection, we apply the methods developed in Theorems \ref{DecayThm1}, \ref{ThmDecayGeneral1} and \ref{ThmDecayGeneral2} to a concrete nonlinear partially dissipative hyperbolic system: the damped compressible Euler equations \eqref{Euler} with a general pressure function $P(\rho)$ satisfying \eqref{P}.


\subsection{Proof of global existence and time-decay estimates}\label{subsectionexistence}

To prove the global existence of the system \eqref{Euler}, we establish a-priori estimates as follows.

 \begin{lemma}\label{eulerpriori}$($A priori estimates$)$
 Let $(\rho,u)$ be the solution to the system \eqref{Euler} on $[0,T)$ for any given time $T>0$. Define
 \begin{equation}\label{Xt}
 \begin{aligned}
 X(t):&=\sup_{\tau\in[0,t]}\Big(\|(\rho-\bar{\rho},u)(\tau)\|_{H^{2}}^2+\tau\|u(\tau)\|_{L^2}^2+\tau\|\partial_x (\rho-\bar{\rho},u)(\tau)\|_{L^2}^2\Big)\\
 &\quad+\int_{0}^{t}\Big(\|\partial_x (\rho-\bar{\rho})(\tau)\|_{H^1}^2+\|u(\tau)\|_{H^{2}}^2+\tau\|\partial_x u(\tau)\|_{L^2}^2\Big)\,d\tau.
 \end{aligned}
 \end{equation}
 There exists a small constant $\delta_1$ independent of $T$ such that if
\begin{align}
    &  X(t)\leq \delta_1,\quad\quad 0<t<T,\label{smallness}
\end{align}
then there exists a generic constant $C_{0}>0$ such that
\begin{align}
    &X(t)\leq C_{0} \|(\rho_{0}-\bar{\rho},u_{0})\|_{H^2}^2,\quad\quad 0<t<T.\label{eulerpriorie}
\end{align}

 \end{lemma}

\begin{proof}
We use similar arguments to those used in the Subsection \ref{subsectionL}. Denote the purturbation
\[
n:=\rho-\bar{\rho}.
\]
It is easy to check that the basic energy equality for \eqref{Euler} holds:
\begin{equation}\label{EulerL2}
\begin{aligned}
\frac{d}{dt}E_{euler}(t)+\lambda \|u(t)\|_{L^2}^2=0.
\end{aligned}
\end{equation}
Here $E_{euler}(t)$ is given by
\[
E_{euler}(t):=\int_{\mathbb{R}}\Big( \frac{1}{2}\rho |u|^2 +\rho\int_{\bar{\rho}}^{\rho}\frac{P'(s)}{s}\,ds \Big)\,dx\sim \|(n,u)(t)\|_{L^2}^2.
\]
To derive higher-order estimates for $(n,u)$, we write  \eqref{Euler} as
\begin{equation}\label{Eulerr}
\left\{
\begin{aligned}
&\partial_{t} n+\rho \partial_x u=- u \partial_x n,\\
&\partial_{t} u +G(n)\partial_x n +\lambda  u= -u\partial_x u,\\
\end{aligned}
\right.
\end{equation}
where 
\[
G(n):=\frac{P'(\bar{\rho}+n)}{\bar{\rho}+n}>0.
\]
We have
\begin{equation}\label{eulerh1}
\begin{aligned}
  & \frac{d}{dt}\int_{\mathbb{R}}\Big( |\partial_{x}n|^2+\frac{\rho}{G(n)}|\partial_{x}u|^2\Big)\,dx+\int_{\mathbb{R}}\Big(\frac{2\lambda\rho }{G(n)}-\partial_{t}\frac{\rho}{G(n)}\Big)|\partial_{x}u|^2\,dx\\
  &\leq 4\|\partial_{x} n(t)\|_{L^{\infty}} \|\partial_{x} n(t)\|_{L^2}\|\partial_{x}u(t)\|_{L^2}\\
&\quad+2\Big\|\frac{\rho}{G(n)}\Big\|_{L^{\infty}_{t}(L^{\infty})}\big(\|\partial_{x} (u\partial_{x}u)(t)\|_{L^2}+\|\partial_{x}G(n)\partial_{x}n(t)\|_{L^2}\big) \|\partial_{x}u(t)\|_{L^2},
\end{aligned}
\end{equation}
and
\begin{equation}\label{eulerh11}
\begin{aligned}
  & \frac{d}{dt}\int_{\mathbb{R}}t\Big( |\partial_{x}n|^2+\frac{\rho}{G(n)}|\partial_{x}u|^2\Big)\,dx+t\int_{\mathbb{R}}\Big(\frac{2\lambda\rho }{G(n)}-\partial_{t}\frac{\rho}{G(n)}\Big)|\partial_{x}u|^2\,dx\\
  &\leq \int_{\mathbb{R}}\Big( |\partial_{x}n|^2+\frac{\rho}{G(n)}|\partial_{x}u|^2\Big)\,dx\\
  &\quad+4t\|\partial_{x} n(t)\|_{L^{\infty}} \|\partial_{x} n(t)\|_{L^2}\|\partial_{x}u(t)\|_{L^2}\\
&\quad+2t\Big\|\frac{\rho}{G(n)}\Big\|_{L^{\infty}_{t}(L^{\infty})}\big(\|\partial_{x} (u\partial_{x}u)(t)\|_{L^2}+\|\partial_{x}G(n)\partial_{x}n(t)\|_{L^2}\big) \|\partial_{x}u(t)\|_{L^2}.
\end{aligned}
\end{equation}

Unlike in the linear setting, $H^2$-regularity estimates are needed to control the nonlinear terms. The the system satisfied by $(\partial_{x}^2n, \partial_{x}^2 u)$ reads
\begin{equation}\label{Eulerr1}
\left\{
\begin{aligned}
&\partial_{t}\partial_{x}^2 n+u \partial_{xxx}^3 n+\rho\partial_{xxx}^3 u=\mathcal{R}_1,\\
&\partial_{t}\partial_{x}^2 u +u \partial_{xxx}^3 u+G(n)\partial_{xxx}^3 n +\lambda \partial_{x}^2 u=\mathcal{R}_{2},
\end{aligned}
\right.
\end{equation}
with the commutator terms
\[\mathcal{R}_1:=[u,\partial_{x}^2]\partial_x n+[n,\partial_{x}^2]\partial_x u \quad \text{and} \quad \mathcal{R}_2:=[u,\partial_{x}^2]\partial_x u+[G(n),\partial_{x}^2]\partial_x n.\]
It thence holds that
\begin{equation}\label{eulerh2}
    \begin{aligned}
        &\frac{d}{dt}\int_{\mathbb{R}}\Big( |\partial_{x}^2 n|^2+\frac{\rho}{G(n)} |\partial_{x}^2 u|^2\Big)\,dx+\int_{\mathbb{R}}\bigg( \frac{ 2\lambda\rho}{G(n)}- \partial_t \frac{\rho}{G(n)} \bigg) |\partial_{x}^2 u|^2 \,dx\\
        &\leq 2\|\partial_x n(t)\|_{L^{\infty}} \|\partial_{x}^2 n(t)\|_{L^2}\|\partial_{x}^2 u(t)\|_{L^2}+ \|\partial_x u(t)\|_{L^{\infty}}\|\partial_{x}^2 n(t)\|_{L^2}^2\\
        &\quad+\Big(\Big\|\frac{\rho}{G(n)}\Big\|_{L^{\infty}_{t}(L^{\infty})}\| \partial_x u(t)\|_{L^{\infty}}+\Big\|\partial_x\frac{\rho}{G(n)}\Big\|_{L^{\infty}_{t}(L^{\infty})}\| u(t)\|_{L^{\infty}}\Big)\|\partial_{x}^2u(t)\|_{L^2}^2\\
        &\quad+2\|\mathcal{R}_1(t)\|_{L^2}\|\partial_{x}^2n(t)\|_{L^2}+2\Big\|\frac{\rho}{G(n)}\Big\|_{L^{\infty}_{t}(L^{\infty})}\|\mathcal{R}_2(t)\|_{L^2} \|\partial_{x}^2 u(t)\|_{L^2}.
    \end{aligned}
\end{equation}
Furthermore, to capture time-decay information for $n$, from \eqref{Eulerr} we have
\begin{equation}\label{eulercross}
    \begin{aligned}
       &\frac{d}{dt}\int_{\mathbb{R}} (u\partial_x n+\partial_x u\partial_{x}^2 n)\,dx\\
       &\quad+\int_{\mathbb{R}}\Big( G(n)(|\partial_x n|^2+|\partial_x^2 n|^2)-\rho(|\partial_x u|^2+|\partial_x^2 u|^2)+\lambda u \partial_x n+\lambda \partial_x u \partial_x^2n \Big)\,dx\\
        &\leq \|u\partial_{x}n(t)\|_{H^1}\|\partial_x u(t)\|_{H^1}+\|u\partial_{x}u(t)\|_{H^1}\|\partial_x n(t)\|_{H^1}+\|\mathcal{R}_0(t)\|_{L^2}\|\partial_x^2n(t)\|_{L^2}.
    \end{aligned}
\end{equation}
Let $c_{1},c_{2}\in(0,1)$ be two constants to be chosen later. Define
\begin{equation}\nonumber
    \begin{aligned}
        \mathcal{L}_{euler}(t)&:=E_{euler}(t)+(1+c_{1}t)\int_{\mathbb{R}}\Big( |\partial_{x}n|^2+\frac{\rho}{G(n)}|\partial_{x}u|^2\Big)\,dx\\
        &\quad+\int_{\mathbb{R}}\Big( |\partial_{x}^2 n|^2+\frac{\rho}{G(n)} |\partial_{x}^2 u|^2\Big)\,dx+c_{2}\int_{\mathbb{R}} (u\partial_x n+\partial_x u\partial_{x}^2 n)\,dx,
    \end{aligned}
\end{equation}
and
\begin{equation}\nonumber
    \begin{aligned}
        \mathcal{D}_{euler}(t)&=\lambda \|u(t)\|_{L^2}^2+(1+c_{1}t)\int_{\mathbb{R}}\Big( \frac{ 2\lambda\rho}{G(n)}- \partial_t \frac{\rho}{G(n)} \Big) |\partial_{x}^2 u|^2 \,dx\\
        &\quad+\int_{\mathbb{R}}\Big( \frac{ 2\lambda\rho}{G(n)}- \partial_t \frac{\rho}{G(n)} \Big) |\partial_{x}^2 u|^2 \,dx\\
        &\quad+c_{2}\int_{\mathbb{R}}\Big( G(n)(|\partial_x n|^2+|\partial_x^2 n|^2)-\rho(|\partial_x u|^2+|\partial_x^2 u|^2)\\
        &\quad\quad\quad\quad\quad+\lambda u \partial_x n+\lambda \partial_x u \partial_x^2n \Big)\,dx.
    \end{aligned}
\end{equation}
Using \eqref{smallness} with $\delta_1$ suitably small, we have
\begin{align}
    &0<\frac{\bar{\rho}}{2P'(\bar{\rho})}\leq \frac{\rho}{G(n)}(x,t)\leq \frac{2\bar{\rho}}{P'(\bar{\rho})},\quad\quad (x,t)\in\mathbb{R}\times(0,T),\label{wss}
\end{align}
and
\begin{equation}
\begin{aligned}
&\|\partial_t \frac{\rho}{G(n)}\|_{L^{\infty}_{t}(L^{\infty}_{x})}\\
&\quad\leq C\|\partial_t n\|_{L^{\infty}_{t}(L^{\infty}_{x})}\\
&\quad\leq C\|u\|_{L^{\infty}_{t}(L^{\infty}_{x})}\|\partial_xn\|_{L^{\infty}_{t}(L^{\infty}_{x})}+C\|\partial_x u\|_{L^{\infty}_{t}(L^{\infty}_{x})}(1+\|n\|_{L^{\infty}_{t}(L^{\infty}_{x})})\leq \lambda.\label{wss1}
\end{aligned}
\end{equation}
Adjusting the coefficients $c_{1},c_{2}$ suitably and making use of \eqref{wss} and \eqref{wss1}, we obtain
\begin{equation}\label{eulersim}
    \begin{aligned}
    \mathcal{L}_{euler}(t)\sim \|(n,u)(t)\|_{H^2}^2+c_{1}t\|\partial_x (n,u)(t)\|_{L^2}^2,
    \end{aligned}
\end{equation}
and
\begin{equation}\label{eulerleq}
    \begin{aligned}
      \mathcal{D}_{euler}(t)&\gtrsim  \|u(t)\|_{H^2}^2+\|\partial_x n(t)\|_{H^1}^2+c_1 t\|\partial_x u(t)\|_{L^2}^2.
    \end{aligned}
\end{equation}
Then, it follows from \eqref{eulerh1}, \eqref{eulerh11}, \eqref{eulerh2}, \eqref{eulercross} and \eqref{wss} that
\begin{equation}
    \begin{aligned}
  &\frac{d}{dt}\mathcal{L}_{euler}(t)+\mathcal{D}_{euler}(t) \\ &\lesssim \|(\partial_x n,\partial_x u)(t)\|_{L^{\infty}} \|(\partial_{x} u,\partial_{x}n)(t)\|_{L^{2}}^2\\
    &\quad+ \|u\partial_{x}n(t)\|_{H^1}\|(n,\partial_x u)(t)\|_{H^1}+\|u\partial_{x}u(t)\|_{H^1}\|(\partial_x n, u)(t)\|_{H^1}\\
    &\quad+\|\mathcal{R}_1(t)\|_{L^2}\|\partial_{x}^2n(t)\|_{L^2}+\|\mathcal{R}_2(t)\|_{L^2} \|\partial_{x}^2 u(t)\|_{L^2}\\
     &\quad+ t\|\partial_{x} n(t)\|_{L^{\infty}} \|\partial_{x} n(t)\|_{L^2}\|\partial_{x}u(t)\|_{L^2}\\
     &\quad+t\big(\|\partial_{x} (u\partial_{x}u)\|_{L^2}+\|\partial_{x}G(n)\partial_{x}n(t)\|_{L^2}\big) \|\partial_{x}u(t)\|_{L^2}.\label{eulerlyapunov1}
    \end{aligned}
\end{equation}
The nonlinear terms on the right-hand side of \eqref{eulerlyapunov1} are analyzed as follows. First, by the Sobolev embedding one has
\begin{equation}\nonumber
    \begin{aligned}
    &\|(\partial_x n,\partial_x u)(t)\|_{L^{\infty}} \lesssim \|(\partial_x n,\partial_x u)(t)\|_{H^1}.
    \end{aligned}
\end{equation}
Since $H^1$ is an algebra, we obtain
\begin{equation}\nonumber
    \begin{aligned}
       & \|u\partial_{x}n(t)\|_{H^1}\|(n,\partial_x u)(t)\|_{H^1}+\|u\partial_{x}u(t)\|_{H^1}\|(\partial_x n, u)(t)\|_{H^1}\\
       &\quad\quad\quad\quad\quad\quad\quad\quad\quad\quad\quad\quad\quad\lesssim \|(n,u)(t)\|_{H^2}\|(\partial_x n, u)(t)\|_{H^1}.
    \end{aligned}
\end{equation}
From standard commutator estimates (cf. \cite{katocom}) and $H^1\hookrightarrow L^{\infty}$, one gets
\begin{equation}\nonumber
    \begin{aligned}
\|\mathcal{R}_1(t)\|_{L^2}&\lesssim \|[u,\partial_{x}^2]\partial_{x} n(t)\|_{L^2}+\|[n,\partial_{x}^2]\partial_x u(t)\|_{L^2}\\
&\lesssim \|(n,u)(t)\|_{H^2}\|\partial_{x}^2(n,u)(t)\|_{L^2},
       \end{aligned}
\end{equation}
and similarly,
\begin{equation}\nonumber
    \begin{aligned}
\|\mathcal{R}_2(t)\|_{L^2}&\lesssim \|[u,\partial_{x}^2]\partial_x u(t)\|_{L^2}+\|[G(n),\partial_{x}^2]\partial_x n(t)\|_{L^2}\\
&\lesssim \|(n,u)(t)\|_{H^2}\|\partial_{x}^2(n,u)(t)\|_{L^2}.
       \end{aligned}
\end{equation}
Concerning the time-weighted nonlinear terms, we have
\begin{equation}\nonumber
    \begin{aligned}
    &t\|\partial_{x} n(t)\|_{L^{\infty}} \|\partial_{x} n(t)\|_{L^2}\|\partial_{x}u(t)\|_{L^2}+t\big(\|\partial_{x} (u\partial_{x}u)\|_{L^2}+\|\partial_{x}G(n)\partial_{x}n(t)\|_{L^2}\big) \|\partial_{x}u(t)\|_{L^2}\\
    &\lesssim \|(\partial_{x}n,u)(t)\|_{L^2}\Big(t^{\frac{1}{2}} \|\partial_{x}(n,u)(t)\|_{L^2} \Big) \Big( t^{\frac{1}{2}} \|\partial_{x}u(t)\|_{L^2}\Big)\\
    &\lesssim \eta_0 t\|\partial_{x}u(t)\|_{L^2}^2+\frac{1}{\eta_0}\|(\partial_{x}n,u)(t)\|_{L^2}^2 t\|\partial_{x}(n,u)(t)\|_{L^2}^2.
    \end{aligned}
\end{equation}
Substituting the above estimates into \eqref{eulerlyapunov1}, choosing a suitable small constant $\eta_0>0$ and using  \eqref{smallness}, \eqref{eulersim} and \eqref{eulerleq}, we derive
\begin{equation}\label{eulerlyapunov}
    \begin{aligned}
    &\frac{d}{dt}\mathcal{L}_{euler}(t)+\|u(t)\|_{H^2}^2+\|\partial_x n(t)\|_{H^1}^2+t\|\partial_x u(t)\|_{L^2}^2\\
    &\quad\lesssim \|(\partial_{x}n,u)(t)\|_{L^2}\mathcal{L}_{euler}(t).
    \end{aligned}
\end{equation}
Employing Gr\"onwall's inequality to \eqref{eulerlyapunov} and using the fact that $\|(\partial_{x}n,u)(t)\|_{L^2}$ is uniformly integrable due to \eqref{smallness}, we get
\begin{equation}\label{asdf}
    \begin{aligned}
      &\|(n,u)(t)\|_{H^{2}}^2+t\|\partial_x (n,u)(t)\|_{L^2}^2\\
      &\quad+\int_{0}^{t}\Big(\|\partial_x n(\tau)\|_{H^1}^2+\|u(\tau)\|_{H^{2}}^2  +\tau\|\partial_x u(\tau)\|_{L^2}^2\Big)\,d\tau\lesssim \|(\rho_{0}-\bar{\rho},u_{0})\|_{H^2}^2.
    \end{aligned}
\end{equation}

Finally, taking the $L^2$ inner product of $\eqref{Eulerr}_{2}$ with $v$, we have
\begin{equation}\label{vinq}
    \begin{aligned}
    &\frac{d}{dt}\|u(t)\|_{L^2}^2+2\lambda\|u(t)\|_{L^2}^2\lesssim ( \|u\|_{L^{\infty}_{t}(L^{\infty}_{x})}\|\partial_x u(t)\|_{L^2}+ \|\partial_x n(t)\|_{L^2}) \|u(t)\|_{L^2}.
    \end{aligned}
\end{equation}
This, together with Gr\"onwall's inequality and \eqref{asdf}, leads to
\begin{equation}\nonumber
    \begin{aligned}
    \|u(t)\|_{L^2}&\lesssim e^{-\lambda t}\|u_{0}\|_{L^2}+\int_{0}^{t}e^{-\lambda(t-\tau)}\|\partial_x(n,u)(\tau)\|_{L^2}\\
    &\lesssim(1+t)^{-\frac{1}{2}} \|(\rho_{0}-\bar{\rho},u_{0})\|_{H^2},
    \end{aligned}
\end{equation}
which concludes the proof of Lemma \ref{eulerpriori}.

\end{proof}

\vspace{2ex}

\noindent
\textbf{Proof of Theorems \ref{ThmEuler1} and \ref{ThmEuler2}:} According to classical local well-posedness results (see e.g. \cite{HJR,BHN,Majdalocal,Serre,XK1}), there exists a time $T_{0}>0$ such that the system \eqref{Euler} associated to the initial datum $(\rho_{0},u_{0})$ admits a unique solution $(\rho,u)$ satisfying $(\rho-\bar{\rho}, u)\in C([0,T_0]; H^2 )$. Then, according to the a-priori estimates \eqref{eulerpriorie} established in Lemma \ref{eulerpriorie} and a standard bootstrap argument, one can extend the solution $(\rho,u)$ globally in time and recover the property  \eqref{eulerdecay1} as long as $C_{0} \|(\rho_{0}-\bar{\rho},u_{0})\|_{H^2}^2<\delta_1$. 

The time-decay estimates \eqref{eulerdecay2} and \eqref{eulerdecay3} in Theorem \ref{ThmEuler2} are proved in Lemmas \ref{lemmaeulerdecay1} and \ref{lemmaeulerdecay2} in the next subsection.

\subsection{Enhanced time-decay rates}

 \begin{lemma}\label{lemmaeulerdecay1}
Let $(\rho,u)$ be the global solution of the system \eqref{Euler} associated to the initial datum $(\rho_{0},u_{0})$, and $\rho_*$ be the global solution of the system \eqref{parabolicEuler} subject to the initial datum $\rho_0$. Assume  \eqref{eulera1} and $|x| u_0\in L^2 $, then $\rho-\rho_*$ verifies \eqref{errorEuler}. Additionally, if $|x|^{\mu}(\rho_0-\bar{\rho})$ holds with $0<\mu\leq 1$, then $(\rho,u)$ fulfills \eqref{eulerdecay2}. 
 \end{lemma}

\begin{proof}
Following a similar procedure to the one used in the proof of Theorem \ref{ThmDecayGeneral2}, we introduce the damped mode
\[
R^*:=\frac{P'(\bar{\rho})}{\bar{\rho}} n+\lambda \int^x_{-\infty}  u\,dy,
\]
so as to rewrite $\eqref{Eulerr}_{1}$ as
\begin{equation}\label{rhopara}
\begin{aligned}
&\partial_{t}\rho-\frac{P'(\bar{\rho})}{\lambda}\partial^2_x\rho= -\frac{\bar{\rho}}{\lambda}\partial_x^2 R^*-\partial_x(nu).
\end{aligned}
\end{equation}
One observes that $R^*$ satisfies a damped equation
\begin{equation}\label{R*}
\begin{aligned}
&\partial_t R^*+\lambda R^*=F_3,
\end{aligned}
\end{equation}
with 
\[
F_3:=- \frac{1}{2}u^2-I(n)+\frac{P'(\bar{\rho})}{\bar{\rho}}\partial_{t} n,\quad \text{and} \quad I(n):=\int_{0}^{n}\Big(\frac{P'(\bar{\rho}+s)}{\bar{\rho}+s}-\frac{P'(\bar{\rho})}{\bar{\rho}}\Big)\,ds.
\]
The proof of the decay estimates \eqref{errorEuler}-\eqref{eulerdecay2} is split in six steps.

\smallbreak
\noindent $\bullet$  \emph{Step 1: Decay of $R^*$}. Let $|x| u_0\in L^2 $. From \eqref{R*}, we have
\begin{equation}\label{R*es}
\begin{aligned}
\|R^*(t)\|_{L^2}&\leq e^{-\lambda t}\|R^*(0)\|_{L^2}+\int_{0}^{t} e^{-\lambda (t-\tau)}\|F_3(\tau)\|_{L^2}\,d\tau.
\end{aligned}
\end{equation}
By virtue of \eqref{CKNbest}, the first term on the right-hand side of \eqref{R*es} is controlled by
\begin{equation}\label{R*es1}
\begin{aligned}
\|R^*(0)\|_{L^2}\lesssim \|\rho_0-\bar{\rho}\|_{L^2}+\| |x| u_0\|_{L^2}.
\end{aligned}
\end{equation}
We now estimate the nonlinear term $F_3$. Recalling the estimate \eqref{eulerdecay1} and using Gagliardo-Nirenberg inequality, we get
\begin{equation}\label{R*es2}
\begin{aligned}
\|u^2(t)\|_{L^2}&\leq \|u(t)\|_{L^2} \|u(t)\|_{L^{\infty}}\\
&\lesssim \|u(t)\|_{L^2}^{\frac{3}{2}}\|\partial_x u(t)\|_{L^2}^{\frac{1}{2}}\lesssim \|(\rho_0-\bar{\rho},u_0)\|_{H^2} (1+t)^{-1}.
\end{aligned}
\end{equation}
With $\eqref{Eulerr}_{1}$, we have
\begin{equation}\label{R*es3}
\begin{aligned}
\|\partial_{t} n(t)\|_{L^2}&\leq (\bar{\rho}+\|n\|_{L^{\infty}_{t}(L^{\infty})})\|u_x(t)\|_{L^2}+\|u\|_{L^{\infty}_{t}(L^{\infty})}\|\partial_x n(t)\|_{L^2}\\
&\lesssim \|(\rho_0-\bar{\rho},u_0)\|_{H^2} (1+t)^{-\frac{1}{2}}.
\end{aligned}
\end{equation}
Since $I(0)=I'(0)=0$ holds,  \eqref{eulerdecay1} and Gagliardo-Nirenberg inequality ensure that
\begin{equation}\label{R*es4}
\begin{aligned}
\|I(n)(t)\|_{L^2}&\lesssim \|n(t)\|_{L^{\infty}} \|n(t)\|_{L^2}\\
&\lesssim \|n(t)\|_{L^2}^{\frac{3}{2}} \|\partial_x n(t)\|_{L^2}^{\frac{1}{2}}\lesssim \|(\rho_0-\bar{\rho},u_0)\|_{H^2} (1+t)^{-\frac{1}{4}}.
\end{aligned}
\end{equation}
Putting the above estimates \eqref{R*es1}-\eqref{R*es4} into \eqref{R*es}, we obtain the decay of $R^*$ as follows
\begin{equation}\label{decayR*}
\begin{aligned}
\|R^*(t)\|_{L^2}&\lesssim (\|(\rho_0-\bar{\rho},u_0)\|_{H^2}+\| |x|u_0\|_{L^2})(1+t)^{-\frac{1}{4}}.
\end{aligned}
\end{equation}

\smallbreak
\noindent $\bullet$  \emph{Step 2: Decay of $\rho-\rho_*$}. Next, we aim to establish the stability of the parabolic profile, i.e. \eqref{errorEuler}. By \eqref{parabolicEuler} and \eqref{rhopara}, the error $\rho-\rho_*$ solves
\begin{equation}\label{rhorho*error}
\begin{aligned}
&\partial_{t}(\rho-\rho_*)-\frac{P'(\bar{\rho})}{\lambda}\partial^2_x(\rho-\rho_*)= -\frac{\bar{\rho}}{\lambda}\partial_x^2 R^*-\partial_x(nu),
\end{aligned}
\end{equation}
with the initial datum $(\rho-\rho_*)(0,x)=0$. Thus, Duhamel's principle for \eqref{rhorho*error} implies
\begin{equation}\label{duhamelerroreuler}
\begin{aligned}
\rho-\rho_*=-\int_{0}^{t}e^{\frac{P'(\bar{\rho})}{\lambda}\partial^2_x (t-s)}\Big(\frac{\bar{\rho}}{\lambda}\partial_x^2 R^*+\partial_x(nu) \Big)\,ds.
\end{aligned}
\end{equation}
Similarly to \eqref{Gs1}-\eqref{Gs}, one deduces from \eqref{eulerdecay1} and \eqref{decayR*} that for $0< \var <1/2$
\begin{equation}
\begin{aligned}
\|e^{\frac{P'(\bar{\rho})}{\lambda}\partial^2_x (t-s)}\partial_x^2 R^*(s)\|_{L^2}&\lesssim (t-s)^{-(1-\var)} (\|R^*(s)\|_{L^2}+\|\partial_x R^*(s)\|_{L^2})\\
&\lesssim (t-s)^{-(1-\var)} s^{-\frac{1}{4}}(\|(\rho_0-\bar{\rho},u_0)\|_{H^2}+\| |x|u_0\|_{L^2})\\
&\quad+(t-s)^{-(1-\var)} s^{-\frac{1}{2}}\|(\rho_0-\bar{\rho},u_0)\|_{H^2}.
\end{aligned}
\end{equation}
Owing to \eqref{eulerdecay1} and
\[
\|nu(s)\|_{L^2}\leq \|u(s)\|_{L^2}\|n(s)\|_{L^{\infty}}\lesssim \|u(s)\|_{L^2}\|n(s)\|_{L^{2}}^{\frac{1}{2}}\|\partial_x n(s)\|_{L^2}^{\frac{1}{2}},
\]
we have
\begin{equation}
\begin{aligned}
&\|e^{\frac{P'(\bar{\rho})}{\lambda}\partial^2_x (t-s)}\partial_x(nu)(s)\|_{L^2}\\
&\quad\lesssim (t-s)^{-\frac{1}{2}} \|nu(s)\|_{L^2}\lesssim (t-s)^{-\frac{1}{2}}  s^{-\frac{3}{4}} \|(\rho_0-\bar{\rho},u_0)\|_{H^2}.
\end{aligned}
\end{equation}
Hence, applying $L^2$-norm of \eqref{duhamelerroreuler} yields
\begin{equation}\nonumber
\begin{aligned}
\|(\rho-\rho^*)(t)\|_{L^2}&\lesssim \int_{0}^{t}  (t-s)^{-(1-\var)} s^{-\frac{1}{4}}\,ds   \Big(\|(\rho_0-\bar{\rho},u_0)\|_{H^2}+\| |x|u_0\|_{L^2} \Big)\\
&\quad+\bigg(\int_{0}^{t}(t-s)^{-(1-\var)} s^{-\frac{1}{4}}\,ds+\int_{0}^{t}(t-s)^{-\frac{1}{2}} s^{-\frac{3}{4}}\,ds\bigg) \|(\rho_0-\bar{\rho},u_0)\|_{H^2}\\
&\leq \bigg(t^{-\frac{1}{4}+\var}+t^{-\frac{1}{2}+\var}+t^{-\frac{1}{4}}\bigg) (\|(\rho_0-\bar{\rho},u_0)\|_{H^2}+\| |x|u_0\|_{L^2}).
\end{aligned}
\end{equation}
Together with the uniform $L^2$-bound of $\rho-\rho^*$, the error $\rho-\rho^*$ satisfies the decay estimate \eqref{errorEuler} for any $0<\var<\frac{1}{4}$.

\smallbreak
\noindent $\bullet$ \emph{Step 3: Decay of $\rho-\bar{\rho}$ for $0<\mu<1/2$}. We assume further
\[
\widetilde{X}(0):=\|(\rho_0-\bar{\rho},u_0)\|_{H^2}+\||x|^{\mu}(\rho_0-\bar{\rho})\|_{L^2}+\||x|u_0\|_{L^2}<\infty.
\]
Then, for all $0<\mu\leq 1$, a direct application of Lemma \ref{lemmaparabolic} yields
\begin{equation}\label{rho*decay}
\left\{
\begin{aligned}
&\|(\rho^*-\bar{\rho})(t)\|_{L^2}\lesssim (1+t)^{-\frac{\mu}{2}} \|(1+|x|^{\mu})(\rho_0-\bar{\rho})\|_{L^2},\\
&\|\partial_x^k(\rho^*-\bar{\rho})(t)\|_{L^2}\lesssim (1+t)^{-\frac{k}{2}-\frac{\mu}{2}} \|(1+|x|^{\mu})(\rho_0-\bar{\rho})\|_{L^2},\quad k=1,2,...
\end{aligned}
\right.
\end{equation}
In the case $0<\mu<1/2$, the combination of \eqref{errorEuler} with $0<\var<\frac{1-2\mu}{4}$ and \eqref{rho*decay} yields
\begin{equation}
\begin{aligned}
\|(\rho-\bar{\rho})(t)\|_{L^2}&\leq \|(\rho^*-\bar{\rho})(t)\|_{L^2}+\|(\rho-\rho^*)(t)\|_{L^2}\lesssim (1+t)^{-\frac{\mu}{2}}\widetilde{X}(0).
\end{aligned}
\end{equation}

\smallbreak
\noindent $\bullet$ \emph{Step 4: Decay of $\rho-\bar{\rho}$ for $1/2\leq \mu< 1$}. When $\mu\geq 1/2$, the $L^2$-decay rate of the error $\rho-\rho^*$ in \eqref{errorEuler} is not enough. To overcome this difficulty, one needs to improve the decay of $\rho-\rho^*$. In the case $1/2\leq \mu< 1$, the lowest term $I(n)$ on the right-hand side of \eqref{R*} can be handled as
\begin{equation}\label{improveIn}
\begin{aligned}
\|I(n)(t)\|_{L^2}&\lesssim \|n(t)\|_{L^{\infty}} \|n(t)\|_{L^2}\\&\lesssim \|n(t)\|_{L^2}^{\frac{3}{2}} \|\partial_x n(t)\|_{L^2}^{\frac{1}{2}}\\
&\lesssim \|n(t)\|_{L^2}^{\frac{1}{2}} \|\partial_x n(t)\|_{L^2}^{\frac{1}{2}}( \|(\rho^*-\bar{\rho})(t)\|_{L^2}+\|(\rho-\rho^*)(t)\|_{L^2})\\
&\lesssim (1+t)^{-\frac{1}{4}-\frac{\mu}{2}} X(0) \bigg( \widetilde{X}(0)+ \sup_{\tau\in[0,t]}(1+\tau)^{\frac{\mu}{2}}\|(\rho-\rho^*)(\tau)\|_{L^2} \bigg),
\end{aligned}
\end{equation}
where we used $n=\rho^*-\bar{\rho}+\rho-\rho^*$, \eqref{eulerdecay1} and \eqref{rho*decay}. By \eqref{R*es}-\eqref{R*es3}, \eqref{improveIn} and $\frac{1}{4}+\frac{\mu}{2}\geq \frac{1}{2}$, we have
\begin{equation}
\begin{aligned}
\|R^*(t)\|_{L^2}&\leq\widetilde{X}(0)(1+t)^{-\frac{1}{2}}+X(0)(1+t)^{-\frac{1}{2}}  \sup_{\tau\in[0,t]}(1+\tau)^{\frac{\mu}{2}}\|(\rho-\rho^*)(\tau)\|_{L^2}.
\end{aligned}
\end{equation}
On the other hand, the global existence result guarantees that
\begin{equation}
\begin{aligned}
\|\partial_xR^*(t)\|_{L^2}&\lesssim \|\partial_x n(t)\|_{L^2}+\|u(t)\|_{L^2}\lesssim X(0)(1+t)^{-\frac{1}{2}}.
\end{aligned}
\end{equation}
Hence, following the computations done in Step 2, we arrive at
\begin{equation}\label{imrpoveR*}
\begin{aligned}
&\|e^{\frac{P'(\bar{\rho})}{\lambda}\partial^2_x (t-s)}\partial_x^2 R^*(s)\|_{L^2}\\
&\quad \lesssim (t-s)^{-(1-\var)} s^{-\frac{1}{2}}\Big(\|(\rho_0-\bar{\rho},u_0)\|_{H^2}+\| |x|^{\mu}(\rho_0-\bar{\rho})\|_{L^2}+\| |x|u_0\|_{L^2})\Big)\\
&\quad\quad+ (t-s)^{-(1-\var)} s^{-\frac{1}{2}}\|(\rho_0-\bar{\rho},u_0)\|_{H^2} \sup_{\tau\in[0,t]}(1+\tau)^{\frac{\mu}{2}}\|(\rho-\rho^*)(\tau)\|_{L^2}.
\end{aligned}
\end{equation}
Similar calculations leads to 
\begin{equation}\label{imrpovedxnu}
\begin{aligned}
&\|e^{\frac{P'(\bar{\rho})}{\lambda}\partial^2_x (t-s)}\partial_x (nu)(s)\|_{L^2}\\
&\lesssim (t-s)^{-\frac{1}{2}} \|u(s)\|_{L^{2}}^{\frac{1}{2}}\|\partial_x u(s)\|_{L^{\infty}} (\|(\rho^*-\bar{\rho})(s)\|_{L^2}+\|(\rho^*-\bar{\rho})(s)\|_{L^2})\\
&\lesssim (t-s)^{-\frac{1}{2}} s^{-\frac{1}{2}-\frac{\mu}{2}} X(0)\Big( \widetilde{X}(0)+\sup_{\tau\in[0,t]}(1+\tau)^{\frac{\mu}{2}}\|(\rho-\rho^*)(\tau)\|_{L^2}\Big).
\end{aligned}
\end{equation}
For $\var\in (0,\frac{1-\mu}{2})$, we combine \eqref{duhamelerroreuler}, \eqref{imrpoveR*} and \eqref{imrpovedxnu} to get
\begin{equation}
\begin{aligned}
&\|(\rho-\rho^*)(t)\|_{L^2}\leq \Big(t^{-\frac{1}{2}+\var}+t^{-\frac{\mu}{2}}\Big)\Big( \widetilde{X}(0)+\sup_{\tau\in[0,t]}(1+\tau)^{\frac{\mu}{2}}\|(\rho-\rho^*)(\tau)\|_{L^2}\Big).
\end{aligned}
\end{equation}
Since $\rho-\rho^*$ is uniformly bounded in $L^2$,  one has
\begin{equation}
\begin{aligned}
&\sup_{\tau\in[0,t]}(1+\tau)^{\frac{\mu}{2}}\|(\rho-\rho^*)(\tau)\|_{L^2}\\
&\quad\leq  X(0)\widetilde{X}(0)+X(0)\sup_{\tau\in[0,t]}(1+\tau)^{\frac{\mu}{2}}\|(\rho-\rho^*)(\tau)\|_{L^2}.
\end{aligned}
\end{equation}
Together with the fact that $X(0)$ is suitably small, this yields
\begin{equation}\label{improveerrormmm}
\begin{aligned}
&\|(\rho-\rho^*)(t)\|_{L^2}\leq (1+t)^{-\frac{\mu}{2}}\widetilde{X}(0),\quad\quad \frac{1}{2}\leq \mu<1.
\end{aligned}
\end{equation}
By \eqref{rho*decay} and \eqref{improveerrormmm}, $\rho-\bar{\rho}$ has the decay estimate $\eqref{eulerdecay2}_1$ for $1/2\leq \mu<1$.

\smallbreak
\noindent $\bullet$ \emph{Step 5: Decay of $\rho-\bar{\rho}$ for $\mu= 1$}. We now deal with the limit case $\mu=1$. It suffices to improve the $L^2$-rate of $\|(\rho-\rho^*)(t)\|_{L^2}$ to $(1+t)^{-1/2}$. To this end, following the idea in Lemma \ref{lemmaerror}, we consider the system satisfied by $(\rho-\rho^*, u)$:
\begin{equation}\label{Eulerrrr}
\left\{
\begin{aligned}
&\partial_{t}(\rho-\rho^*)+\bar{\rho}\partial_x u=F_{1}-\frac{P'(\bar{\rho})}{\lambda}\partial_x^2\rho^*,\\
&\partial_{t}u +\frac{P'(\bar{\rho})}{\bar{\rho}}\partial_x (\rho-\rho^*) +\lambda u=F_{2}-\frac{P'(\bar{\rho})}{\bar{\rho}}\partial_x\rho^*,\\
&(\rho-\rho^*,u)|_{t=0}=(0,u_{0}).
\end{aligned}
\right.
\end{equation}
Doing $L^2$-energy estimate of \eqref{Eulerrrr} with the time-weight $t$ gives
\begin{equation}\label{weighterroreuler}
\begin{aligned}
&t\|(\rho-\rho^*,u)(t)\|_{L^2}^2+\int_{0}^{t}\tau\|u(\tau)\|_{L^2}^2\,d\tau\\
&\quad\lesssim \int_{0}^{t}\|(\rho-\rho^*,u)(\tau)\|_{L^2}^2\,d\tau\\
&\quad\quad+\int_{0}^{t} \tau\int_{\mathbb{R}}\bigg( (F_{1}-\frac{P'(\bar{\rho})}{\lambda}\partial_x^2\rho^*)(\rho-\rho^*)+(F_{2}-\frac{P'(\bar{\rho})}{\bar{\rho}}\partial_x\rho^*) u \bigg)\,dxd\tau.
\end{aligned}
\end{equation}
To bound the first term on the right-hand side of \eqref{weighterroreuler}, we recall that the $L^2(0,t;L^2)$-estimate of $u$ has been obtained in \eqref{eulerpriorie}, and we take the inner product of \eqref{rhorho*error}  with $-\int^{x}_{-\infty}\int^{y}_{-\infty}(\rho-\rho^*)(z,t)\,dzdy$ such that
\begin{equation}\label{errorL2disEuler}
\begin{aligned}
\int_{0}^{t} \|(\rho-\rho^*)(\tau)\|_{L^2}^2\,d\tau&\lesssim  \int_{0}^{t}(\|\partial_x R^*(\tau)\|_{L^2}^2+\|nu(\tau)\|_{L^2}^2)\,d\tau\lesssim X(0).
\end{aligned}
\end{equation}
In order to bound the second term in \eqref{weighterroreuler}, arguing similarly as in  \eqref{matcalL*}-\eqref{mu1N}, we perform the energy argument on \eqref{parabolicEuler} to obtain
\begin{equation}\label{rho*mu1}
\begin{aligned}
&\sum_{0\leq k'\leq k-1} t^{k'+1}\|\partial_x^{k'}(\rho^*-\bar{\rho})(t)\|_{L^2}^2+\sum_{0\leq k'\leq k-1 }\int_{0}^{t} \tau^{k'+1}\|\partial_x^{k'+1}(\rho^*-\bar{\rho})(\tau)\|_{L^2}^2\,d\tau\\
&\quad\lesssim \| |x|(\rho_0-\bar{\rho})\|_{L^2}^2,
\end{aligned}
\end{equation}
for any $k=1,2,...$. Therefore, taking advantage of \eqref{errorL2disEuler} and \eqref{rho*mu1}, we have
\begin{equation}\nonumber
\begin{aligned}
\int_{0}^{t}\tau \int_{\mathbb{R}} |\partial_x^2\rho^* (\rho-\rho^*)|\,dxd\tau&\leq \int_{0}^{t}(\tau^2 \|\partial_x^2\rho^*(\tau)\|_{L^2}^2+\|(\rho-\rho^*)(\tau)\|_{L^2}^2)\,d\tau\\
&\lesssim \widetilde{X}(0),
\end{aligned}
\end{equation}
and, for some constant $\eta_1>0$ to be chosen later,
\begin{equation}\nonumber
\begin{aligned}
&\int_{0}^{t}\tau \int_{\mathbb{R}} |\partial_x\rho^* u|\,dxd\tau\lesssim \eta_1 \int_{0}^{t}\tau \|u(\tau)\|_{L^2}^2d\tau+\frac{1}{\eta_1}  \int_{0}^{t}\tau  \|\partial_x (\rho^*-\bar{\rho})(\tau)\|_{L^2}^2\,d\tau.
\end{aligned}
\end{equation}
Integrating by parts, we obtain
\begin{equation}\nonumber
\begin{aligned}
&\int_{0}^{t}\tau\int_{\mathbb{R}} F_1 (\rho-\rho^*) \,dxd\tau\\
&\quad=\int_{0}^{t}\tau \int_{\mathbb{R}}\partial_x( (\rho-\rho^*)u) (\rho-\rho^*) \,dxd\tau+\int_{0}^{t}\tau \partial_x( (\rho^*-\bar{\rho})u) (\rho-\rho^*) \,dxd\tau\\
&\quad\leq \int_{0}^{t}\tau \int_{\mathbb{R}} \Big(-\frac{1}{2}u_x (\rho-\rho^*)^2+ \partial_x(\rho^*-\bar{\rho}) u (\rho-\rho^*)+(\rho^*-\bar{\rho}) \partial_x u (\rho-\rho^*) \Big)\,dxd\tau\\
&\quad\leq \sup_{\tau\in[0,t]} \tau \|(\rho-\rho^*)(\tau)\|_{L^2}^2 \int_0^t \|u_x(\tau)\|_{L^{\infty}}d\tau+ \|(\rho-\bar{\rho},\rho^*-\bar{\rho})\|_{L^{\infty}_{t}(L^{\infty})}^2 \tau\|u(\tau)\|_{L^2}^2\\
&\quad+\Big(\sup_{\tau\in[0,t]} \tau \|(\rho-\rho^*)(\tau)\|_{L^2}^2 \Big)^{\frac{1}{2}} \Big(\int_0^t \tau \|\partial_x u(\tau)\|_{L^2}^2\,d\tau\Big)^{\frac{1}{2}}   \Big(\int_0^t \|(\rho -\rho^*)(\tau)\|_{L^2}^2\,d\tau\Big)^{\frac{1}{2}}.
\end{aligned}
\end{equation}
Similarly, one has
\begin{equation}\nonumber
\begin{aligned}
&\int_{0}^{t} \tau\int_{\mathbb{R}} F_2 u\,dxd\tau\\
&\quad\lesssim \|\partial_x u\|_{L^{\infty}_{t}(L^{\infty})}\int_{0}^{t}\tau \|u(\tau)\|_{L^2}^2\,d\tau+\eta_1 \int_{0}^{t}\tau \|u(\tau)\|_{L^2}^2\,d\tau\\
&\quad\quad+\frac{1}{\eta_1} \int_0^t \|\partial_x n(\tau)\|_{L^2}^2\,d\tau \Big(\sup_{\tau\in[0,t]} \tau \|\rho^*-\bar{\rho})(\tau)\|_{L^2}^2+\sup_{\tau\in[0,t]} \tau \|(\rho-\rho^*)(\tau)\|_{L^2}^2\Big).
\end{aligned}
\end{equation}
Substituting the above estimates into \eqref{weighterroreuler} and making use of \eqref{eulerpriorie} and \eqref{errorL2disEuler}, we end up with
\begin{equation}\label{mmdgdgdgg}
\begin{aligned}
&t\|(\rho-\rho^*,u)(t)\|_{L^2}^2+\int_{0}^{t}\tau\|u(\tau)\|_{L^2}^2\,d\tau\\
&\quad\lesssim (1+\frac{1}{\eta_1})\widetilde{X}(0)\\
&\quad\quad+(X(0)+\eta_1)\Big(\sup_{\tau\in[0,t]}\tau\|(\rho-\rho^*,u)(\tau)\|_{L^2}^2+\int_{0}^{t}\tau\|u(\tau)\|_{L^2}^2\,d\tau\Big).
\end{aligned}
\end{equation}
Choosing a suitably small constant $\eta_1$ and recalling that $X(0) \ll 1$, we derive the $(1+t)^{-\frac{1}{2}}$ time-decay estimates of $\|(\rho-\rho^*)(t)\|_{L^2}$. Together with \eqref{rho*decay}, we obtain the decay rate of $\rho-\bar{\rho}$ in $\eqref{eulerdecay2}_1$ with  $\mu=1$.

\smallbreak
\noindent $\bullet$ \emph{Step 6: Decay of $u$ and $\partial_x(\rho-\bar{\rho},u)$}. Since we showed $\eqref{eulerdecay2}_1$ for all $0<\mu\leq 1$, applying Lemma \ref{decayineq} to the Lyapunov inequality \eqref{eulerlyapunov} gives
\begin{equation}
\begin{aligned}
&\|\partial_x(\rho-\bar{\rho},u)(t)\|_{L^2}\lesssim (1+t)^{-\frac{\mu}{2}-\frac{1}{2}}\widetilde{X}(0).\nonumber
\end{aligned}
\end{equation}
This, together with $\eqref{eulerdecay2}_1$ and \eqref{vinq}, gives rise to the desired decay estimate of $u$. We thus finish the proof of Lemma \ref{lemmaeulerdecay1}.

\end{proof}

 \begin{lemma}\label{lemmaeulerdecay2}
Let $(\rho,u)$ be the global solution to the Cauchy problem of the system \eqref{Euler} associated with the initial datum $(\rho_{0},u_{0})$. In addition to \eqref{eulera1}, assume $|x|^{\mu}(\rho_{0}-\bar{\rho})\in L^{2} $,  $|x|^{\mu-1/2}u_{0}\in L^{2} $ for $1/2<\mu\leq 1$, $\partial_t\rho|_{t=0}=-\partial_x(\rho_{0}u_{0})$ and $P'(\bar{\rho})\leq 1$. Then, \eqref{eulerdecay3} holds.
 \end{lemma}

\begin{proof}
In order to apply the method in Section \ref{sec:wavemethod}, a key step is to consider the momentum $m=\rho u=(n+\bar{\rho})u$ instead of the velocity $v$. Then, the system \eqref{Euler} is rewritten, in terms of $(n,m)$, as
\begin{equation}\label{Eulerr2}
\left\{
\begin{aligned}
&\partial_{t}n+\partial_x m=0,\\
&\partial_{t}m +P'(\bar{\rho})\partial_x n +\lambda m= F_{3}:=-\partial_x\Big(\frac{m^2}{\bar{\rho}+n}+P(\bar{\rho}+n)-P(\bar{\rho})-P'(\bar{\rho}) n\Big),\\
&(n,m)|_{t=0}=(\rho_{0}-\bar{\rho},\rho_{0}u_{0}).
\end{aligned}
\right.
\end{equation}
As in Section \ref{sec:wavemethod}, we introduce the wave unknown
\[
M(x,t):=\int_{-\infty}^{x} n(y,t)\,dy
\]
such that
\begin{align}
    &\partial_{t}^2 M-P'(\bar{\rho})\partial_{x}^2 M+\lambda \partial_t M=F_{3}.
\end{align}
Then, following the computation done in the proof of Lemma \ref{lemmaW}, we can show that
\begin{equation}\label{fffff}
    \begin{aligned}
        &\int_{\mathbb{R}} (1+t+|x|)^{2\mu-1}(|\partial_t M|^2+|\partial_x M|^2)\,dx\\
    &\quad\quad+\int_{0}^{t}\int_{\mathbb{R}}\Big( (1+t+|x|)^{2\mu-1} |\partial_t M |^2+(1+t+|x|)^{2\mu-2} |\partial_x M |^2\Big)\,dxd\tau \\
        &\quad\lesssim \widetilde{Y}_{0}^2 +\int_{0}^{t}\int_{\mathbb{R}}(1+t+|x|)^{2\mu-1}|F_{3}|^2\,dxd\tau.
    \end{aligned}
\end{equation}
From \eqref{eulerpriorie} and composition estimates, we obtain
\begin{equation}\nonumber
    \begin{aligned}
    &|F_{3}|\lesssim |m| |\partial_x m|+ |n||\partial_x n|.
    \end{aligned}
\end{equation}
It thus follows that
\begin{equation}
    \begin{aligned}
    &\int_{0}^{t}\int_{\mathbb{R}}(1+t+|x|)^{2\mu-1}|F_{3}|^2\,dxd\tau\\
    &\quad \lesssim \int_{0}^{t}\|\partial_x(m,n)(\tau)\|_{L^{\infty}}^2\int_{\mathbb{R}}(1+t+|x|)^{2\mu-1}(|m|^2+|n|^2)\,dxd\tau \\
    &\quad\lesssim \int_{0}^{t}\|\partial_x(n,u)(\tau)\|_{H^1}^2\int_{\mathbb{R}}(1+t+|x|)^{2\mu-1}(|\partial_t M|^2+|\partial_x M|^2)\,dxd\tau,
    \end{aligned}
\end{equation}
where one has used the facts that $\partial_x M=n$ and $\partial_t M=-m$. Inserting the above estimate into \eqref{fffff} and using \eqref{eulerpriorie} and Gr\"onwall's inequality, we get
\begin{equation}\label{fffff1}
    \begin{aligned}
        &\int_{\mathbb{R}} (1+t+|x|)^{2\mu-1}(|n|^2+|m|^2) \,dx+\int_{0}^{t}\int_{\mathbb{R}} (1+t+|x|)^{2\mu-2} |m |^2\,dxd\tau \lesssim \widetilde{Y}_{0}^2,
    \end{aligned}
\end{equation}
which implies $\eqref{eulerdecay3}_1$. Finally, in view of $\eqref{eulerdecay3}_1$,  \eqref{vinq} and Lemma \ref{decayineq} to the Lyapunov inequality  \eqref{eulerlyapunov}, we are able to show the faster decay of $v$ and $\partial_x(\rho-\bar{\rho},u)$ in $\eqref{eulerdecay3}_2$. The proof of Lemma \ref{lemmaeulerdecay2} is complete.
\end{proof}

\section{Proof of Theorem \ref{ThmEulerNLD}} \label{sec:EulerNonlinD}

\subsection{Global existence for the $p$-system with nonlinear damping}
In this section, we prove the global existence of the nonlinearly damped $p$-system \eqref{syst:EulerNLD}. For brevity, we omit the details concerning the local well-posedness of solutions to \eqref{syst:EulerNLD} subject to initial data in $H^1 $ since it can be proved by standard iteration arguments, see e.g. \cite{HJR,BHN,Serre,Majdalocal,XK1}. To extend the local solution to a global one, we establish the uniform a-priori estimates as follows.
\smallbreak
\noindent $\bullet$ $L^2$ estimates:
Standard energy estimates lead to
\begin{align}
&\dfrac{d}{dt}\| (\rho,u)(t)\|_{L^2}^2+2\|u(t)\|_{L^{r+1}}^{r+1}=0.\label{L2}
\end{align}

\noindent $\bullet$ $H^1$ estimates: From direct energy estimates in \eqref{syst:EulerNLD}, we get
\begin{align}
  \dfrac{d}{dt}\|(\partial_x \rho,\partial_x u)(t)\|_{L^2}^2 +r\int_{\mathbb{R}}|u|^{r-1}(\partial_x u)^2\,dx=0,\label{H1}
\end{align}
where we used that
\begin{equation}\nonumber
\begin{aligned}
\int_{\mathbb{R}} \partial_x(|u|^{r-1}u)\partial_x u \,dx&=\int_{\mathbb{R}} \partial_x|u|^{r-1} u\partial_x u\,dx+\int_{\mathbb{R}} |u|^{r-1}(\partial_x u)^2\,dx\\
&=r\int_{\mathbb{R}}|u|^{r-1}(\partial_x u)^2\,dx.
\end{aligned}
\end{equation}
\noindent $\bullet$ Dissipation for $\partial_x u$: Multiplying $\eqref{syst:EulerNLD}_{2}$ by $|u|^{r-1}\partial_x \rho$, we infer
\begin{align}
&\int_{\mathbb{R}} \partial_t u |u|^{r-1}\partial_x \rho\,dx+\int_{\mathbb{R}} |u|^{r-1}|\partial_x \rho|^2 \,dx-\int_{\mathbb{R}} |u|^{2r-2} u \partial_x \rho\,dx=0.\label{1}
\end{align}
Similarly, from $\eqref{syst:EulerNLD}_{1}$, we get
\begin{align}
&\int_{\mathbb{R}} \frac{1}{r}|u|^{r-1} u \partial_x \partial_t \rho \,dx+\int_{\mathbb{R}} \frac{1}{r}|u|^{r-1} u \partial_{x}^2 u \,dx=0.\label{2}
\end{align}
By \eqref{1} and \eqref{2}, the fact that $\partial_t u  |u|^{r-1} =\partial_t (|u|^{r-1} u)/r$ and integration by parts, we obtain
\begin{align}
&\frac{d}{dt}\int_{\mathbb{R}} \frac{1}{r} |u|^{r-1} u \partial_x \rho \, dx+\int_{\mathbb{R}} ( |u|^{r-1}|\partial_x \rho|^2- |u|^{2r-2} u \partial_x \rho- |u|^{r-1} |\partial_x u|^2) \, dx=0.\label{Hcross}
\end{align}
Defining 
\begin{align}
\mathcal{W}_{*}(t)&:=\|(\rho,u,\partial_x \rho, \partial_x u)(t)\|_{L^2}^2+\eta_{2}\int \frac{1}{r} |u|^{r-1} u \partial_x \rho \,dx,\nonumber\\
\mathcal{H}_{*}(t)&:=2\int_{\mathbb{R}} |u|^{r-1} (|u|^2+|\partial_x u|^2)\,dx\nonumber\\
&\quad+\eta_{2} \int_{\mathbb{R}} ( |u|^{r-1}|\partial_x \rho|^2- |u|^{2r-2} u \partial_x \rho- |u|^{r-1} |\partial_x u|^2) \,dx,\nonumber
\end{align}
we obtain from \eqref{L2}, \eqref{H1} and \eqref{Hcross} that
\begin{align}
&\frac{d}{dt}\mathcal{W}_{*}(t)+\mathcal{H}_{*}(t)=0.\label{Lnon}
\end{align}
Since due to \eqref{L2}, \eqref{H1} and the Gagliardo-Nirenberg inequality,
 \[
 \|u\|_{L^{\infty}_{t}(L^{\infty}_{x})}\lesssim \|u\|_{L^{\infty}_{t}(L^2_x)}^{\frac{1}{2}}\|\partial_x u\|_{L^{\infty}_{t}(L^{\infty}_{x})}^{\frac{1}{2}}\lesssim 1,
 \]
we are able to choose a suitably small constant $\eta_{2}>0$ such that
\begin{equation}
\left\{
\begin{aligned}
&\mathcal{W}_{*}(t)\sim \|(\rho,u,\partial_x \rho, \partial_x u)(t)\|_{L^2}^2,\\
&\mathcal{H}_{*}(t)\gtrsim \int_{\mathbb{R}} |u|^{r-1} (|u|^2+|\partial_x \rho|^2+|\partial_x u|^2)\,dx .\label{sim}
\end{aligned}
\right.
\end{equation}
Integrating \eqref{Lnon} over $[0,t]$ and making use of \eqref{sim}, we have
\begin{equation}\nonumber
\begin{aligned}
&\| (\rho,u)(t)\|_{H^1}^2+\int_{0}^{t}\|u(\tau)\|_{L^{r+1}}^{r+1}+\|(\partial_x \rho^{(r+1)/2},\partial_x u^{(r+1)/2})(\tau)\|_{L^2}^2d\tau\lesssim \| (\rho_{0},u_{0})\|_{H^1}^2.
\end{aligned}
\end{equation}
The above estimates enable us to prove the global existence of the solution to \eqref{syst:EulerNLD} with a standard bootstrap argument.

\subsection{Wave formulation}\label{subsectionwaven}

Differentiating $\eqref{syst:EulerNLD}_{1}$ and $\eqref{syst:EulerNLD}_{2}$ with respect to $t$, we rewrite $\eqref{syst:EulerNLD}$ into two damped wave-like equations
\begin{equation}\label{wave}
\partial_{t}^2 \rho-\partial_{x}^2 \rho+r|u|^{r-1}\partial_t \rho=0,\quad
\partial_{t}^2 u-\partial_{x}^2 u+r|u|^{r-1} \partial_t u=0.
\end{equation}
From the equation $\eqref{syst:EulerNLD}_{1}$ it follows that
\begin{align}
u=-\partial_t \int_{-\infty}^{x} \rho(y,t)\,dy,\label{v}
\end{align}
from which we infer that
\begin{equation}\nonumber
    \begin{aligned}
        r|u|^{r-1}\partial_t \rho&=r\bigg|\partial_t \int_{-\infty}^{x} \rho(y,t)\,dy\bigg|^{r-1}\, \partial_x  \partial_t \int_{-\infty}^{x} \rho(y,t)\,dy\\
        &= \partial_x \bigg(\bigg|\partial_t \int_{-\infty}^{x} \rho(y,t)\,dy\bigg|^{r-1} \partial_t \int_{-\infty}^{x} \rho(y,t)\,dy  \bigg).
    \end{aligned}
\end{equation}
Thus, defining the new unknown $w$ by
\begin{align}
w:=\int_{-\infty}^{x} u(y,t)\,dy\nonumber
\end{align}
and integrating the equation $\eqref{wave}_{1}$ over $(-\infty,x)$, we obtain the nonlinearly damped wave equation:
\begin{align}
&\partial_{t}^2w-\partial_{x}^2 w+|\partial_{t}w|^{r-1} \partial_t w=0.\label{dampedwave}
\end{align}
From $\eqref{syst:EulerNLD}_{1}$, we see that
\begin{equation}\nonumber
\begin{aligned}
&\|\partial_t w(t)\|_{L^2}^2=\|u(t)\|_{L^2}^2,\quad\quad \|\partial_x w(t)\|_{L^2}^2=\|\rho(t)\|_{L^2}^2.
\end{aligned}
\end{equation}
Thus, once we get the decay rate of $\|(\partial_t w,\partial_x w)(t)\|_{L^2}^2$, the $L^2$ decay of $ (\rho,u)$ follows.

\subsection{Asymptotic estimates}
In this subsection, we prove Theorem \ref{ThmEulerNLD} and derive logarithmic time-decay rates for the solutions to the system \eqref{syst:EulerNLD}. To that matter, we capture the nonlinear dissipative structures in \eqref{dampedwave} by adapting the method developed by Theorem \ref{ThmDecayGeneral2} and the work of  Mochizuki and Motai in \cite{MochizukiMotai}. A key ingredient is the $L^2$ coercive estimates for $\partial_t w$ that we derive from the damped term $|\partial_t w|^{r-1}\partial_t w$ with suitable weights (see \eqref{qqq3} below).
\smallbreak
Let two weight functions $\varphi_{1}(s), \varphi_{2}(s)$ for $s\geq0$ to be determined later. Taking the $L^2$ inner product of $\eqref{dampedwave}$ with $\varphi_{1} (t+|x|)\partial_{t}w$, we obtain
\begin{equation}\label{lb1}
\begin{aligned}
&\frac{d}{dt}\int_{\mathbb{R}} \frac{1}{2}\varphi_{1} (t+|x|) (|\partial_t w|^2+|\partial_x w|^2) \,dx\\
&\quad+\int_{\mathbb{R}} \Big(\varphi_{1} (t+|x|)|\partial_t w|^{r}-\frac{1}{2}\varphi_{1}' (t+|x|)|\partial_x w|^2\Big) dx=\int_{\mathbb{R}}\frac{1}{2}\varphi_{1}' (t+|x|) |\partial_t w|^2\,dx.
\end{aligned}
\end{equation}
In addition, multiplying $\eqref{dampedwave}$ by $\varphi_{1}' (t+|x|) w$ and integrating by parts, we obtain
\begin{equation}\label{lb2}
\begin{aligned}
&\frac{d}{dt}\int_{\mathbb{R}} \Big( \varphi_{1}' (t+|x|) w \partial_t w-\frac{1}{2}\varphi_{1}'' (t+|x|)|w|^2 \Big) \,dx\\
&\quad+\int_{\mathbb{R}}\Big(\varphi_{1}'(t+|x|)|\partial_x w|^2-\varphi_{1}'' (t+|x|)\delta_{0}(x)|w|^2\Big)\,dx\\
&=\int_{\mathbb{R}} \Big(\varphi_{1}'(t+|x|)|\partial_t w|^2- \varphi_{1}' (t+|x|)|\partial_t w|^{r-1}\partial_t w w\Big)\,dx,
\end{aligned}
\end{equation}
where $\delta_{0}(x)$ denotes the Dirac function at $0$ and we have used that
\begin{equation}\nonumber
\begin{aligned}
\int_{\mathbb{R}} \varphi_{1}' (t+|x|)  w \partial^2_t w \,dx&=\frac{d}{dt}\int_{\mathbb{R}} \Big(\varphi_{1}' (t+|x|)  w \partial_t w-\frac{1}{2} \varphi_{1}'' (t+|x|)|w|^2\Big)\, dx\\
&\quad+\int_{\mathbb{R}}\Big(\frac{1}{2} \varphi_{1}'''(t+|x|)|w|^2-\int_{\mathbb{R}}\varphi_{1}'(t+|x|)|\partial_t w|^2\Big)\,dx,
\end{aligned}
\end{equation}
and
\begin{equation}\nonumber
\begin{aligned}
&-\int_{\mathbb{R}} \varphi_{1}' (t+|x|)  w \partial_{x}^2 w\, dx\\
&\quad=\int_{\mathbb{R}}\bigg( \varphi_{1}'(t+|x|) |\partial_x w|^2 -\Big( \frac{1}{2} \varphi_{1}''' (t+|x|)+ \varphi_{1}'' (t+|x|)\delta_{0}(x)\Big)|w|^2 \bigg) \, dx.
\end{aligned}
\end{equation}
To control the second term on the right-hand side of \eqref{lb2}, one needs to capture the dissipation of $|w|^{r+1}$ with a suitable weight. To this matter, a direct calculation yields
\begin{equation}\label{lb3}
\begin{aligned}
&\frac{d}{dt}\int_{\mathbb{R}}\varphi_{2}(t+|x|)|w|^{r+1}\,dx-\int_{\mathbb{R}}\varphi'_{2}(t+|x|)|w|^{r+1} \,dx\\
&=-\int_{\mathbb{R}} (r+1)\varphi_{2}(t+|x|)|w|^{r-1}w \partial_t w \, dx.
\end{aligned}
\end{equation}
For some small constant $\eta_{3}>0$ to be chosen later, we define
\begin{equation}\nonumber
\begin{aligned}
\mathcal{W}_{\varphi}(t):&=\int_{\mathbb{R}}\frac{1}{2}\varphi_{1} (t+|x|) (|\partial_t w|^2+|\partial_x w|^2)\,dx\\
&\quad+\eta_{3}\int_{\mathbb{R}} \Big( \varphi_{1}' (t+|x|) w \partial_t w-\frac{1}{2}\varphi_{1}'' (t+|x|)|w|^2+\varphi_{2}(t+|x|)|w|^{r+1}\Big) \,dx,
\end{aligned}
\end{equation}
and
\begin{equation}\nonumber
\begin{aligned}
\mathcal{H}_{\varphi}(t):&=\int_{\mathbb{R}} \varphi_{1} (\tau+|x|)|\partial_t w|^{r} \,dx\\
&\quad+\eta_{3}\int_{\mathbb{R}} \Big(\varphi_{1}'(t+|x|)|\partial_x w|^2-\varphi_{1}'' (t+|x|)\delta_{0}(x)|w|^2-\varphi'_{2}(t+|x|)|w|^{r+1} \Big)\, dx.
\end{aligned}
\end{equation}
Thus, from \eqref{lb1}, \eqref{lb2} and \eqref{lb3}, we get
\begin{equation}
\begin{aligned}
&\frac{d}{dt}\mathcal{W}_{\varphi}(t)+\mathcal{H}_{\varphi}(t)\\
&=\int_{\mathbb{R}}\Big(\frac{1+\eta_{3}}{2}\varphi_{1}' (t+|x|)|\partial_t w|^2-\eta_{3}\varphi_{1}' (t+|x|)|\partial_t w|^{r-1}\partial_t w w\Big)\,dx\\
&\quad -\eta_{3}(r+1)\int_{\mathbb{R}}\varphi_{2}(\tau+|x|)|w|^{r-1}w \partial_t w \Big) \,dx.\label{LNL}
\end{aligned}
\end{equation}
In order to control $\mathcal{W}_{\varphi}(t), \mathcal{H}_{\varphi}(t)$ and derive the desired dissipation estimates, we require 
\begin{align}
&\varphi_{1}>0,\quad \varphi_{1}'>0,\quad\varphi_{1}''<0,\quad |\varphi_{1}'|^2\leq C\varphi_{1}|\varphi_{1}''|,\quad \varphi_{2}>0,\quad\varphi'_{2}<0.\label{varphicondition1}
\end{align}
Indeed, under the condition \eqref{varphicondition1}, one has
\begin{equation}\label{qqq1}
\begin{aligned}
&\int_{\mathbb{R}} \varphi_{1}' (t+|x|) |w \partial_t w| \,dx\\
&\quad\leq  C\int_{\mathbb{R}}\varphi_{1} (t+|x|)|\partial_t w|^2 \,dx-\frac{1}{4}\int_{\mathbb{R}} \varphi_{1}'' (t+|x|)|w|^2 \,dx,
\end{aligned}
\end{equation}
which implies
\begin{equation}\label{qqq10}
\begin{aligned}
\mathcal{W}_{\varphi}(t)&\geq \int_{\mathbb{R}}(\frac{1}{2}-C\eta_{3})\varphi_{1} (t+|x|) (|\partial_t w|^2+|\partial_x w|^2)\,dx\\
&\quad+\int_{\mathbb{R}}\Big(-\frac{\eta_{3}}{4}\varphi_{1}'' (t+|x|)|w|^2+\eta_{3}\varphi_{2}(t+|x|)|w|^{r+1}\Big) \,dx.
\end{aligned}
\end{equation}
Moreover, \eqref{varphicondition1} also leads to
\begin{equation}\label{qqq101}
\begin{aligned}
\mathcal{H}_{\varphi}(t)&\geq\int_{\mathbb{R}}  \varphi_{1} (t+|x|)|\partial_t w|^{r}\,dx\\
&\quad+\eta_{3} \int_{\mathbb{R}} \Big(\varphi_{1}'(t+|x|)|\partial_x w|^2-\eta_{3}\varphi'_{2}(t+|x|)|w|^{r+1} \Big)\, dx.
\end{aligned}
\end{equation}
We now focus on the estimation of the right-hand side terms \eqref{LNL}. First, a use of Young's inequality gives
\begin{align}
&\frac{3}{2}\int_{0}^{t}\int_{\mathbb{R}}\varphi_{1}' (\tau+|x|)|\partial_t w|^2\,dxd\tau \nonumber\\
&\quad\leq\frac{3}{2}\Big(\int_{0}^{t}\int_{\mathbb{R}} \varphi_{1}(\tau+|x|)|\partial_t w|^{r+1}\,dxd\tau\Big)^{\frac{2}{r+1}}\Big(\int_{0}^{t} \int_{\mathbb{R}} \frac{|\varphi_{1}'(\tau+|x|)|^{\frac{r+1}{r-1}}}{\varphi_{1}(\tau+|x|)^{\frac{2}{r-1}}} \,dxd\tau \Big)^{\frac{r-1}{r+1}}\nonumber\\
&\quad\leq \frac{1}{4}\int_{0}^{t}\int_{\mathbb{R}} \varphi_{1} (\tau+|x|)|\partial_t w|^{r+1}\,dxd\tau+C\int_{0}^{t}\int_{\mathbb{R}} \frac{|\varphi_{1}'(\tau+|x|)|^{\frac{r+1}{r-1}}}{\varphi_{1}(\tau+|x|)^{\frac{2}{r-1}}} \,dxd\tau.\label{qqq3}
\end{align}
Similarly,  we infer
\begin{equation}\label{qqq4}
\begin{aligned}
\int_{0}^{t}\int_{\mathbb{R}} \varphi_{1}' (\tau+|x|)|\partial_t w|^{r}|w|\, dxd\tau
&\leq \frac{1}{4}\int_{0}^{t}\int_{\mathbb{R}} \varphi_{1} (\tau+|x|)|\partial_t w|^{r+1}\,dxd\tau\\
&+C\int_{0}^{t} \int_{\mathbb{R}}\frac{|\varphi_{1}' (\tau+|x|)|^{r+1}}{\varphi_{1} (\tau+|x|)^{r}}  |w|^{r+1}\,dxd\tau,
\end{aligned}
\end{equation}
and
\begin{equation}\label{qqq5}
\begin{aligned}
\hspace{-0.2cm}(r+1)\int_{0}^{t}\int_{R} \varphi_{2}(\tau+|x|)|w|^{r}|\partial_t w|\, dxd\tau
&\leq \frac{1}{4}\int_{0}^{t}\int_{\mathbb{R}} \varphi_{1} (\tau+|x|)|\partial_t w|^{r+1}\,dxd\tau
\\&+C\int_{0}^{t} \int_{\mathbb{R}}\frac{|\varphi_{2}(\tau+|x|)|^{r+1}}{\varphi_{1} (\tau+|x|)^{r}}  |w|^{r+1}\,dxd\tau.
\end{aligned}
\end{equation}
Let $\eta_{3}=1/(4C)$. Then it follows from \eqref{qqq1}-\eqref{qqq5} that
\begin{align}
 & \int_{\mathbb{R}}\bigg(\frac{1}{4}\varphi_{1} (t+|x|) (|\partial_t w|^2+|\partial_x w|^2)-\frac{1}{4}\varphi_{1}'' (t+|x|)|w|^2+\varphi_{2}(t+|x|)|w|^{r+1}dx\bigg)\, dx\nonumber\\
 &\quad\quad+\int_{0}^{t}\int_{\mathbb{R}}\Big(\frac{1}{2}\varphi_{1} (\tau+|x|)|\partial_t w|^{r}+\frac{1}{2}\varphi_{1}' (\tau+|x|)|\partial_x w|^2+\mathcal{C}_{1}(\tau+|x|)|w|^{r+1}\Big)\,dxd\tau\nonumber\\
 &\leq \mathcal{W}_{\varphi}(0)+ C\int_{0}^{t}\int_{\mathbb{R}} \mathcal{C}_{2}(\tau+|x|)\,dxd\tau,\label{phimm}
\end{align}
with
\begin{equation}\nonumber
\begin{aligned}
&\mathcal{C}_{1}(s):=-\varphi'_{2}(s)-\frac{C(|\varphi_{1}' (s)|^{r+1}+|\varphi_{2}(s)|^{r+1})}{\varphi_{1} (s)^{r}},\quad\quad \mathcal{C}_{2}(s):=\frac{|\varphi_{1}'(s)|^{\frac{r+1}{r-1}}}{\varphi_{1}(s)^{\frac{2}{r-1}}} .
\end{aligned}
\end{equation}
Therefore, one needs to choose $\varphi_{1}$ and $\varphi_{2}$ such that
\begin{equation}\label{varphicondition2}
\begin{aligned}
\mathcal{C}_{1}(s)>0,\quad\quad \int_{0}^{\infty}\int_{\mathbb{R}} \mathcal{C}_{2}(\tau+|x|)\,dxd\tau<\infty.
\end{aligned}
\end{equation}
For all $q>0$ and a suitable large constant $a$, we choose the functions
\begin{align}
\varphi_{1}(s)=\log^{2q}{(a+s)},\quad\quad \varphi_{2}(s)=\frac{\log^{2q-r+1}{(a+s)}}{|a+s|^{r}},\nonumber
\end{align}
which fulfill the conditions \eqref{varphicondition1} and \eqref{varphicondition2}. Indeed, for suitable large $a>0$, it is easy to verify that
\begin{align}
&\varphi_{1}'(s)=\frac{2q \log^{2q-1}{(a+s)}}{a+s}>0,\nonumber\\
&\varphi_{1}''(s)=-\frac{ 2q\log^{2q-2}{(a+s)}(\log{(a+s)}-1+2q)}{|a+s|^2}<0,\nonumber \\
&|\varphi_{1}'|^2\leq \frac{1}{4}\varphi_{1}|\varphi_{1}''|,\nonumber\\
&\varphi’_{2}(s)=-\frac{\log{(a+s)}^{2q-r}(r\log{(a+s)}-2q+r-1)}{|a+s|^{r+1}}<0,\nonumber
\end{align}
and
\begin{equation}\nonumber
\begin{aligned}
\mathcal{C}_{1}(s)&\geq \frac{\log{(a+s)}^{2q-r}}{|a+s|^{r+1}}\bigg(r\log{(a)}-2q+r-1\\
&\quad\quad\quad-C\Big((2q)^{r+1}\log^{-1}{(a)}-C\log^{r-r^2+1}{(a)} \Big)a^{-r^2+1} \bigg)>0.
\end{aligned}
\end{equation}
 The condition $1<r<3$ comes into play for the second term. It implies $(r+1)/(r-1)>2$, and therefore
\begin{equation}\label{conditionr3}
\begin{aligned}
\int_{0}^{\infty}\int_{\mathbb{R}} \mathcal{C}_{2}(\tau+|x|)\,dxd\tau&=\int_{0}^{\infty}\int_{\mathbb{R}} \frac{\log^{2q-\frac{r+1}{r-1}}{(a+\tau+|x|)}}{|a+\tau+|x||^{\frac{r+1}{r-1}}}\,dxd\tau<\infty.
\end{aligned}
\end{equation}
By \eqref{varphicondition1}, \eqref{varphicondition2} and the facts that $\partial_x w=\rho$, $\partial_tw=-u$, substituting $\varphi_{1}(s)=\log{(a+s)}^{2q}$ into \eqref{phimm}, we obtain
\begin{equation}\nonumber
\begin{aligned}
\log^{2q}{(1+t)}\int_{\mathbb{R}}(|\rho|^2+|u|^2)\, dx&\leq \int_{\mathbb{R}}\varphi_{1}(t+|x|)(|\partial_t w|^2+|\partial_x w|^2)\, dx\\
&\leq C\mathcal{W}_{\varphi}(0)+C.
\end{aligned}
\end{equation}
 To bound the initial energy $\mathcal{W}_{\varphi}(0)$, we use \eqref{NDw} and find that
\begin{equation}\nonumber
\begin{aligned}
\mathcal{W}_{\varphi}(0)&\leq C\int_{\mathbb{R}}\log^{2q}{(1+ |x|)} (|\rho_{0}|^2+|u_{0}|^2)\, dx\\
&\quad+C\int_{\mathbb{R}}\frac{\log^{2q-2}{(1+ |x|)}}{(1+|x|)^2}|\int^{x}_{-\infty}\rho_{0}(y)\,dy|^2\, dx \\
&\quad+C\int_{\mathbb{R}}\frac{\log^{2q-r-1}(1+|x|)}{(1+ |x|)^{r}}|\int^{x}_{-\infty}\rho_{0}(y)\,dy|^{r+1}\, dx\\
&\leq C+C\|\rho_{0}\|_{L^1}^2\int_{\mathbb{R}}\frac{1}{(1+|x|)^{2-\eta}}\,dx +C\|\rho_{0}\|_{L^1}^{r+1}\int_{\mathbb{R}}\frac{1}{(1+ |x|)^{r-\eta}}\,  dx <\infty,
\end{aligned}
\end{equation}
for some sufficiently small $\eta\in (0,\min\{1,r-1\})$. Gathering the last two estimates, we obtain \eqref{logdecay} which concludes the proof of Theorem \ref{ThmEulerNLD}.
\qed

\section{Extensions and open problems} \label{sec:ext}
We have analyzed the time-asymptotic behaviour of general hyperbolic systems without Fourier analysis on the real line. Our work opens up several possible extensions and open problems. We list some of them below.

\begin{enumerate}
    \item[1.] \emph{The compressible Euler system with nonlinear damping.}
As an extension of Section \ref{sec:EulerNonlinD}, one can consider the Euler system \eqref{Euler} with a nonlinear damping $\rho |u|^{r-1}u$ with $r>1$, a relevant model for gas transport, see e.g.\cite{EggerGiesselmann}.
Following the approach used in Section \ref{sec:EulerLin}, similar decay rates to the one obtained for the nonlinearly damped $p$-system \eqref{syst:EulerNLD} in Theorem \ref{ThmEulerNLD} can be derived if one can construct a solution of the nonlinearly damped Euler system belonging to $L^2(\mathbb{R}_{+};\dot{H}^1 )$. However, in contrast with linear damping, the nonlinear damping $\rho |u|^{r-1}u$ does not seem sufficient, to the best of our knowledge, to control the advection terms and, thus, to ensure the existence of a unique global-in-time solution.

  \item[2.] \emph{Numerics.}
As an application of the method developed here, inspired by  \cite{PorrettaZuazua2016}, one can prove that a centered finite-difference approximation of the partially dissipative system \eqref{GE} in the whole space preserves the asymptotic properties of the continuous solutions as $t\to\infty$. Such result would highlight that the hyperbolic hypocoercive nature of the system can be preserved at the semi-discrete level.


\item[3.] \emph{Multi-dimensional setting.}
In the multi-dimensional setting, one can investigate the $n$-component systems in $\mathbb{R}^d$ ($d\geq1$) of the type:
\begin{equation}
\frac{\partial V}{\partial t} + \sum_{j=1}^dA^j\frac{\partial V}{\partial x_j}=BV, \label{GEQSYM}
\end{equation}
where $A^j$ ($j=1,..,d)$ are symmetric matrices, $B$ is  symmetric satisfying \eqref{BD} and the unknown $V=V(x,t)\in\mathbb{R}^n$ depends on the time and space variables
$(x,t)\in \mathbb{R}^d\times \mathbb{R}_+$.
The hyperbolic hypocoercivity approach in \cite{BZ}, presented in the one-dimensional setting here, can be extended to the multi-dimensional case, and similar time-decay rates can be recovered under the Kalman rank condition. However,  our approach does not allow to consider multi-dimensional systems of the general form  \eqref{GEQSYM}. The issue comes from the appearance of mixed derivatives (due the multi-dimensional setting) when differentiating in time the low-order corrector term in the Lyapunov functional, and it is unclear how to handle them without Fourier analysis.

Nevertheless, it is possible to obtain results under additional structural conditions on \eqref{GEQSYM}: for instance, when \eqref{GEQSYM} has a structure similar to the multi-dimensional compressible Euler system with damping. Indeed, for the multi-dimensional version of \eqref{Euler}, straightforward computations show that the Lyapunov functional
\[\mathcal{L}(t)= \|(\rho-\bar{\rho},u)(t)\|_{H^1}^2+ \eta t(\nabla\rho,\nabla u)(t)\|_{L^2}^2+\int_{\mathbb{R}^d} u\cdot\nabla \rho\,dx \]
allows to recover time-decay rates in any dimension. 
\end{enumerate}

\section{Appendix}

\subsection{Fourier analysis of partially dissipative hyperbolic systems}\smallbreak 
\label{sec:appedixSK}
\begin{lemma}[\cite{BZ}] \label{appendix1}
Let $U_0\in L^1\cap L^2$, $A$ be a symmetric matrix and $B$ a matrix satisfying \eqref{BD} and \eqref{Strong Dissipativity}. Then the following assertions are equivalent{\textrm:}
\begin{itemize}
\item The pair $(A,B)$ satisfies the Kalman rank condition \eqref{K}.
    \item The solution $U$ of \eqref{SystGen} satisfies
    \begin{equation}\label{decaybf}
    \left\{
    \begin{aligned}
 &\| U^\ell\|(t) \Vert_{L^{\infty}} \leq C t^{-\frac{1}{2}}\|U_0\|_{^1},\\
 & \|U^h(t) \|_{L^2} \leq C e^{-\gamma_{*} t}  \| U_0\|_{L^2},\\
&\|U(t)\|_{L^2} \leq C t^{-\frac{1}{4}} \|U_0 \|_{L^2},
\end{aligned}
\right.
\end{equation}
with $U^\ell(\xi,t):=\widehat{U}(\xi,t)\mathds{1}_{\{|\xi|<1\}}$ and $U^h(\xi,t):=\widehat{U}(\xi,t)\mathds{1}_{\{|\xi|>1\}}$,  where $\xi\in\mathbb{R}$ is the frequency parameter, and $C,\gamma_{*}$ are positive constants depending only on $A$ and $B$.
\end{itemize}
In addition, if $U_0\in H^1$ and $(A,B)$ satisfies the Kalman rank condition,  we have
\begin{align}\label{decayH1}
\|U(t)\|_{\dot{H}^1} \leq C t^{-\frac{3}{4}} \Vert U_0 \Vert_{H^1}.
\end{align}
\end{lemma}

\textit{Sketch of the Proof of Lemma \ref{appendix1}.} According to Proposition \ref{PropBZ}, one introduces the following Lyapunov functional in the Fourier space
 \begin{equation}\nonumber
  \mathcal{L}_{\xi}(t)\triangleq |\widehat{U}|^2+\min\Big\{ \frac{1}{|\xi|},|\xi| \Big\} {\textrm{Re}} \sum_{k=1}^{n-1} \epsilon_k \langle BA^{k-1}\widehat{U}\cdot BA^k\widehat{U}\rangle ,
  \end{equation}
where $\langle\,\cdot\,\rangle$ designates the Hermitian scalar  product in $\mathbb{C}^{n}$.
For $\epsilon_{0}=\kappa_{0}/2$ and suitably small coefficients $\epsilon_k$ ($k=1,2,...,n-1)$, one deduces from \eqref{ODE} that
\[
\dfrac{d}{dt}\mathcal{L}_{\xi}(t)+\min\{1,|\xi|^2\}\sum_{k=0}^{n-1} \epsilon_k |B A^k\widehat{U}|^2\leq 0
\quad \text{and} \quad\mathcal{L}_{\xi}(t)\sim |\widehat{U}|^2.\]
Then, the Kalman rank condition \eqref{K} for $(A,B)$ implies that 
\[
|\widehat{U}(\xi,t)|^2\lesssim |\widehat{U_{0}}(\xi)|^2e^{-N_{*}\min\{1,|\xi|^2\}t},
\]
with
\[
N_{*}:=\inf\Big\{ \sum_{k=0}^{n-1}\epsilon_k |B A^k y|^2,~y\in\mathbb{S}^{n-1}\Big\}>0,
\]
and a low-high frequency splitting argument allows to conclude the classical time-decay rates \eqref{decaybf} and \eqref{decayH1}. 
\begin{remark}
In  Lemma \ref{appendix1}, the solutions can achieve the decay rates $(1+t)^{-\frac{1}{2}(\frac{1}{q}-\frac{1}{2})}$ in $L^2$ under more general $L^{q}$ assumptions with $q\in[1,2)$. The $L^1$ assumption on the initial data can be replaced by a $\dot{B}^{-\frac{1}{2}}_{2,\infty}$ assumption, cf. \cite{CBD2,XK2,XK1D}. Moreover, these decay rates are optimal in the sense that they follow the ones of the heat equation, which is expected by the low frequencies (the slowly-decaying part) of the solution.
\end{remark}

\subsection{Technical lemmas}

\begin{lemma}$($Caffarelli-Kohn-Nirenberg inequality$)$ {\textrm(}\cite{CKN,CW}{\textrm)}.  For all $h\in\mathcal{C}_{c}(\mathbb{R})$, it holds that
    \begin{align}
    &\| |x|^{\mu_{1}} h\|_{L^{p}}\leq C_{\mu_{1},\mu_{2}} \| |x|^{\mu_{2}}\partial_x h\|_{L^2},\label{CKNbest}
    \end{align}
    with $\mu_{2}>\frac{1}{2}$, $\mu_{2}-1\leq \mu_{1}\leq \mu_{2}-\frac{1}{2}$ and $p=\frac{2}{2(\mu_{2}-\mu_{1})-1}$. If $\mu_{1}=\mu_{2}-1$, then $p=2$ and the best constant in \eqref{CKNbest} is 
    \[
    C_{\mu_{1},\mu_{2}}=C_{\mu_{2}-1,\mu_{2}}=\frac{2}{2\mu_{2}-1}.
    \]

\end{lemma}

\begin{lemma}\label{decayineq}
Let $T>0$ be given time, and $E_{1}(t), E_{2}(t)$ be two nonnegative and absolutely continuous functions on $[0,T)$. Suppose that
\begin{align}
E_1(t)\leq a_1 t^{-\alpha},\label{decayE1}
\end{align}
and
\begin{align}
\frac{d}{dt}\Big(E_{1}(t)+ \eta_{0} t E_{2}(t) \Big)+a_{2} E_{2}(t)\leq 0,\quad\quad t\in (0,T),\label{decayineq1}
\end{align}
where $\eta_{0}$, $a_{1}, a_{2}$ and $\alpha$ are constants satisfying 
\[
a_{1}, a_{2}, \alpha>0,\quad\quad  0< \eta_{0}<\min\{\frac{a_{2}}{\alpha}, a_{2}\}.
\]
Then it holds that
\begin{align}
& E_{2}(t)\leq Ca_{1} t^{-\alpha-1},\quad\quad t\in (0,T),\label{E2decay}
\end{align}
where $C>0$ is a constant independent of $T$, $a_{1}$ and $a_{2}$.
\end{lemma}

\begin{proof}
 Let the constant $p$ satisfy $\max\{1,\alpha\}<p<a_{2}/\eta_{0}$. Multiplying \eqref{decayineq1} with $t^{p}$, we get
\begin{align}
\frac{d}{dt}\Big( t^{p}E_{1}(t)+ \eta_{0} t^{p+1} E_{2}(t) \Big)+(a_{2}-p \eta_{0}) t^{p} E_{2}(t)\leq  p t^{p-1} E_{1}(t).\label{decayineq2}
\end{align}
Noticing $a_{2}-p \eta_{0}>0$ and \eqref{decayE1}, we prove, after integrating \eqref{decayineq2} over $[0,t]$, that 
\[
t^{p}E_{1}(t)+ \eta_{0} t^{p+1} E_{2}(t)\leq  p a_1 \int_{0}^{t} \tau^{p-\alpha-1}\, d\tau= \frac{pa_1}{p-\alpha} t^{p-\alpha}.
\]
Hence, \eqref{E2decay} follows.
\end{proof}

\subsection{Optimal decay rates for the heat equation}\label{sec:heat}
In this section, we recover time-decay rates for the heat equation that are consistent with the result obtained in Theorem \ref{ThmDecayGeneral1}.
We consider the one-dimensional heat equation
\begin{equation}\label{heat}
\partial_{t}u-\partial_{x}^2u=0,\quad  (x,t)\in \mathbb{R}\times \mathbb{R}_{+}.
\end{equation}
The asymptotics for the heat equation or incompressible flows in space-weighted spaces has been intensively analyzed, cf. \cite{baejin,kukavica,yoshikomiyakawa} and references therein. Here, we prove a different estimate \eqref{heatdecay2} pertaining to the decay of \eqref{heat} without $L^1$ assumptions. The rates are optimal since our proof relies on the sharp $L^{p}$-$L^{q}$ estimates for the heat flow.

\begin{lemma}
\label{lemmaheat0}
Let $u$ be the solution to the heat equation above with initial datum $u_{0}\in L^2$. Then, for all   $k\geq0$ and $t>0$,
 \begin{align}\label{eq:decayheat1}
\|\partial^{k}_x u(t)\|_{L^2}\leq C_k t^{-\frac{k}{2}}\|u_{0}\|_{L^2}, 
\end{align}
where $C_k>0$ is a constant dependent only on $k$.
\begin{itemize}
\item If we assume that $u_0\in L^p$ with $p\in[1,2)$, then
\begin{align}\label{intro:decayheat}
\|\partial_x^k u(t)\|_{L^2}\leq C_{k,p}t^{-\frac{1}{2}(\frac{1}{p}-\frac{1}{2})-\frac{k}{2}}\|u_{0}\|_{L^{p}},
\end{align}
where $C_{k,p}>0$ is a constant dependent only on $k$ and $p$.
\medbreak
\item  If we further assume that $|x|^{\mu}u_{0}\in L^2 $ with some fixed $\mu>0$, then
\begin{align}
&\|\partial_x^{k}u(t)\|_{L^2}\leq C_{k,\mu} t^{-\frac{\mu}{2}-\frac{k}{2}}\| |x|^{\mu}u_{0}\|_{L^2},\label{heatdecay2}
\end{align}
where $C_{k,\mu}>0$ is a constant dependent only on $k$ and $\mu$.
\end{itemize}
\end{lemma}
In Lemma \ref{lemmaheat0}, the inequalities \eqref{intro:decayheat} and \eqref{eq:decayheat1} are standard, cf \cite{bookgiga}. Below, we give the proof of \eqref{heatdecay2}.

\begin{proof}

The case $\mu=0$ corresponds to the classical decay rate of the heat equation. We divide the proof of the case $\mu>0$ in three cases.

\smallbreak
\noindent $\bullet$  {\textbf{Case 1:} $1\leq \mu \leq 3/2$.} We define the unknown $u_{1}(x,t):=\int^{x}_{-\infty} u(y,t)\,dy$ that satisfies 
\[
\partial_x u_{1}=u,\quad \quad u_{1}(x,0)=u_{1,0}(x):=\int^{x}_{-\infty} u_{0}(y,t)\,dy,
\]
and
\begin{equation}\label{heat1}
\partial_{t}u_{1}-\partial_{x}^2u_{1}=0.
\end{equation}
Note that the Caffarelli-Kohn-Nirenberg inequality \eqref{CKNbest} implies
\begin{align}
&\|u_{1,0}\|_{L^{p_{1}}}\leq \frac{2}{2\mu-1} \| |x|^{\mu}u_{0}\|_{L^2},\nonumber
\end{align}
where $p_{1}:=2/(2\mu-1)$ fulfills $1\leq p_{1}\leq 2$ due to $1\leq \mu \leq 3/2$. Thence it follows from the well-known $L^2$-$L^{p_{1}}$ decay estimates for the $k$-order derivative of the solution to \eqref{heat1} (cf. \cite{bookgiga}) that
\begin{equation}
    \begin{aligned}
    \|\partial_x^{k} u(t)\|_{L^2}&=\|\partial_x^{k+1} u_{1}(t)\|_{L^2}\\
    &\lesssim t^{-\frac{1}{2}(\frac{1}{p_{1}}-\frac{1}{2})-\frac{1}{2}-\frac{k}{2}}\|u_{1,0}\|_{L^{p_{1}}}\lesssim t^{-\frac{\mu}{2}-\frac{k}{2}}\| |x|^{\mu}u_{0}\|_{L^2}.
    \end{aligned}
\end{equation}

\smallbreak
\noindent $\bullet$  {\textbf{Case 2:} $0< \mu< 1$}. The result from Case 1 gives
\begin{align}
&\|\partial_x^{k} u(t)\|_{L^2}\lesssim t^{-\frac{1}{2}-\frac{k}{2}}\| |x| u_{0}\|_{L^2}.\nonumber
\end{align}
Recall that for the case $\mu=0$, one has
\begin{align}
&\|\partial_x^{k} u(t)\|_{L^2}\leq t^{-\frac{k}{2}}\|u_{0}\|_{L^2}.\nonumber
\end{align}
Therefore, employing the Stein-Wassin interpolation theorem (e.g., \cite[Theorem 5.4.1]{berghinter} or \cite{steinwassin}), we have the time-decay estimate \eqref{heatdecay2} for $0<\mu<1$.

\smallbreak
\noindent $\bullet$  {\textbf{Case 3:}  $\mu> 3/2$}. First, we consider the $\mu$ such that there exists a $i\in \mathbb{N}^*$ satisfying $i\leq \mu\leq 1/2+i$. This implies that 
\[
p_{i}:=2/(2(\mu-i+1)-1) \in [1,2].
\]
Then, similarly to Case 1, we note that $u_{i}:=\int^{x}_{-\infty}\int^{x_{i-1}}_{-\infty}\cdot\cdot\cdot \int^{x_{1}}_{-\infty} u(y,t)\,dydx_{1}...d x_{i-1}$ satisfies 
\[
\partial_x^{i}u_{i}=u,\quad\quad u_{i}(x,0)=u_{i,0}(x):=\int^{x}_{-\infty}\int^{x_{i-1}}_{-\infty}\cdot\cdot\cdot \int^{x_{1}}_{-\infty} u_{0}(y,t)\,dydx_{1}...d x_{i-1},
\]
and
\begin{equation}\label{heat2}
\partial_{t}u_{i}-\partial_{x}^2u_{i}=0.
\end{equation}
It follows from the Caffarelli-Kohn-Nirenberg inequality \eqref{CKNbest} that
\begin{align}
&\|u_{i,0}\|_{L^{p_{i}}}\lesssim \| |x|^{\mu-(i-1)}u_{i-1,0}\|_{L^2}\lesssim \| |x|^{\mu-(i-2)}u_{i-2,0}\|_{L^2}\cdot\cdot\cdot \lesssim \| |x|^{\mu}u_{0}\|_{L^2}.\nonumber
\end{align}
This, together with classical $L^2$-$L^{p_{i}}$ decay estimates for the $(k+i)$-order derivative of the solution to \eqref{heat2}, yields
\begin{align}
&\|\partial_x^{k} u(t)\|_{L^2}=\|\partial_x^{k+i}u_{i}(t)\|_{L^2}\lesssim t^{-\frac{1}{2}(\frac{1}{p_{i}}-\frac{1}{2})-\frac{i}{2}-\frac{k}{2}}\|u_{i,0}\|_{L^{p_{i}}}\lesssim t^{-\frac{\mu}{2}-\frac{k}{2}}\| |x|^{\mu}u_{0}\|_{L^2}, \nonumber
\end{align}
For the complementary case where there exists a $i\in \mathbb{N}^*$ such that $i-1/2< \mu<i$,  we get the desired decay estimates by using a similar interpolation argument as in Case 2. The details are omitted. 

\end{proof}
\section*{Acknowledgments}
 T. Crin-Barat and E. Zuazua have been funded by the Alexander von Humboldt-Professorship program and the Transregio 154 Project “Mathematical Modelling, Simulation and Optimization Using the Example of Gas Networks” of the DFG. E. Zuazua has been funded by the ModConFlex Marie Curie Action, HORIZON-MSCA-2021-DN-01, the COST Action MAT-DYN-NET, grants PID2020-112617GB-C22 and TED2021-131390B-I00 of MINECO (Spain), and by the Madrid Goverment -- UAM Agreement for the Excellence of the University Research Staff in the context of the V PRICIT (Regional Programme of Research and Technological Innovation). L.-Y. Shou is supported by the National Natural Science Foundation of China (12301275) and the China Postdoctoral Science Foundation  (2023M741694). Part of this work was done while L.-Y. Shou was visiting the Chair for Dynamics, Control, Machine Learning and Numerics and the Research Center Mathematics of Data (MoD) of the Friedrich-Alexander-Universit\"at Erlangen-N\"urnberg (FAU). L.-Y. Shou is grateful to Prof. E. Zuazua and Dr. T. Crin-Barat for their kind hospitality.



\bibliographystyle{abbrv} 

\bibliography{AIHP_V1}
\vfill 








\end{document}